\documentclass{article}

\usepackage[T1]{fontenc}
\usepackage{ajf}
\usepackage{amsthm}
\usepackage[all]{xy}
\usepackage{accents}
\usepackage[rgb,usenames,dvipsnames]{xcolor}
\usepackage{tikz}
\usepackage{graphicx}
\usepackage{subcaption}
\usepackage{multirow}
\usepackage{url}

\newcommand{\stalkiso}{
\coordinate[label=right:$U$] (base) at (4, 0);
\draw[very thick] (0, 0) -- (base);

\coordinate[label=right:$e$] (e) at (4, 1.9);
\coordinate[label=right:$\Upsilon^U_x e$] (f) at (4, 1.1);
\draw[name path=E, ietlagoon, very thick, tension=1] plot[smooth] coordinates {(0, 1.1) (2, 1.7) (e)};
\draw[name path=F, ietcoast, very thick, tension=0.7] plot[smooth] coordinates {(0, 2.8) (1.3, 2.5) (2.8, 1.5) (f)};

\path[name intersections={of=E and F}];
\coordinate (xStalk) at (intersection-1);
\path let \p1=(xStalk) in coordinate[label={[label distance=2pt]below:$x$}] (xBase) at (\x1, 0);

\draw[shorten <=4pt] (xBase) -- ++(90:3.5) coordinate[label=above:$E_x$];
\fill (xStalk) circle (2pt);
\fill (xBase) circle (2pt);
}

\newcommand{\sectiondev}{
\stalkiso

\coordinate[label=right:$\upsilon^U_{yx} e$] (ue) at (4, 3);
\draw[name path=uE, ietlagoon, very thick, tension=1] plot[smooth] coordinates {(0, 2) (2, 2.7) (ue)};

\path[name intersections={of=F and uE}];
\coordinate (yStalk) at (intersection-1);
\path let \p1=(yStalk) in coordinate[label={[label distance=2pt]below:$y$}] (yBase) at (\x1, 0);

\draw[shorten <=4pt] (yBase) -- ++(90:3.5) coordinate[label=above:$E_y$];
\fill (yStalk) circle (2pt);
\fill (yBase) circle (2pt);
}

\definecolor{pwpink}{RGB}{227, 25, 135}
\definecolor{pworange}{RGB}{245, 153, 30}
\definecolor{pwgold}{RGB}{255, 228, 0}
\definecolor{pwgreen}{RGB}{142, 195, 58}
\definecolor{pwblue}{RGB}{0, 151, 191}
\definecolor{pwviolet}{RGB}{147, 83, 155}
\definecolor{pwbeige}{RGB}{172, 146, 122}

\pgfdeclarepatternformonly{vertical leaves}{\pgfpoint{-1pt}{0pt}}{\pgfqpoint{1pt}{100pt}}{\pgfqpoint{2pt}{100pt}}%
{
  \pgfsetlinewidth{1pt}
  \pgfpathmoveto{\pgfqpoint{0pt}{0pt}}
  \pgfpathlineto{\pgfqpoint{0pt}{100pt}}
  \pgfusepath{stroke}
}

\newcommand{\foliate}[3]{
\fill[preaction={fill, #1!#2}, pattern=vertical leaves, pattern color=#1!#3]
}

\newcommand{\freelabel}[4]{
\node[inner sep=1pt] at #3 [shape=circle, fill=#1!#4] {#2};
}

\newcommand{\sectorlabel}[4]{\freelabel{#1}{#2}{(#3 * 0.65, 0)}{#4}}
\newcommand{\sing}[2]{
\draw[#2, line width=2pt, shift={#1}] (45:3pt) -- (225:3pt);
\draw[#2, line width=2pt, shift={#1}] (-45:3pt) -- (-225:3pt);
}

\newcommand{\gennotchdown}[6]{
\foliate{#1}{#5}{#6} (0, 0) -- (-1, -1) -- (-1, 1) -- cycle;
\foliate{pwbeige}{#5}{#6} (0, 0) -- (-1, 1) -- (1, 1) -- cycle;
\foliate{#3}{#5}{#6} (0, 0) -- (1, 1) -- (1, -1) -- cycle;
\sectorlabel{#1}{#2}{-1}{#5}
\sectorlabel{#3}{#4}{1}{#5}
}
\newcommand{\gennotchup}[6]{
\foliate{#1}{#5}{#6} (0, 0) -- (-1, -1) -- (-1, 1) -- cycle;
\foliate{pwbeige}{#5}{#6} (0, 0) -- (1, -1) -- (-1, -1) -- cycle;
\foliate{#3}{#5}{#6} (0, 0) -- (1, 1) -- (1, -1) -- cycle;
\sectorlabel{#1}{#2}{-1}{#5}
\sectorlabel{#3}{#4}{1}{#5}
}

\newcommand{\notchdown}[4]{
\gennotchdown{#1}{#2}{#3}{#4}{30}{60}
\sing{(0, 0)}{pwbeige}
}
\newcommand{\notchup}[4]{
\gennotchup{#1}{#2}{#3}{#4}{30}{60}
\sing{(0, 0)}{pwbeige}
}

\newcommand{\dupnotchdown}[4]{
\gennotchdown{#1}{#2}{#3}{#4}{10}{20}
\sing{(0, 0)}{pwbeige!45}
}
\newcommand{\dupnotchup}[4]{
\gennotchup{#1}{#2}{#3}{#4}{10}{20}
\sing{(0, 0)}{pwbeige!45}
}

\newcommand{\critleaf}[2]{
\draw[pworange, line width=#2, cap=rect] (#1, {abs(#1)}) -- (#1, -1);
}
\newcommand{\divdleaf}[2]{
\draw[pwbeige, line width={2*(#2)}, cap=rect] (#1, {abs(#1)}) -- (#1, -1);
\draw[pwgold, line width={(2/3)*(#2)}, cap=rect] (#1, {abs(#1)}) -- (#1, -1);
}

\newcommand{\linednotchup}[1]{
\foliate{pwbeige}{10}{20} (0, 0) -- (-1, 1) -- (-1, -1) -- (1, -1) -- (1, 1) -- cycle;
\begin{scope}
\clip (0, 0) -- (-1, 1) -- (-1, -1) -- (1, -1) -- (1, 1) -- cycle;
#1{0}{2pt}
#1{0.251}{1.5pt}
#1{-0.672}{1.5pt}
#1{0.811}{1pt}
#1{-0.423}{1pt}
#1{-0.820}{0.5pt}
#1{0.709}{0.5pt}
\end{scope}
}
\newcommand{\critnotchup}{
\linednotchup{\critleaf}
\sing{(0, 0)}{pworange}
}
\newcommand{\divdnotchup}{
\linednotchup{\divdleaf}
}

\newcommand{\zoomnotchup}[1]{
\fill[pwbeige!10] (0, 0) -- (-1, 1) -- (-1, -1) -- (1, -1) -- (1, 1) -- cycle;
\clip (-1, -1) rectangle (1, 1);
\draw[pwgold, line width=#1, cap=rect] (0, 0) -- (0, -1);
\fill[white] (0, 3*#1/4) circle (#1/2);
\draw[pwbeige, line width=#1, cap=round] (-#1, 3*#1/4) -- (-#1, -1);
\draw[pwbeige, line width=#1, cap=round] (#1, 3*#1/4) -- (#1, -1);
}

\newcommand{\bentlane}[3]{
\begin{scope}[rotate=#1*60]
\draw[pwbeige, line width=#3, join=round, xshift=2*#3] (-30:#2) -- (0, 0) -- (30:#2);
\end{scope}
}
\newcommand{\zoomprongs}[2]{
\fill[pwbeige!10] (0, 0) circle (#1);
\begin{scope}
\clip (0, 0) circle (#1);
\foreach \n in {0,1,2} {\draw[pwgold, line width=#2] ({-90+\n*60}:#1) -- ({90+\n*60}:#1);}
\fill[white] (0, 0) circle (2*#2);
\foreach \n in {0,...,5} {\bentlane{\n}{#1}{#2}}
\end{scope}
}

\newcommand{\veerradius}{1.6cm}
\newcommand{\veerlanewidth}{1.6mm}

\newcommand{\veerleft}{
\zoomprongs{\veerradius}{\veerlanewidth}
\begin{scope}[rotate=-2*60]
\clip (0, 0) circle (\veerradius);
\draw[-stealth, line width=\veerlanewidth, join=round, xshift=2*\veerlanewidth] (30:2) -- (0, 0) -- (-30:1.2);
\end{scope}

\coordinate (start) at (-\veerlanewidth, -\veerradius);
\node (startlabel) at ($(start) + (225:9mm)$) {$\llane{w}$};
\draw[thin, shorten >=2pt] (startlabel) -- (start);
}

\newcommand{\fallmiddle}{
\fill[pwbeige!10] (0, 0) circle (\veerradius);
\begin{scope}
\clip (0, 0) circle (\veerradius);
\draw[line width=\veerlanewidth] (-90:\veerradius) -- (0, 0);
\foreach \n in {1,...,5} {\draw[pwgold, line width=\veerlanewidth] ({-90+\n*60}:\veerradius) -- (0, 0);}
\fill[white] (0, 0) circle (2*\veerlanewidth);
\foreach \n in {0,...,5} {\bentlane{\n}{2}{\veerlanewidth}}
\draw[-stealth, line width=\veerlanewidth] (-90:\veerradius) -- (-90:0.8);
\end{scope}

\coordinate (start) at (0, -\veerradius);
\node (startlabel) at ($(start) + (270:8mm)$) {$\median{w}$};
\draw[thin, shorten >=2pt] (startlabel) -- (start);
}

\newcommand{\veerright}{
\zoomprongs{\veerradius}{\veerlanewidth}
\begin{scope}[rotate=-1*60]
\clip (0, 0) circle (\veerradius);
\draw[-stealth, line width=\veerlanewidth, join=round, xshift=2*\veerlanewidth] (-30:2) -- (0, 0) -- (30:1.2);
\end{scope}

\coordinate (start) at (\veerlanewidth, -\veerradius);
\node (startlabel) at ($(start) + (315:9mm)$) {$\rlane{w}$};
\draw[thin, shorten >=2pt] (startlabel) -- (start);
}

\newcommand{\earringlabel}[4]{\freelabel{#1}{#2}{(#3 * \bighair * 0.6, 0)}{#4}}

\newcommand{\bighair}{1.4}
\newcommand{\radface}{0.29}

\newcommand{\radfade}{0.8}
\newcommand{\hairfade}{
\fill[white, path fading=south] (-\bighair, \bighair) rectangle (\bighair, {\radfade*\bighair});
\fill[white, path fading=west] (\bighair, \bighair) rectangle ({\radfade*\bighair}, -\bighair);
\fill[white, path fading=north] (\bighair, -\bighair) rectangle (-\bighair, -{\radfade*\bighair});
\fill[white, path fading=east] (-\bighair, -\bighair) rectangle (-{\radfade*\bighair}, \bighair);
\draw[white] (-\bighair, -\bighair) rectangle (\bighair, \bighair);
}

\newcommand{\hairdown}[4]{
\foliate{#1}{30}{60} ($\radface*(-\bighair, -\bighair)$) -- (-\bighair, -\bighair) -- (-\bighair, \bighair) -- ($\radface*(-\bighair, \bighair)$) -- cycle;
\foliate{pwbeige}{30}{60} ($\radface*(-\bighair, \bighair)$) -- (-\bighair, \bighair) -- (\bighair, \bighair) -- ($\radface*(\bighair, \bighair)$) -- cycle;
\foliate{#3}{30}{60} ($\radface*(\bighair, \bighair)$) -- (\bighair, \bighair) -- (\bighair, -\bighair) -- ($\radface*(\bighair, -\bighair)$) -- cycle;
\earringlabel{#1}{#2}{-1}{30}
\earringlabel{#3}{#4}{1}{30}
\hairfade
}
\newcommand{\hairup}[4]{
\foliate{#1}{30}{60} ($\radface*(-\bighair, -\bighair)$) -- (-\bighair, -\bighair) -- (-\bighair, \bighair) -- ($\radface*(-\bighair, \bighair)$) -- cycle;
\foliate{pwbeige}{30}{60} ($\radface*(\bighair, -\bighair)$) -- (\bighair, -\bighair) -- (-\bighair, -\bighair) -- ($\radface*(-\bighair, -\bighair)$) -- cycle;
\foliate{#3}{30}{60} ($\radface*(\bighair, \bighair)$) -- (\bighair, \bighair) -- (\bighair, -\bighair) -- ($\radface*(\bighair, -\bighair)$) -- cycle;
\earringlabel{#1}{#2}{-1}{30}
\earringlabel{#3}{#4}{1}{30}
\hairfade
}

\definecolor{ietocean}{RGB}{0, 30, 140}
\definecolor{ietcoast}{RGB}{0, 150, 173}
\definecolor{ietlagoon}{RGB}{0, 216, 180}
\definecolor{ietpapaya}{RGB}{255, 121, 0}
\definecolor{ietgold}{RGB}{255, 228, 0}

\pgfmathsetmacro{\lea}{sin(18)}
\pgfmathsetmacro{\leb}{sin(36) - sin(18)}
\pgfmathsetmacro{\lec}{sin(54) - sin(36)}
\pgfmathsetmacro{\led}{sin(72) - sin(54)}
\pgfmathsetmacro{\lee}{sin(90) - sin(72)}

\newcommand{\ininterval}[3]{
\fill[#3] (#1 + 0.005, -0.25) rectangle (#2 - 0.005, -0.2);
}
\newcommand{\outinterval}[3]{
\fill[#3] (#1 + 0.005, 0.2) rectangle (#2 - 0.005, 0.25);
}
\newcommand{\pmtarrow}[4]{
\draw (#1 + 0.005, -0.18) edge[out=#3,in=#4+180,->] (#2 - 0.005, 0.18);
}

\newcommand{\exampleiet}{
\ininterval{0}{\lea}{ietocean}
\ininterval{\lea}{\lea + \leb}{ietcoast}
\ininterval{\lea + \leb}{\lea + \leb + \lec}{ietlagoon}
\ininterval{\lea + \leb + \lec}{\lea + \leb + \lec + \led}{ietpapaya}
\ininterval{\lea + \leb + \lec + \led}{\lea + \leb + \lec + \led + \lee}{ietgold}

\outinterval{0}{\lec}{ietlagoon}
\outinterval{\lec}{\lec + \lee}{ietgold}
\outinterval{\lec + \lee}{\lec + \lee + \leb}{ietcoast}
\outinterval{\lec + \lee + \leb}{\lec + \lee + \leb + \led}{ietpapaya}
\outinterval{\lec + \lee + \leb + \led}{\lec + \lee + \leb + \led + \lea}{ietocean}

\pmtarrow{\lea/2}{\lec + \lee + \leb + \led + \lea/2}{90}{60}
\pmtarrow{\lea + \leb/2}{\lec + \lee + \leb/2}{90}{90}
\pmtarrow{\lea + \leb + \lec/2}{\lec/2}{120}{90}
\pmtarrow{\lea + \leb + \lec + \led/2}{\lec + \lee + \leb + \led/2}{90}{90}
\pmtarrow{\lea + \leb + \lec + \led + \lee/2}{\lec + \lee/2}{90}{120}
}

\definecolor{sky}{Hsb}{195,1.00,1.00}
\definecolor{rock}{Hsb}{22,.36,.38}
\definecolor{lshed}{RGB}{0, 150, 173}
\definecolor{rshed}{RGB}{0, 216, 180}

\pgfmathsetmacro{\mountainheight}{sqrt(3)/6}
\pgfmathsetmacro{\earthdepth}{-1/12}

\newcounter{mountainlevel}

\newcommand{\mountains}{
\fill[rock] (1/3+\earthdepth/2,6*\mountainheight*\earthdepth/2)
	-- (1/2,\mountainheight)
	-- (2/3-\earthdepth/2,6*\mountainheight*\earthdepth/2);
\fill[white] (1/2,\mountainheight)
	-- (4/9,2/3*\mountainheight)
	-- (.4649,.2059)
	-- (.4738,.1978)
	-- (.4890,.2057)
	-- (.5006,.1955)
	-- (.5129,.2139)
	-- (.5294,.1963)
	-- (.5360,.2036)
	-- (.5549,.1925)
	-- (5/9,2/3*\mountainheight);
\fill[black!7] (1/2,.2842) -- (.4501,.1989) -- (.4646,.2088);
\ifnum \value{mountainlevel} > 1
	\addtocounter{mountainlevel}{-1}
	\begin{scope}[scale=1/3]
		\mountains
		\begin{scope}[shift={(2,0)}]\mountains\end{scope}
	\end{scope}
	\addtocounter{mountainlevel}{1}
\fi
}

\newcommand{\landscape}{
\fill[sky] (0,0) rectangle (1,1/3);
\foreach \level / \skycolor in {1 / sky!80, 2 / sky!65, 3 / sky!50, 4 / sky!40, 5 / sky!30} {
	\pgfmathsetmacro{\altitude}{\mountainheight/pow(3,\level-1)}
	\fill[\skycolor] (0,0) rectangle (1,\altitude);
}
\fill[rock] (0,\earthdepth) rectangle (1,0);
\setcounter{mountainlevel}{5}
\mountains
}

\newcommand{\segment}{
\begin{scope}[line width=2pt]
	\draw[ietocean, shorten <=-0.5pt, shorten >=-1pt] (0pt, 0cm) -- ++(0:16pt);
	\begin{scope}[shorten <=-1pt, shorten >=-1pt]
		\draw[ietcoast] (18pt, 0cm) -- ++(0:12pt);
		\draw[ietlagoon] (32pt, 0cm) -- ++(0:12pt);
	\end{scope}
	\draw[ietpapaya, shorten <=-1pt, shorten >=-0.5pt] (46pt, 0cm) -- ++(0:10pt);
	\begin{scope}[shorten <=-0.5pt, shorten >=-0.5pt]
		\draw[ietocean] (l1) -- (r1);
		\draw[ietcoast] (l2) -- (r2);
		\draw[ietlagoon] (l3) -- (r3);
		\draw[ietpapaya] (l4) -- (r4);
	\end{scope}
\end{scope}
\node[bracing, minimum width=56pt] (zbrace) at (28pt, 0cm) {};
\node[below=8pt] at (zbrace.south) {$Z$};
}

\newcommand{\pointlabels}{
\begin{scope}[ietocean]
	\node (p0) at (225:8mm) {$p$};
	\node (pt) at ($(l1) + (135:9mm)$) {$\phi^t p$};
	\draw[thin, shorten >=2pt] (p0) -- (0pt, 0cm);
	\draw[thin, shorten >=2pt] (pt) -- (l1);
\end{scope}
}

\newcommand{\exits}{
\begin{tikzpicture}[line width=1pt]
\fill[pwbeige!10]
	(-10pt, 0cm) -- ++(0cm, 1cm) arc (0:15:6cm-10pt) arc (15:35:3cm-10pt) arc (35:65:2cm-10pt) -- ++(65:8pt)
	arc (65:35:2cm-2pt) arc (35:15:3cm-2pt) arc (15:0:6cm-2pt) -- (-2pt, 0cm) -- cycle;
\foreach \x in {-8, -6, ..., -4} {
\draw[pwbeige!20] (\x pt, 0cm) -- ++(90:1cm) arc (0:15:6cm+\x pt) arc (15:35:3cm+\x pt) arc (35:65:2cm+\x pt);
}
\draw[pwbeige] (-2pt, 0cm) -- ++(90:1cm) arc (0:15:6cm-2pt) arc (15:35:3cm-2pt) arc (35:65:2cm-2pt);
\draw[pwgold] (-1pt, 0cm) -- ++(90:1cm) arc (0:15:6cm-1pt) arc (15:35:3cm-1pt);
\fill[pwbeige!10]
	(0pt, 0cm) -- ++(0cm, 1cm) arc (0:15:6cm+0pt) arc (15:35:3cm+0pt) arc (215:180:2cm+0pt) coordinate (l1) -- ++(0:16pt) coordinate (r1)
	arc (180:215:2cm-16pt) arc (35:15:3cm+16pt) arc (15:0:6cm+16pt) -- (16pt, 0cm) -- cycle;
\draw[ietocean] (0pt, 0cm) -- ++(0cm, 1cm) arc (0:15:6cm+0pt) arc (15:35:3cm+0pt) arc (215:180:2cm+0pt);
\foreach \x in {2, 4, ..., 14} {
\draw[pwbeige!20] (\x pt, 0cm) -- ++(90:1cm) arc (0:15:6cm+\x pt) arc (15:35:3cm+\x pt) arc (215:180:2cm-\x pt);
}
\draw[pwbeige] (16pt, 0cm) -- ++(90:1cm) arc (0:15:6cm+16pt) arc (15:35:3cm+16pt) arc (215:180:2cm-16pt);
\draw[pwgold] (17 pt, 0cm) -- ++(90:1cm) arc (0:15:6cm+17pt);
\fill[pwbeige!10]
	(18pt, 0cm) -- ++(0cm, 1cm) arc (0:15:6cm+18pt) arc (195:170:4cm-18pt) coordinate (l2) -- ++(-10:12pt) coordinate (r2)
	arc (170:195:4cm-30pt) arc (15:0:6cm+30pt) -- (30pt, 0cm) -- cycle;
\draw[pwbeige] (18pt, 0cm) -- ++(90:1cm) arc (0:15:6cm+18pt) arc (195:170:4cm-18pt);
\foreach \x in {20, 22, ..., 28} {
\draw[pwbeige!20] (\x pt, 0cm) -- ++(90:1cm) arc (0:15:6cm+\x pt) arc (195:170:4cm-\x pt);
}
\draw[pwbeige] (30pt, 0cm) -- ++(90:1cm) arc (0:15:6cm+30pt) arc (195:170:4cm-30pt);
\draw[pwgold] (31pt, 0) -- (31pt, 1cm);
\fill[pwbeige!10]
	(32pt, 0cm) -- ++(90:1cm) arc (180:165:6cm-32pt) arc (-15:5:3cm+32pt) coordinate (l3) -- ++(5:12pt) coordinate (r3)
	arc (5:-15:3cm+44pt) arc (165:180:6cm-44pt) -- (44pt, 0cm) -- cycle;
\draw[pwbeige] (32pt, 0cm) -- ++(90:1cm) arc (180:165:6cm-32pt) arc (-15:5:3cm+32pt);
\foreach \x in {34, 36, ..., 42} {
\draw[pwbeige!20] (\x pt, 0cm) -- ++(90:1cm) arc (180:165:6cm-\x pt) arc (-15:5:3cm+\x pt);
}
\draw[pwbeige] (44pt, 0cm) -- ++(90:1cm) arc (180:165:6cm-44pt) arc (-15:5:3cm+44pt);
\draw[pwgold] (45pt, 0cm) -- ++(90:1cm) arc (180:165:6cm-45pt);
\fill[pwbeige!10]
	(46pt, 0cm) -- ++(90:1cm) arc (180:165:6cm-46pt) arc (165:135:4cm-46pt) coordinate (l4) -- ++(-45:10pt) coordinate (r4) -- ++(-45:10pt)
	arc (135:165:4cm-66pt) arc (165:180:6cm-66pt) -- (66pt, 0cm) -- cycle;
\draw[pwbeige] (46pt, 0cm) -- ++(90:1cm) arc (180:165:6cm-46pt) arc (165:135:4cm-46pt);
\foreach \x in {48, 50, ..., 54, 60, 62, 64} {
\draw[pwbeige!20] (\x pt, 0cm) -- ++(90:1cm) arc (180:165:6cm-\x pt) arc (165:135:4cm-\x pt);
}
\draw[pwgold] (57pt, 0cm) -- ++(90:1cm) arc (180:165:6cm-57pt) arc (165:135:4cm-57pt);
\foreach \x in {56, 58} {
\draw[pwbeige] (\x pt, 0cm) -- ++(90:1cm) arc (180:165:6cm-\x pt) arc (165:135:4cm-\x pt);
}
\segment
\pointlabels
\end{tikzpicture}
}

\newcommand{\parallelogramcoords}{
\coordinate (slope) at (-0.25, 0);
\coordinate (leaf) at (0, 3);
\coordinate (a) at (-3, 0);
\coordinate (b) at ($(leaf) + (slope)$);
}

\newcommand{\parallelogramsings}{
\sing{(0, 0)}{pwbeige}
\sing{(a)}{pwbeige}
\sing{($(a)+(b)$)}{pwbeige}
\sing{(b)}{pwbeige}
}

\newcommand{\spiralparallelogram}{
\parallelogramcoords
\foliate{pwbeige}{20}{40} (0, 0) -- ++(a) -- ++(b) -- ++($-1*(a)$) -- cycle;
\node[anchor=south] at ($(leaf) + 1/2*(slope)$) {$m$};
\draw (b) -- (leaf) -- ++(0, -0.1);
\draw[dotted] (leaf) + (0, -0.1) -- (0, 0);

\foreach \j in {0, ..., 5} { \draw[pwpink, line width=1.5pt-\j*0.21pt] ($(slope) + \j*(slope)$) -- ++(leaf); }
\foreach \j in {0, ..., 5} { \draw[pworange, line width=1.5pt-\j*0.21pt] ($(a) - \j*(slope)$) -- ++(leaf); }

\parallelogramsings

\draw[-stealth,shorten <=1mm] ($(b)+1/2*(a)$) -- ++($1/5*(b)$) node[anchor=south] {$B$};
\draw[-stealth,shorten <=1mm] ($(a)+1/2*(b)$) -- ++($1/5*(a)$) node[anchor=east] {$A$};
}

\newcommand{\returnparallelogram}{
\parallelogramcoords

\foliate{ietcoast}{75}{100} ($1/2*(b)$) -- ++($1/2*(b)$) -- ++($-1/2*(leaf)$) -- cycle;
\foliate{ietcoast}{75}{100} ($(a) + 1/2*(b)$) -- ++($1/2*(leaf)$) -- ++($1/2*(slope)$) -- cycle;
\foliate{ietcoast}{75}{100} (a) rectangle ++($1/2*(leaf) - 1/2*(slope)$);

\foliate{ietgold}{75}{100} ($1/2*(leaf) + (slope)$) rectangle ++($1/2*(b)$);
\foliate{ietgold}{75}{100} (0, 0) -- ++($1/2*(b)$) -- ++($-1/2*(leaf)$) -- cycle;
\foliate{ietgold}{75}{100} (a) -- ++($1/2*(leaf)$) -- ++($1/2*(slope)$) -- cycle;

\foliate{pwbeige}{20}{40} ($(b) + 1/2*(slope)$) rectangle ++($(a) - (slope) - 1/2*(leaf)$);
\foliate{pwbeige}{20}{40} ($1/2*(b)$) rectangle ++($(a) - (slope) - 1/2*(leaf)$);

\parallelogramsings

\draw[line width=1.5] ($1/2*(b)$) -- ++(a);
\node[anchor=west] at ($1/2*(b)$) {$Z$};
}

\theoremstyle{plain}
\newtheorem{prop}{Proposition}[subsection]
\newtheorem{lemma}[prop]{Lemma}
\newtheorem{thm}[prop]{Theorem}
\newtheorem{cor}[prop]{Corollary}
\newtheorem{secprop}{Proposition}[section]

\usetikzlibrary{cd, intersections, patterns, matrix, calc, fadings, decorations.pathreplacing, arrows, shapes.geometric}
\usepgflibrary{fpu}
\tikzset{bracing/.style={rectangle,inner sep=0pt,below delimiter=\}}}

\newcommand{\thru}{\unskip\,--\,\ignorespaces}
\newcommand{\secs}{\S\S\,\ignorespaces}

\newcommand{\blankbox}{
	\begin{tikzpicture}
	\node[fill=black!15, inner sep=0] {\phantom{$h$}};
	\end{tikzpicture}
}
\newcommand{\chars}[1]{\mathcal{X}(#1)}
\newcommand{\nechars}[1]{\mathcal{X}_\text{ne}(#1)}
\newcommand{\twchars}[1]{\tilde{\mathcal{X}}(#1)}

\newcommand{\netwchars}[1]{\tilde{\mathcal{X}}_\text{ne}(#1)}
\newcommand{\circgp}{\mathbb{T}}
\newcommand{\llane}[1]{{\accentset{\leftarrow}{#1}}}
\newcommand{\rlane}[1]{{\accentset{\rightarrow}{#1}}}
\newcommand{\median}[1]{{\hat{#1}}}
\newcommand{\divd}[1]{{\accentset{\leftrightarrow}{#1}}}
\newcommand{\frakd}[1]{{\accentset{\text{\guilsinglleft}\,\text{\guilsinglright}}{#1}}}
\DeclareMathOperator{\trim}{trim}
\DeclareMathOperator{\seal}{seal}
\newcommand{\nom}[1]{{\accentset{\text{-\,-}}{#1}}}
\newcommand{\gr}[1]{\on{grade} #1}

\newcommand{\ip}[2]{\langle #1, #2 \rangle}
\newcommand{\Hyp}{\mathbb{H}}

\newenvironment{tocsection}[1]{
{\item[] \textbf{\textit{#1}}}
\begin{samepage}
\begin{itemize}
}{
\end{itemize}
\end{samepage}
}
\newcommand{\tocitem}[1]{\item[] \textbf{#1} \;\;}

\pdfpageattr{/Group <</S /Transparency /I true /CS /DeviceRGB>>}

\usepackage{hyperref}
\hypersetup{
	colorlinks,
	linkcolor={ietcoast},
	citecolor={ietcoast},
	urlcolor={ietcoast}
}

\title{\Large A dynamical perspective on shear-bend coordinates}
\author{Aaron Fenyes}
\date{}

\begin{document}
\maketitle
\begin{abstract}
Twisted $\on{SL}_2 \C$ local systems on surfaces of finite type appear often in geometry and physics. Most of them arise geometrically as local systems of charts for pleated hyperbolic structures. Bonahon and Thurston's {\em shear-bend coordinates} parameterize these local systems of charts.

On a surface with punctures, Gaiotto, Hollands, Moore and Neitzke's {\em abelianization} process computes the shear-bend coordinates of a twisted $\on{SL}_2 \C$ local system without reference to its hyperbolic geometry. Using tools from dynamics, we'll generalize abelianization to compact surfaces, leading to a dynamical recipe for the shear-bend parameterization. This recipe lends itself well to numerical approximation, and it may clarify the changes of coordinates that relate different shear parameterizations.
\end{abstract}
\section{Introduction}\label{intro}
\subsection{Context}\label{context}
\subsubsection{Why study twisted $\on{SL}_2 \C$ local systems?}
Bonahon and Thurston's {\em shear-bend coordinates} parameterize twisted $\on{SL}_2 \C$ local systems on a surface of finite type. Before we dive into the dynamics of these local systems, let's see why they're worth studying, and what makes the shear-bend parameterization good for studying them.

Let $C$ be a connected smooth surface of finite type. A {\em twisted $\on{SL}_2 \C$ local system} on $C$ is an $\on{SL}_2 \C$ local system on the unit tangent bundle\footnote{In spite of its name, the unit tangent bundle can be defined without reference to any metric. See \cite[\S 3.3.1]{warping-geom} for an explicit definition.} $UC$, with holonomy $-1$ around each fiber~\cite[\S 22]{curves-on-surfaces}\cite[\secs 2.3 \thru 2.4]{higher-teich}\cite[\S 10.1]{spec-nets}.\footnote{\label{gen-twisted} Some authors identify twisted local systems with representations of central extensions of $\pi_1 C$~\cite{geom-char-var}\cite[\secs 4.4 \thru 4.5]{alg-aspects}. Our definition is a special case. When $C$ is compact, you can see this by presenting $\pi_1 UC$ as
\[ \langle a_1, b_1, \ldots, a_g, b_g, u \mid [a_1, b_1] \cdots [a_g, b_g] = u^{2(1-g)}, u \text{ central} \rangle, \]
where $a_1, b_1, \ldots, a_g, b_g$ project to a standard generating set for $\pi_1 C$ and $u$ loops once counterclockwise around the fiber over the base point of $C$~\cite[\S 5.6.3]{mcg-primer}.} You can push it forward to an ordinary local system on $C$ if you forget about the difference between $-1$ and $1$. This gives a natural map from twisted $\on{SL}_2 \C$ local systems to ordinary $\on{PSL}_2 \C$ local systems.

The twisted $\on{SL}_2 \C$ local systems on $C$ and the natural isomorphisms between them form a groupoid, which is essentially a complex Lie groupoid. Its Morita equivalence class is a complex differentiable stack, called the {\em twisted $\on{SL}_2 \C$ character stack} of $C$.
Its ordinary points are the isomorphism classes of irreducible local systems. They form a complex manifold, which I'll call the {\em twisted $\on{SL}_2 \C$ character variety} of $C$.\footnote{This terminology isn't quite standard. The character variety is typically defined as a singular algebraic variety which includes both the ordinary points and the stacky points of the character stack. The complex manifold I'm calling the character variety is the smooth part of the usual character variety~\cite[\S 4.5]{alg-aspects}.}

In geometry, the fundamental example of a twisted $\on{SL}_2 \C$ local system comes from a complete hyperbolic structure on $C$. The local isometries from $C$ to the hyperbolic plane form a $\on{PSL}_2 \R$ local system, and knowing that local system up to isomorphism is the same as knowing the hyperbolic structure up to isotopy \cite[\S 2]{warping-geom}. A $\on{PSL}_2 \R$ local system arising from a hyperbolic structure lifts to a twisted $\on{SL}_2 \R$ local system in a canonical way~\cite[\S 23]{curves-on-surfaces}\cite[\S 3.3.1]{warping-geom}. Thus, the space of complete hyperbolic structures on $C$ embeds naturally in the twisted $\on{SL}_2 \R$ character variety of $C$.

A more general example of a twisted $\on{SL}_2 \C$ local system comes from a pleated hyperbolic structure on $C$. This is a complete hyperbolic structure with some extra data: an atlas of ``pleated path-isometries'' into hyperbolic $3$-space~\cite{shearing-bending}. The pleated path-isometries form a $\on{PSL}_2 \C$ local system, which lifts canonically to a twisted $\on{SL}_2 \C$ local system. This maps the space of pleated hyperbolic structures onto a dense open subset of the twisted $\on{SL}_2 \C$ character variety. When $C$ is compact, the image comprises the {\em non-elementary} local systems: the ones whose structure groups don't reduce to the upper-triangular or unitary subgroups of $\on{SL}_2 \C$.\footnote{Here's a proof sketch for connoisseurs. Let $\mathcal{F}(C)$ be the space of complete hyperbolic structures up to isotopy. A pleated hyperbolic structure is determined by its underlying hyperbolic structure and its pleating lamination, so $\mathcal{F}(C) \times \mathcal{ML}(C)$ parameterizes pleated hyperbolic structures up to isotopy. Let $\mathcal{P}(C)$ be the space of complex projective structures up to isotopy. Let $\chars{C}$ and $\twchars{C}$ be the ordinary $\on{PSL}_2 \C$ and twisted $\on{SL}_2 \C$ character varieties. The non-elementary local systems form dense open submanifolds $\nechars{C} \subset \chars{C}$ and $\netwchars{C} \subset \twchars{C}$.

You can turn a pleated hyperbolic surface into a complex projective surface by grafting along its pleating lamination. Kamishima, Tan, and Thurston showed this gives a homeomorphism $\mathcal{F}(C) \times \mathcal{ML}(C) \to \mathcal{P}(C)$~\cite{cp-dfm-spaces}. Heijal, Earle, and Hubbard proved in turn that the holonomy map $\mathcal{P}(C) \to \mathcal{X}(C)$ is a local homeomorphism~\cite{heijals-thm}. Its image is $\nechars{C}$, as shown by Gallo, Kapovich, and Marden~\cite{cp-monodromy}. Altogether, we have a surjective local homeomorphism $\mathcal{F}(C) \times \mathcal{ML}(C) \to \nechars{C}$ which sends each pleated hyperbolic structure to its local system of pleated path-isometries.

As I mentioned earlier, a $\on{PSL}_2 \C$ local system arising in this way lifts canonically to a twisted $\on{SL}_2 \C$ local system. The lifted holonomy map $\mathcal{F}(C) \times \mathcal{ML}(C) \to \netwchars{C}$ is continuous, and thus still a local homeomorphism. To see it's still surjective, note that the projection $\netwchars{C} \to \nechars{C}$ is a covering map, and $\netwchars{C}$ and $\nechars{C}$ are both connected.}

In physics, twisted $\on{SL}_2 \C$ local systems parameterize an interesting family of solutions to the Yang-Mills equation with gauge group $\on{SO}_3$. For this example, assume $C$ is compact, and give it a flat metric with finitely many conical singularities. Let's study Yang-Mills fields on $C \times \mathbb{T}^2$ which extend smoothly over the singularities of $C$, and are invariant under translations of the flat torus $\mathbb{T}^2$. We can learn about the quantum behavior of these fields, in the classical limit, by looking at the configurations that minimize the Yang-Mills action~\cite[\S 4.8]{qft-invitation}. The minimizers turn out to be the fields whose curvature forms are either self-dual or anti-self-dual \cite{arbi-chern}.\footnote{The argument in \cite{arbi-chern} takes a little extra care in our setting, because the integral in the Yang-Mills action has to avoid the singularities. The trick is to observe that we can still use the integral formula for the second Chern number, because we're looking at fields which extend smoothly over the singularities.} The self-dual fields, through a lovely construction by Donaldson, correspond one-to-one with the irreducible $\on{PSL}_2 \C$ local systems on $C$~\cite[\S 11]{geom-top-moduli}. Composing this correspondence with the natural map from twisted $\on{SL}_2 \C$ local systems to ordinary $\on{PSL}_2 \C$ local systems gives the promised parameterization.\footnote{This parameterization only describes the self-dual fields that correspond to $\on{PSL}_2 \C$ local systems that lift to twisted $\on{SL}_2 \C$ local systems. It can be extended to all self-dual fields using the generalized twisted local systems mentioned in footnote~\ref{gen-twisted}~\cite[\S 9]{self-duality-eqns}.} The twisted $\on{SL}_2 \R$ local systems arising from hyperbolic structures on $C$ map to a particularly simple family of self-dual fields~\cite[\S 11]{geom-top-moduli}.

The space of complete hyperbolic structures on $C$ is naturally a real analytic symplectic manifold.\footnote{We'll set its symplectic form to $-2$ times the Weil-Petersson form, using Goldman's normalization for the latter~\cite[\S 2.6.3]{geom-of-chars}.} It sits inside the $\on{SL}_2 \C$ character variety as a maximal totally real analytic submanifold~\cite[\S 5.1]{sympl-dfm-space}. Its symplectic structure therefore extends uniquely to a complex symplectic structure on the twisted character variety~\cite[\S 6.2]{sympl-dfm-space}, first described by Goldman~\cite{geom-of-chars}.

For each singular flat metric on $C$, the space of self-dual Yang-Mills fields on $C \times \mathbb{T}^2$ is naturally a complex symplectic manifold. The parameterization of self-dual fields by twisted $\on{SL}_2 \C$ local systems thus defines another complex symplectic structure on the twisted character variety, first described by Hitchin~\cite[equation~6.5]{self-duality-eqns}. The space of complete hyperbolic structures is a complex Lagrangian submanifold with respect to this complex symplectic structure~\cite[p.~95 and \S 11]{self-duality-eqns}.

At first glance, the geometric and physical motivations for studying $\on{SL}_2 \C$ local systems might seem to have little in common. The complex symplectic structures they lead to, however, are related in a surprising way. Goldman's and Hitchin's complex symplectic forms turn out to have the same imaginary part. They can be written as
\begin{align*}
\Omega_\text{Gold} & = -\omega_1 + i \omega_3 \\
\Omega_\text{Hit} & = \hphantom{-}\omega_2 + i \omega_3
\end{align*}
in terms of real symplectic forms $\omega_1, \omega_2, \omega_3$, which fit together into a hyperk\"{a}hler structure on the twisted $\on{SL}_2 \C$ character variety.\footnote{\label{hyperkahler-note}In Section~6 of \cite{self-duality-eqns}, Hitchin describes this hyperk\"{a}hler structure in terms of an explicitly defined metric $g$, the complex structure $I$ of the space of self-dual Yang-Mills fields, and another complex structure $J$, which is revealed on p.~112 of \cite{self-duality-eqns} to be the complex structure of the $\on{SL}_2 \C$ character variety. You can get explicit formulas for $\omega_1, \omega_2, \omega_3$ using the formula for $J$ found on p.~109. You can then verify that $\Omega_\text{Gold} = -\omega_1 + i \omega_3$ by direct computation, using the formula for $\Omega_\text{Gold}$ from p.~125. To compare this expression with Proposition~11.17 of \cite{self-duality-eqns}, note that the $\omega_1$ in the proposition should be $-\omega_1$, and that Hitchin's normalization of the Weil-Petersson form is $-1/2$ times Goldman's. (Many thanks to Brice Loustau for explaining all of this.)} This extra structure isn't apparent from the geometric or the physical perspective alone. You need both eyes to see it.
\subsubsection{Charting hyperbolic structures with shear-bend coordinates}\label{shear-intro}
The space of complete hyperbolic structures on a finite-type surface comes with an interesting family of parameterizations, called {\em shear parameterizations}. On a punctured surface $C'$, you get a shear parameterization by choosing an ideal triangulation and a winding direction around each puncture. When you put a hyperbolic structure on $C'$, the edges of the triangulation ``snap tight'' to geodesics, winding in the chosen direction around each funnel-shaped end. They cut $C'$ into hyperbolic ideal triangles and funnels.
\begin{center}
\begin{tikzpicture}
\matrix[row sep=0.2cm, column sep=1cm]{
\node {\includegraphics[width=3cm]{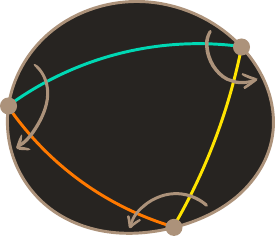}}; &
\node {\includegraphics[width=4cm]{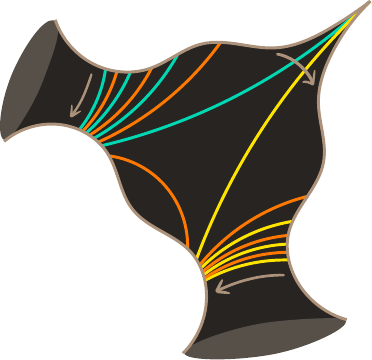}}; \\
\node[align = center, text width=4cm] {\small Ideal triangulation and winding directions}; &
\node[align = center, text width=4cm] {\small Snapping tight}; \\
};
\end{tikzpicture}
\end{center}
A hyperbolic ideal triangle comes with a horocyclic ``contact triangle'' at its center~\cite[\S 3.1]{length-convexity}. Along each edge $e$ of the triangulation, measure the displacement $x_e \in \R$ between the corners of the neighboring contact triangles. By convention, $x_e$ is positive when someone standing at the corner of one contact triangle can look across $e$ and see the other contact triangle to the left.
\begin{center}
\begin{tikzpicture}
\matrix[row sep=0.2cm]{
\node {\includegraphics[width=52mm]{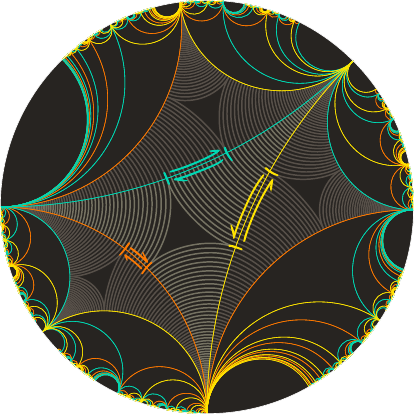}}; \\
\node[align = center, text width=8cm]{\small Displacements between contact triangles, on the universal cover of $C'$}; \\
};
\end{tikzpicture}
\end{center}
The numbers $X_e = \exp x_e$ are the {\em shear coordinates} of the hyperbolic structure on $C'$. They tell you how to piece $C'$ together out of ideal triangles and funnels, determining its hyperbolic structure up to isotopy. In this way, $\R_{> 0}$-valued functions on the edge set of the triangulation parameterize the complete hyperbolic structures on $C'$. That's the shear parameterization.

The shear coordinates can't take just any values. Their product is always one or more around each counterclockwise puncture, and one or less around each clockwise puncture, with equality when the puncture is a cusp. This constraint cuts the parameter space down to a closed polyhedral cone, with the cusped hyperbolic structures on the boundary.

On a compact surface $C$, you get a shear parameterization by choosing a {\em maximal measured geodesic lamination}, which is like an ideal triangulation with no vertices and a Cantor set of edges. When you pick a hyperbolic structure on $C$, the leaves of the lamination snap tight to geodesics, cutting $C$ into finitely many hyperbolic ideal triangles.
\begin{center}
\begin{tikzpicture}
\matrix[column sep=1.5cm]{
\node {\includegraphics{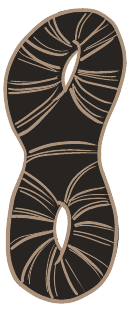}}; &
\node {\includegraphics[width=52mm]{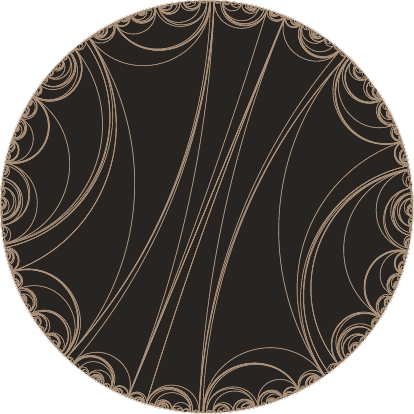}}; \\
\node[align = center, text width=4cm] {\small Maximal measured geodesic lamination}; &
\node[align = center, text width=4cm] {\small Snapping tight, on the universal cover of $C$}; \\
};
\end{tikzpicture}
\end{center}
Given a swath of leaves $H$ separating a pair of ideal triangles, you can measure a well-defined displacement $x_H \in \R$ between the corners of the contact triangles~\cite[\S 2]{shearing-bending}, as illustrated in \cite{length-convexity} at the end of Section~3.2. The {\em shear parameters} $X_H = \exp x_H$ fit together into an $\R_{> 0}$-valued {\em transverse cocycle} on the lamination~\cite[\S 1]{shearing-bending}. Surprisingly, the transverse cocycles form a finite-dimensional vector space. An open convex cone in this space parameterizes the hyperbolic structures on $C$~\cite[\S 6]{shearing-bending}.

Choosing an ideal triangulation of $C'$ picks out a class of pleated hyperbolic structures: the ones pleated along the edges of the triangulation. You can parameterize these pleated hyperbolic structures using $\C^\times$-valued {\em shear-bend coordinates}, which record both the displacement and the pleating angle between neighboring triangles. Similarly, on $C$, the shear parameterization picked out by a maximal measured geodesic lamination extends to a {\em shear-bend parameterization} of the hyperbolic structures pleated along that lamination~\cite[\secs 7 \thru 10]{shearing-bending}. Its values are $\C^\times$-valued transverse cocycles.

Each shear-bend parameterization is holomorphic with respect to the natural complex structure of the twisted $\on{SL}_2 \C$ character variety, and each one provides a simple description of Goldman's complex symplectic structure. On $C'$, the Poisson brackets of the shear-bend coordinates can be written explicitly as $\{X_e, X_{e'}\} = \ip{e}{e'} X_e X_{e'}$, where the pairing $\ip{\blankbox}{\blankbox}$ counts adjacencies between triangulation edges~\cite[\secs 4.2 and 5.5]{wkb}. On $C$, the description is similar, though less explicit. The space of $\C^\times$-valued transverse cocycles comes with a complex symplectic form---Thurston's {\em intersection form}. It pulls back along the shear-bend parameterization to half of Goldman's complex symplectic form~\cite{wp-and-thurs}.

The groupoid of twisted $\on{SL}_2 \C$ local systems is not only essentially a Lie groupoid, but essentially an algebraic groupoid. That makes the twisted character stack and variety algebraic too. On $C'$, each shear-bend parameterization is rational, and the relationships between different shear-bend parameterizations reveal a new structure on the twisted $\on{SL}_2 \C$ character variety. Each edge of an ideal triangulation can be seen as the diagonal of a quadrilateral. If you {\em flip} the edge, by replacing it with the opposite diagonal, you get a new triangulation. The shear-bend coordinates you get from the old and new triangulations are related by a special kind of rational transformation, called a {\em cluster transformation}. The shear-bend parameterizations thus form an atlas for a {\em cluster variety} structure~\cite[end of \S 3.3]{geom-clust-from-surf}\cite[Theorem 1.17]{higher-teich}.
\subsubsection{Computing shear-bend coordinates with abelianization}\label{spec-coords}
The shear-bend parameterization, as we've seen it, lives solidly in the world of hyperbolic geometry. On a punctured surface, however, Gaiotto, Hollands, Moore, and Neitzke have found an elegant way to describe it in terms of local systems, building on the work of Fock and Goncharov. Their version of the shear-bend parameterization turns twisted $\on{SL}_2 \C$ local systems on a punctured surface $C'$ into ordinary $\C^\times$ local systems on a double cover of $C'$. They call this process {\em abelianization}.

You get an abelianization process by giving $C'$ a certain kind of geometric structure, called a {\em half-translation structure}. We'll learn more about these in Section~\ref{tras-ha-tras}. Briefly, a half-translation structure is a flat metric with conical singularities and a foliation by straight lines. For simplicity, let's pick one in which the punctures are shaped like half-infinite cylinders. I'll assume our half-translation structure is typical in various ways.\footnote{There's at least one singularity, all the singularities have cone angle $3\pi$, there are no saddle connections (and hence no periodic leaves), and the leaves spiral down the cylindrical ends rather than plunging straight in.}

The leaves of the foliation trace out an ideal triangulation of $C'$ as they run from puncture to puncture, and they pick out a winding direction around each puncture as they spiral down the cylindrical ends. Our half-translation structure thus contains all the data you need to get a shear-bend parameterization.

It also determines a branched double cover $\Sigma' \to C'$, whose sheets track the two local orientations of the foliation. The surface $\Sigma'$ has a {\em translation structure}: a half-translation structure whose foliation is oriented. Away from the singularity set $\mathfrak{B}$, where the covering is branched, you can see $\Sigma'$ as a submanifold of the unit tangent bundle $UC$. Hence, twisted $\on{SL}_2 \C$ local systems on $C'$ restrict to ordinary $\on{SL}_2 \C$ local systems on $\Sigma' \smallsetminus \mathfrak{B}$. This way of ``untwisting'' a twisted local system will reappear, with a few sophistications, in Section~\ref{untwisting}. An untwisted $\on{SL}_2 \C$ local system turns out to be {\em almost-flat}, meaning it has holonomy $-1$ around each singularity~\cite{fenchel-nielsen}.

Now we're ready to describe the abelianization process. Take a generic\footnote{In the sense of Section~\ref{punkd-ab}.} twisted $\on{SL}_2 \C$ local system on $C'$ and untwist it into an ordinary $\on{SL}_2 \C$ local system $\mathcal{E}$ on $\Sigma' \smallsetminus \mathfrak{B}$. Hollands and Neitzke observed that if you cut $\mathcal{E}$ along the ``critical leaves'' of the foliation, it naturally reassembles into a local system $\mathcal{F}$ with the diagonal subgroup $\C^\times \hookrightarrow \on{SL}_2 \C$ as its structure group. Because $\mathcal{E}$ is almost-flat, $\mathcal{F}$ comes out with holonomy $1$ around each singularity.\footnote{Hollands and Neitzke consider the case where $\mathcal{E}$ extends over the singularities, so $\mathcal{F}$ comes out almost-flat.} In this way, the $\C^\times$ local systems on $\Sigma'$ parameterize a dense open subset of the twisted $\on{SL}_2 \C$ character variety of $C'$, cut out by our genericity assumption.

Abelianization provides direct access to the shear-bend coordinates, which appear as holonomies of the abelianized local system~\cite[\secs 4.4 \thru 4.5]{fenchel-nielsen}. It also offers a tidy coordinate-free description of Goldman's complex symplectic structure. Goldman's construction gives complex symplectic structures to all sorts of character varieties, including both the twisted $\on{SL}_2 \C$ character variety of $C'$ and the ordinary $\C^\times$ character variety of $\Sigma'$. Each abelianization map is a local holomorphic symplectomorphism between these structures~\cite[\S 10.4]{spec-nets}.
\subsection{Results}
\subsubsection{Abelianization without punctures}
In the original account of abelianization, the punctures in $C'$ played an important role, but Gaiotto, Moore, and Neitzke noted that their construction ought to work on compact surfaces as well~\cite[open problem 7]{spec-nets}. The main result of this paper is to confirm that abelianization can be carried out on compact surfaces. In this setting, the ease of working with half-translation structures rather than measured laminations can make abelianization an especially attractive approach to the shear-bend parameterization. The abelianization process is well-suited to numerical approximation, as discussed in Section~\ref{numerical-note}. It may also provide a good view of how the shear parameters vary under change of measured lamination. The geometry of the space of half-translation structures suggests, with some numerical support, that any two abelianization processes should be related by an infinite ordered composition of cluster transformations~\cite{cluster-like}.

Abelianization on a compact half-translation surface $C$ works basically the same as it did on a punctured surface. The catch is that, instead of running from puncture to puncture, the leaves of $C$ spiral around, filling the surface densely. This introduces two challenges, discussed in Section~\ref{no-punks}. First, the cutting and gluing across the critical leaves of $\Sigma \to C$ is supposed to be determined by the behavior of $\mathcal{E}$ at the ends of the leaves. Without punctures, the leaves have nowhere to end. Second, it's tricky to cut and glue along a dense set of lines.

We'll overcome both challenges by looking at local systems on $\Sigma$ from a dynamical point of view. When $\mathcal{E}$ satisfies a condition called {\em uniform hyperbolicity}, its sections grow or decay along the leaves of $\Sigma$ in the same way as they would near the punctures of $\Sigma'$. This dynamical end behavior tells us how to cut and glue, and ensures that the reassembled local system $\mathcal{F}$ diagonalizes as expected. We'll do all our work on an enlargement of $\Sigma$ where the flow along the leaves is nicer, and it's easier to cut and glue along the critical leaves.

The main result, with all that dynamical machinery accounted for, is summarized below. Its statement is very condensed, making it look intimidatingly remote, but the trip from here to there will hopefully feel more like a long hike up a gentle slope than a short climb up a sheer cliff. Words and ideas that won't be formally introduced until later along the route are tagged with references to the relevant sections. You can get an idea of what those sections are about from the table of contents in Section~\ref{contents}.
\begin{thm}[Sections \ref{running-assumptions} \thru \ref{slithering}]
Let $\Sigma$ be a compact translation surface with no saddle connections (\ref{tras-surfs}), and let $\mathfrak{B}$ be its finite set of singularities. Let $\divd{\Sigma}$ be the associated divided surface (\ref{dividing}), whose groupoid of $\on{SL}_2 \C$ local systems is equivalent to the groupoid of $\on{SL}_2 \C$ local systems on $\Sigma \smallsetminus \mathfrak{B}$ (Theorem~\ref{loc-sys-equiv-2d}).

Given a uniformly hyperbolic (\ref{unif-hyp}) $\on{SL}_2 \C$ local system $\mathcal{E}$ on $\divd{\Sigma}$, we can find a new $\on{SL}_2 \C$ local system $\mathcal{F}$ and a stalkwise isomorphism $\Upsilon \maps \mathcal{E} \to \mathcal{F}$, supported on a dense subspace of $\divd{\Sigma}$, with the following properties:
\begin{itemize}
\item The deviation (\ref{deviations-intuition} \thru \ref{dev-loc-const}) of $\Upsilon$ behaves like the deviation of an abelianized local system from its original would behave on a punctured surface (\ref{ab-overview}).
\item The structure group of $\mathcal{F}$ reduces to the diagonal subgroup $\C^\times \hookrightarrow \on{SL}_2 \C$.
\end{itemize}
\end{thm}
\begin{proof}
Sections \ref{ab-conv} \thru \ref{ab-deliv}.
\end{proof}
As a demonstration, in Section~\ref{toy-example}, we'll use abelianization on a compact translation torus to parameterize the $\on{SL}_2 \C$ local systems on a punctured torus.

On a compact surface, abelianization once again provides access to the shear parameterization. Although it's not as direct as it was in the punctured case, it's convenient for numerical computations, as discussed in Section~\ref{numerical-note}.
\begin{thm}[Section~\ref{shear-params}]\label{shear-access}
Let $C_\text{hyp}$ be a compact surface equipped with a maximal measured geodesic lamination. Given a hyperbolic structure on $C_\text{hyp}$, you can use Gupta's ``collapsing'' process to build a translation surface $\Sigma_\text{flat}$, whose geometry encodes the shear parameters of the hyperbolic structure. The associated divided surface $\divd{\Sigma}_\text{flat}$ comes with some extra data:
\begin{itemize}
\item An $\on{SL}_2 \R$ local system $\mathcal{E}$ describing the hyperbolic structure of $C_\text{hyp}$.
\item An $\R^\times \hookrightarrow \on{SL}_2 \R$ local system $\mathcal{F}$ describing the ``vertical part'' of the translation structure of $\Sigma_\text{flat}$.
\item A stalkwise isomorphism $\Upsilon \maps \mathcal{E} \to \mathcal{F}$ supported on a dense subspace of $\divd{\Sigma}_\text{flat}$.
\end{itemize}
The local system $\mathcal{F}$ and the stalkwise isomorphism $\Upsilon$ are the same as the ones you get by abelianizing $\mathcal{E}$. The shear parameters appear as holonomies of $\mathcal{F}$.
\end{thm}
\begin{proof}
Section~\ref{shear-params}.
\end{proof}
This picture should generalize to the shear-bend parameterization, with the translation surface $\Sigma_\text{flat}$ becoming a ``pleated translation surface'' of some kind.
\subsubsection{Related work in higher rank}
On a compact surface, the twisted $\on{SL}_2 \R$ local systems that come from hyperbolic structures form a connected component of the twisted $\on{SL}_2 \R$ character variety. For every $n \ge 2$, the twisted $\on{SL}_n \R$ character variety has an analogous component, called the {\em Hitchin component}~\cite[\S 7]{higgs-geom}. Bonahon and Dreyer recently generalized the shear parameterization from the $\on{SL}_2 \R$ Hitchin component to any $\on{SL}_n \R$ Hitchin component~\cite{hit-chars}. Their approach pivots on the fact that local systems in the Hitchin component satisfy a dynamical condition called the Anosov property.

This paper shares key technical elements with Bonahon and Dreyer's. One is the {\em slithering jump} we construct in Section~\ref{slithering}, which carries the same kind of information as Bonahon and Dreyer's {\em slithering maps}. Another is our reliance on the fact that uniformly hyperbolic {\em Markov cocycles} have Lipschitz continuous stable distributions, which brings to mind Bonahon and Dreyer's reliance on the fact that Anosov local systems have H\"{o}lder continuous limit curves. These common elements are inherited: both papers descend from Bonahon's original description of the shear-bend parameterization~\cite{shearing-bending}.

The approach used here and the approach of Bonahon and Dreyer have enough machinery in common to provide some hope that they might be combined, extending the shear-bend parameterization to uniformly hyperbolic twisted $\on{SL}_n \C$ local systems for $n \ge 2$. Over a well-behaved dynamical system, the stable distributions of a uniformly hyperbolic $\on{SL}_n \C$ cocycle are H\"{o}lder continuous in most places, as Ara\'{u}jo, Bufetov, and Filip recently proved~\cite{oseledets-holder}. Using their result, the main argument of this paper can be carried out with H\"{o}lder continuity in place of Lipschitz continuity, at the cost of some extra dynamical assumptions on $\Sigma$~\cite{warping-geom}. That version of the argument, though, does some tricks with Lyapunov exponents that may not work in higher rank.
\subsection{Contents}\label{contents}
\begin{itemize}
\begin{tocsection}{Introduction}
\tocitem{Section~\ref{intro}}
Context, summary of results, and various administrative things, including notation that will be used throughout the paper.
\end{tocsection}
\begin{tocsection}{Tools}
\tocitem{Section~\ref{warping}}
Abstracting from the idea of deforming a flat connection on a smooth bundle, we get a more general way of deforming local systems, called {\em warping}.
\tocitem{Section~\ref{dividing}}
After a brief introduction to translation surfaces, we describe a way to enlarge a translation surface by dividing its critical leaves into two-lane {\em critical roads}. We show that the resulting {\em divided surface} has useful topological and dynamical properties, and that it resembles the original surface both dynamically and in terms of its local systems.
\tocitem{Section~\ref{warping-on-divd}}
We show how the warping process from Section~\ref{warping} can be used to make sense of the idea of cutting and gluing a local sytem along the critical roads of a divided surface, even when the critical roads fill the surface densely.
\tocitem{Section~\ref{unif-hyp}}
We take a well-studied condition on dynamical cocycles, called {\em uniform hyperbolicity}, and reinterpret it as a condition on local systems on compact translation surfaces.
\end{tocsection}
\begin{tocsection}{Abelianization}
\tocitem{Section~\ref{ab-princ}}
We review how abelianization works for $\on{SL}_2 \C$ local systems on a translation surface with punctures, point out the obstacles to carrying it out a compact translation surface, and describe how these obstacles can be overcome using the machinery from the previous sections. Section \ref{slithering} contains the main product of the paper: instructions for abelianizing an $\on{SL}_2 \C$ local system on a compact translation surface, which are guaranteed to work under the conditions laid out in Section~\ref{running-assumptions}.
\tocitem{Section~\ref{ab-conv}}
We show that abelianization, as defined by the instructions in Section~\ref{slithering}, produces a well-defined local system, assuming the conditions from Section~\ref{running-assumptions}.
\tocitem{Section~\ref{ab-deliv}}
We reduce the structure group of the abelianized local system to the diagonal subgroup $\C^\times \hookrightarrow \on{SL}_2 \C$, assuming the conditions from Section~\ref{running-assumptions}.
\end{tocsection}
\begin{tocsection}{Examples}
\tocitem{Section~\ref{toy-example}}
A non-rigorous sample computation that uses abelianization to find holomorphic coordinates on the $\on{SL}_2 \C$ character variety of the punctured torus.
\tocitem{Section~\ref{shear-params}}
We see that a geometric technique for computing the shear parameters of a compact hyperbolic surface is a special case of the abelianization process. We mention a numerical recipe for computing shear parameters with abelianization.
\end{tocsection}
\begin{tocsection}{Appendices}
\tocitem{Appendix~\ref{warping-tech}}
A pair of small technical lemmas for Section~\ref{warping}.
\tocitem{Appendix~\ref{rel-dyne}}
A formalism for dynamical systems described by relations, used throughout the paper.
\tocitem{Appendix~\ref{ord-prod}}
Results on infinite ordered products, used heavily in Sections \ref{warping-on-divd}, \ref{ab-conv}, and \ref{ab-deliv}.
\tocitem{Appendix~\ref{euclidean}}
Linear algebra facts about the complex analogue of a Euclidean plane, used in Section~\ref{ab-deliv}.
\tocitem{Appendix~\ref{markov-stable-lines}}
A stand-alone proof of the fact, used in Section~\ref{stable-lipschitz}, that a uniformly hyperbolic {\em Markov cocycle} has Lipschitz continuous stable distributions.
\tocitem{Appendix~\ref{punk-shapes}}
A list of standard puncture shapes for translation surfaces and an explanation of where they come from, included to clarify the review in Section~\ref{ab-princ}.
\end{tocsection}
\end{itemize}
\subsection{Setup}
\subsubsection{Running notation}\label{running-notation}
The terminology of this section hasn't been introduced yet, but will be familiar to readers familiar with translation surfaces. If you'd like to become familiar with translation surfaces, skip ahead to Section~\ref{tras-ha-tras}.

From now on, $\Sigma$ will be a compact translation surface, $\mathfrak{B}$ its set of singularities, and $\mathfrak{W}$ the union of its critical leaves. Within $\mathfrak{W}$, let $\mathfrak{W}^+$ and $\mathfrak{W}^-$ be the unions of the backward- and forward-critical leaves, respectively. Saying that $\Sigma$ has no saddle connections is the same as saying that $\mathfrak{W}^+$ and $\mathfrak{W}^-$ are disjoint. The $\pm$ labeling is meant to evoke the fact that, in the absence of saddle connections, the vertical flow is well-defined on $\mathfrak{W}^+$ for all positive times, and on $\mathfrak{W}^-$ for all negative times.

I should stress that $\Sigma$ doesn't need to be the translation double cover of a half-translation surface, and there are nice examples of abelianization where it isn't. One of these is discussed in Section~\ref{toy-example}.
\subsubsection{Index of symbols}
Symbols can be hard to look up, so here's a list of unusual symbols that appear frequently in this paper, with references to the sections where they're defined. The first three are introduced in this paper, and the fourth is common, but not universal, in analysis.
\begin{center}
\begin{tabular}{rl}
$\divd{\blankbox}$ & Divided interval or surface (Sections \ref{divd-intvl-construct} and \ref{divd-surf-construct}). \\
$\frakd{\blankbox}$ & Fractured interval or surface (same sections). \\
$\nom{\blankbox}$ & Intersection with fractured interval or surface (Sections \ref{frakd-props-1d} and \ref{divd-props-2d}). \\
$\blankbox \lesssim \blankbox$ & Bounded by a constant multiple (Section~\ref{unif-hyp-motive}).
\end{tabular}
\end{center}
\subsection{Acknowledgements}
This paper is a fork of my Ph.D. thesis, which would not have been started or finished without the tireless guidance of Andy Neitzke. Its motivation has been clarified, and its main argument simplified dramatically, following a conversation with Bill Goldman and Richard Wentworth at Jacques Hurtubise's 60th birthday conference.

I'm very grateful to Jen Berg for pointing out the argument used to prove Proposition~\ref{stalk-res-iso}, and to Brice Loustau for walking me through the material in footnote~\ref{hyperkahler-note} (I am, of course, responsible for any mistakes in these passages). I've also enjoyed the benefit of conversations, some short and some long, with Sona Akopian, David Ben-Zvi, Francis Bischoff, Francis Bonahon, Richard Derryberry, Luis Duque, Tim Magee, Tom Mainiero, Taylor McAdam, Amir Mohammadi, Javier Morales, and Max Riestenberg, as well as diction brainstorming with Eliana Fenyes. This research was supported in part by NSF grants 1148490 and 1160461.
\section{Warping local systems}\label{warping}
\subsection{Overview}
Gaiotto, Hollands, Moore, and Neitzke introduced abelianization as a cutting and gluing process acting on flat bundles, rather than local systems. On the punctured surfaces where abelianization was first described, this point of view works great, because the cutting and gluing only has to be done along a few isolated lines. On a compact surface, however, the cutting and gluing happens along lines that fill the surface densely, and keeping track of what it does to the topology of the bundle becomes a serious hassle. By working with local systems, which have no explicit total space to worry about, we'll avoid this inconvenience. Local systems are also less sensitive to the topology of the space they live over, a feature we'll take great advantage of later.

In this section, we'll develop a general tool for deforming locally constant sheaves, called {\em warping}. We'll see in Section~\ref{warping-on-divd} that warping includes our dense cutting and gluing of local systems as a special case.
\subsection{Conventions for local systems}\label{loc-sys-conventions}
\subsubsection{Basics}
Given a group $G$, we define a $G$ local system to be a locally constant sheaf of $G$-sets whose stalks are all $G$-torsors. Until Section~\ref{unif-hyp}, it won't matter much what $G$ is, so we'll often just talk about local systems in general.

The category of $G$-sets is a nice target category for sheaves, because it's a type of algebraic structure~\cite[Tag 007L]{stacks-project}.\footnote{See \cite[\S 2.15]{cat-alg} for proofs of completeness and cocompleteness.} Throughout this article, ``sheaf'' will mean a sheaf whose target category is a type of algebraic structure. With this said, we can define a constant sheaf to be a sheaf of locally constant functions---that is, functions constant on every connected component of their domain. In a locally connected space, every connected component of an open set is open, so the constant sheaf with value $A$ is characterized by the property that it sends every connected, non-empty open set to $A$, and every inclusion of such sets to $1_A$.

If $\mathcal{F}$ is a constant sheaf on a locally connected space, the stalk restriction morphism $\mathcal{F}_{x \in X} \maps \mathcal{F}_x \leftarrow \mathcal{F}_X$ is an isomorphism for every $x \in X$. In our context, the converse is true as well:
\begin{prop}\label{stalk-res-iso}
Suppose $\mathcal{F}$ is a sheaf on a locally connected space $X$. If the stalk restriction $\mathcal{F}_{x \in X}$ is an isomorphism for every $x \in X$, then $\mathcal{F}$ is isomorphic to a constant sheaf.
\end{prop}
\begin{proof}
Let $\bar{\mathcal{F}}$ be the constant sheaf with value $\mathcal{F}_X$. For each $U \subset X$, the restriction $\mathcal{F}_{U \subset X}$ gives a morphism $\bar{\mathcal{F}}_U \to \mathcal{F}_U$, and these morphisms fit together into a natural transformation from $\bar{\mathcal{F}}$ to $\mathcal{F}$. This natural transformation induces an isomorphism on every stalk, so it's an isomorphism of the underlying sheaves of sets, and therefore an isomorphism of sheaves of algebraic structures. (Many thanks to Jen Berg for pointing out this argument.)
\end{proof}
To save ink, let's say a connected open subset of a space is {\em simple} with respect to a sheaf if the restriction of the sheaf to the subset is isomorphic to a constant sheaf. Notice that the value of a $G$ local system on a simple set is a $G$-torsor.

We'll frequently and without fanfare make use of the fact that a sheaf defined on a basis for a topological space extends uniquely (up to canonical isomorphism) to a sheaf on the full poset of open sets~\cite[Tag~009H, Lemma~9]{stacks-project}.
\subsubsection{Linear local systems}\label{lin-loc-sys}
Let's say $G$ is a linear group---a subgroup of the automorphism group of some finite-dimensional vector space $R$. Define a {\em $G$-structure} on a vector space $V$ to be an isomorphism $V \to R$ modulo postcomposition by $G$. When a vector space is equipped with a $G$-structure, we'll call it a {\em $G$ vector space}. Observe that $G$-structures pull back along isomorphisms. You can say an isomorphism between $G$ vector spaces is {\em structure-preserving} if it pulls the $G$-structure on the target back to the $G$-structure on the source.

Lots of familiar structures on an $n$-dimensional complex vector space are examples of $G$-structures.
\begin{itemize}
\item A volume form is an $\on{SL}_n \C$-structure.
\item An inner product is a $\on{U}_n$-structure.
\item A complete flag is a structure of the upper-triangular subgroup of $\on{GL}_n \C$.
\item A line decomposition is a structure of the diagonal subgroup of $\on{GL}_n \C$.
\end{itemize}
In general, if $G$ is defined as the group of automorphisms of $R$ preserving a certain structure, you can turn around and define that structure as a $G$-structure.

The isomorphisms $V \to R$ that define a $G$-structure are interesting in their own right. We'll call them {\em structured frames}. The structured frames for the familiar $G$-structures listed above are also familiar objects.
\begin{itemize}
\item A unit-volume basis is an $\on{SL}_n \C$-structured frame.
\item An orthonormal basis is an $\on{U}_n$-structured frame.
\item A basis subordinate to a complete flag is an upper-triangular-structured frame.
\item A basis subordinate to a line decomposition is a diagonal-structured frame.
\end{itemize}
By definition, the set of structured frames for a $G$ vector space $V$ is the orbit of an isomorphism $V \to R$ under the action of $G$ by postcomposition. Since $G$ acts freely on isomorphisms $V \to R$, it acts freely and transitively on the set of structured frames. In other words, the structured frames form a $G$-torsor. The set of structured frames is what defines a $G$-structure in the first place, so there's a pithier way to say it: a $G$-structure is an example of a $G$-torsor. That's why we're talking about $G$-structures in a section on $G$ local systems.

We can define a functor from $G$ vector spaces to $G$-torsors by sending each $G$ vector space to its $G$-torsor of structured frames. This functor is an equivalence of categories. If we define a ``compound $G$ vector space'' to be a formal direct sum of $G$ vector spaces, the structured frames functor should extend to an equivalence between the category of compound $G$ vector spaces and the category of $G$-sets generated by taking limits and colimits of $G$-torsors.

Let's say a {\em linear $G$ local system} is a locally constant sheaf of compound $G$ vector spaces whose stalks are all single $G$ vector spaces. Using the equivalence above, we can realize any $G$ local system as a linear $G$ local system. Later in the paper, when we're dealing exclusively with $\on{SL}_2 \C$ local systems, let's assume that all our local systems are linear. Concretely, that means we'll be working with locally constant sheaves of two-dimensional real vector spaces with volume forms.
\subsection{Deviations of flat connections}\label{deviations-intuition}
Say $E \to X$ is a smooth bundle, $A$ and $A'$ are flat connections on $E$, and $\mathcal{E}$ and $\mathcal{E}'$ are the corresponding sheaves of flat sections. The stalk $\mathcal{E}_x$ is the space of germs of flat sections of $A$ at $x$. Sending each germ in $\mathcal{E}_x$ to the unique germ in $\mathcal{E}'_x$ that has the same value at $x$ gives an isomorphism $\Upsilon_x \maps \mathcal{E}_x \to \mathcal{E}'_x$.
\begin{center}
\begin{tikzpicture}
\stalkiso
\end{tikzpicture}
\end{center}
If $U \subset X$ is simple with respect to both $\mathcal{E}$ and $\mathcal{E}'$, we can visualize $\Upsilon_x$ by its action $\Upsilon^U_x \maps \mathcal{E}_U \to \mathcal{E}'_U$ on flat sections over $U$, as shown above. Then, for any $x, y \in U$, we can define an automorphism $\upsilon^U_{yx}$ of $\mathcal{E}_U$ that tells us how parallel transport along $A'$ deviates from parallel transport along $A$:
\begin{center}
\begin{minipage}{0.5\textwidth}
\xymatrix@R=4mm@C=8mm{
& \mathcal{E}'_U \ar[dl]_\in \\
\mathcal{E}'_y \ar[dd]_{\Upsilon_y^{-1}} & & \mathcal{E}'_x \ar[ul]_{\in^{-1}} \\ \\
\mathcal{E}_y \ar[dr]_{\in^{-1}} & & \mathcal{E}_x \ar[uu]_{\Upsilon_x} \\
& \mathcal{E}_U \ar[ur]_\in \ar@{.>}@`{p+(30,35),p+(-30,35)}|{\upsilon^U_{yx}}
}
\end{minipage}
\begin{minipage}{0.5\textwidth}
\centering
\begin{tikzpicture}
\sectiondev
\end{tikzpicture}
\end{minipage}
\end{center}
This automorphism is characterized by the property that $\Upsilon^U_y \upsilon^U_{yx} = \Upsilon^U_x$. If $V \subset U$ is a simple neighborhood of $x$ and $y$, the automorphisms $\upsilon^U_{yx}$ and $\upsilon^V_{yx}$ commute with the restriction map $\mathcal{E}_{V \subset U}$, so all these automorphisms fit together into a natural automorphism $\upsilon_{yx}$ of the functor we get by restricting $\mathcal{E}$ to the poset of simple neighborhoods of $x$ and $y$. Restricting $\mathcal{E}$ further to the simple neighborhoods of three points $x$, $y$, and $z$, we can observe that $\upsilon_{zy} \upsilon_{yx} = \upsilon_{zx}$. Collectively, the natural automorphisms $\{\upsilon_{yx}\}_{x, y \in X}$ might be called the {\em deviation} of $A'$ from $A$.
\subsection{Deviations of locally constant sheaves}\label{dev-loc-const}
More generally, say $X$ is just a locally connected space, and $\mathcal{F}$ and $\mathcal{F}'$ are locally constant sheaves on $X$. Given a stalkwise isomorphism $\Upsilon_x \maps \mathcal{F}_x \to \mathcal{F}'_x$, the construction of the natural automorphisms $\upsilon_{yx}$ goes through exactly as before. Since our choice of stalkwise isomorphism was not necessarily canonical, we should really talk about ``the deviation of $\Upsilon$'' instead of ``the deviation of $\mathcal{F}'$ from $\mathcal{F}$.''

Now, consider three locally constant sheaves $\mathcal{G}$, $\mathcal{F}$, and $\mathcal{F}'$ over $X$. If two stalkwise isomorphisms $\Phi \maps \mathcal{G} \to \mathcal{F}$ and $\Psi \maps \mathcal{G} \to \mathcal{F}'$ have the same deviation, $\upsilon$, how similar must they be? As it turns out, they're as good as identical: there's a unique natural isomorphism $T \maps \mathcal{F} \to \mathcal{F}'$ such that $\Psi_x = T_x \Phi_x$ for all $x \in X$. Here's why.

If we can find a natural isomorphism like this, it's clearly unique, because a natural transformation of sheaves is completely described by its action on stalks. Now, let's find one. Choose a basis $\mathcal{B}$ for the topology of $X$ consisting of sets which are simple with respect to all three sheaves. (We can do this because the sheaves are locally constant, and $X$ is locally conected.) Look at any basis element $U \in \mathcal{B}$. For any $x \in U$, define a morphism $T^U_x \maps \mathcal{F}_U \to \mathcal{F}'_U$ by
\[ \xymatrix{
\mathcal{F}_U \ar@{.>}@(ur,ul)[rrrr]^{T^U_x} \ar[r]_\in & \mathcal{F}_x \ar[r]_{{\Phi_x}^{-1}} & \mathcal{G}_x \ar[r]_{\Psi_x} & \mathcal{F}'_x \ar[r]_{\in^{-1}} & \mathcal{F}'_U
} \]
What if we had chosen another point $y \in U$ instead? Let $\upsilon$ be the shared deviation of $\Phi$ and $\Psi$, and consider the diagram
\[ \xymatrix@R=8mm@C=4mm{
& \mathcal{F}_y \ar[rr]^{{\Phi_y}^{-1}} && \mathcal{G}_y \ar[rr]^{\in^{-1}} && \mathcal{G}_U \ar[rr]^\in && \mathcal{G}_y \ar[rr]^{\Psi_y} && \mathcal{F}'_y \ar[dr]^{\in^{-1}} \\
\mathcal{F}_U \ar[ur]^\in \ar[dr]_\in & && && && && & \mathcal{F}'_U \\
& \mathcal{F}_x \ar[rr]_{{\Phi_x}^{-1}} && \mathcal{G}_x \ar[rr]_{\in^{-1}} && \mathcal{G}_U \ar[uu]^{\upsilon^U_{yx}}\ar[rr]_\in && \mathcal{G}_x \ar[rr]_{\Psi_x} && \mathcal{F}'_x \ar[ur]_{\in^{-1}} \\
} \]
Notice that the bottom path is $T^U_x$, and the top path is $T^U_y$. The left and right chambers both commute, because $\upsilon$ is the deviation of both $\Phi$ and $\Psi$, so $T^U_x = T^U_y$ for all $x, y \in U$. Thus, we really have just one morphism $T^U \maps \mathcal{F}_U \to \mathcal{F}'_U$, which can be written in terms of any point in $U$.

For any other basis element $V \subset U$, writing $T^V$ and $T^U$ in terms of the same point $x \in V$ makes it easy to check that the square
\[ \xymatrix{
\mathcal{F}_U \ar[d]_\subset \ar[r]^{T^U} & \mathcal{F}'_U \ar[d]^\subset \\
\mathcal{F}_V \ar[r]_{T^V} & \mathcal{F}'_V
} \]
commutes. Since we've been working with arbitrary basis elements, we now see that the morphisms $\{T^U\}_{U \in \mathcal{B}}$ fit together into a natural transformation $T \maps \mathcal{F} \to \mathcal{F}'$, and it's clear by construction that $\Psi_x = T_x \Phi_x$ for all $x \in X$.
\subsection{Warping locally constant sheaves}\label{warping-sheaves}
We just saw that, under favorable conditions, a stalkwise isomorphism is determined up to canonical isomorphism by its deviation. Let's see if we can go the other way and produce a stalkwise isomorphism with a specified deviation. First, we have to say what it means to specify a deviation.

Suppose $\mathcal{F}$ is a locally constant sheaf on a locally connected space $X$, $D$ is a dense subset of $X$, and $\mathcal{B}$ is a basis for the topology of $X$ consisting of $\mathcal{F}$-simple sets. To specify a {\em deviation} from $\mathcal{F}$ with support $D$, defined over the basis $\mathcal{B}$, we give for each pair of points $x, y \in D$ and each neighborhood $U \in \mathcal{B}$ of $x$ and $y$ an automorphism $\upsilon^U_{yx}$ of $\mathcal{F}_U$. These automorphisms have to fit together as follows:
\begin{itemize}
\item If $V \subset U$ is a basis element containing $x$ and $y$, the automorphisms $\upsilon^U_{yx}$ and $\upsilon^V_{yx}$ commute with the restriction morphism $\mathcal{F}_{V \subset U}$.
\item For any three points $x, y, z \in D$, we have $\upsilon^U_{zy} \upsilon^U_{yx} = \upsilon^U_{zx}$.
\end{itemize}
The first condition just says that $\upsilon_{yx}$ is a natural automorphism of the restriction of $\mathcal{F}$ to the poset of basis elements containing both $x$ and $y$.

It turns out that, given a deviation $\upsilon$ from $\mathcal{F}$, we can always produce a locally constant sheaf $\mathcal{F}'$ and a stalkwise isomorphism $\Upsilon \maps \mathcal{F} \to \mathcal{F}'$, supported on $D$, whose deviation is $\upsilon$. We'll call this process {\em warping}. Here's how it's done. For each $U \in \mathcal{B}$, pick a point $x_U \in U \cap D$. Define $\mathcal{F}'_U$ to be the same as $\mathcal{F}_U$, but with the warped restriction morphism
\[ \mathcal{F}'_{V \subset U} = \mathcal{F}_{V \subset U}\;\upsilon^U_{x_V x_U} \]
for each basis element $V \subset U$. Because every basis element is $\mathcal{F}$-simple, $\mathcal{F}_{V \subset U}$ is an isomorphism, so $\mathcal{F}'_{V \subset U}$ is an isomorphism too. It follows that the stalk restriction $\mathcal{F}'_{x \in U}$ is an isomorphism for any $x \in U$, so every basis element is $\mathcal{F}'$-simple. The stalkwise isomorphism $\Upsilon \maps \mathcal{F} \to \mathcal{F}'$ is given by
\[ \Upsilon_x = \mathcal{F}'_{x \in U}\;\upsilon^U_{x_U x}\;{\mathcal{F}_{x \in U}}^{-1} \]
for any basis element $U$ containing $x \in D$.

There are three claims implicit in the description of $\mathcal{F} \overset{\Upsilon}{\to} \mathcal{F}'$ above:
\begin{itemize}
\item $\mathcal{F}'$ is a locally constant sheaf.
\item The definition of $\Upsilon_x$ doesn't depend on our choice of neighborhood $U$.
\item The deviation of $\mathcal{F}'$ from $\mathcal{F}$ is $\upsilon$.
\end{itemize}
Let's check these claims.
\paragraph{$\mathcal{F}'$ is a locally constant sheaf}
The functoriality of $\mathcal{F}'$ follows easily from the fact that $\upsilon$ is a deviation. To verify that $\mathcal{F}'$ is a sheaf, pick any element $U$ of the basis $\mathcal{B}$. Suppose that for each basis element $V \subset U$, we have an element $s_V$ of $\mathcal{F}'_V$, and these elements commute with the restriction morphisms of $\mathcal{F}'$. We need to find an element $s$ of $\mathcal{F}'_U$ that restricts to $s_V$ on every basis element $V \subset U$. Since each of the restrictions $\mathcal{F}'_{V \subset U}$ is an isomorphism, there can only be one element like this, and it will exist if and only if the elements ${\mathcal{F}'_{V \subset U}}^{-1}\,s_V$ match for all the basis elements $V \subset U$.

For two basis elements $W \subset V$ contained in $U$, we have
\begin{align*}
{\mathcal{F}'_{W \subset U}}^{-1}\;s_W & = {\mathcal{F}'_{W \subset U}}^{-1}\;\mathcal{F}'_{W \subset V}\;s_V \\
& = {\mathcal{F}'_{V \subset U}}^{-1}\;s_V,
\end{align*}
so $W$ and $V$ give the same element. Since any two overlapping basis elements contain another basis element in their intersection, it follows that any two overlapping basis elements give the same element as well. Because $U$ is connected, any two basis elements can be linked by a finite sequence of overlapping basis elements (Appendix~\ref{lily-pads}). Therefore, all the basis elements contained in $U$ give the same element. This is the element $s$ we were looking for, completing our proof that $\mathcal{F}'$ is a sheaf.

To see that $\mathcal{F}'$ is locally constant, first recall that the elements of $\mathcal{B}$ are simple, so the restriction arrows of $\mathcal{F}$ over $\mathcal{B}$ are isomorphisms. Thus, the restriction arrows of $\mathcal{F}'$ are isomorphisms as well. Now, pick any point $x \in X$, not necessarily in $D$. Applying $\mathcal{F}'$ to the poset of basis elements containing $x$ yields a downward-directed diagram whose arrows are all isomorphisms. The defining arrows from a diagram like this to its colimit are always isomorphisms (Appendix~\ref{iso-colim}). In other words, the stalk restriction $\mathcal{F}'_{x \in U}$ is an isomorphism for every $U \in \mathcal{B}$ containing $x$. Since $x$ was an arbitrary point in $X$, it follows by Proposition~\ref{stalk-res-iso} that $\mathcal{F}'$ is locally constant.
\paragraph{$\Upsilon_x$ is well-defined}
To see that the definition of $\Upsilon_x$ doesn't depend on our choice of neighborhood, first observe that for two basis elements $V \subset U$ containing $x$,
\begin{align*}
(\mathcal{F}'_{x \in U})\;\upsilon^U_{x_U x}\;{\mathcal{F}_{x \in U}}^{-1} & = (\mathcal{F}'_{x \in V}\;[\mathcal{F}'_{V \subset U}])\;\upsilon^U_{x_U x}\;{\mathcal{F}_{x \in U}}^{-1} \\
& = \mathcal{F}'_{x \in V}[\mathcal{F}_{V \subset U}\;(\upsilon^U_{x_V x_U}]\;\upsilon^U_{x_U x})\;{\mathcal{F}_{x \in U}}^{-1} \\
& = \mathcal{F}'_{x \in V}\;[\mathcal{F}_{V \subset U}\;(\upsilon^U_{x_V x})]\;{\mathcal{F}_{x \in U}}^{-1} \\
& = \mathcal{F}'_{x \in V}\;[\upsilon^V_{x_V x}\;(\mathcal{F}_{V \subset U}]\;{\mathcal{F}_{x \in U}}^{-1}) \\
& = \mathcal{F}'_{x \in V}\;\upsilon^V_{x_V x}\;({\mathcal{F}_{x \in V}}^{-1}),
\end{align*}
so $V$ and $U$ give the same isomorphism. Since any two basis elements containing $x$ contain another basis element in their intersection, it follows that every basis element containing $x$ gives the same isomorphism.
\paragraph{The deviation of $\mathcal{F}'$ from $\mathcal{F}$ is $\upsilon$}
Finally, let $\delta$ be the deviation of $\Upsilon$. It's easy to calculate $\delta^U_{yx}$ by defining $\Upsilon_x$ and $\Upsilon_y$ in terms of $U$:
\begin{align*}
\delta^U_{yx} & = {\mathcal{F}_{y \in U}}^{-1}\;(\Upsilon_y^{-1})\;\mathcal{F}'_{y \in U}\;{\mathcal{F}'_{x \in U}}^{-1}\;(\Upsilon_x)\;\mathcal{F}_{x \in U} \\
& = {\mathcal{F}_{y \in U}}^{-1}\;(\mathcal{F}_{y \in U}\;\upsilon^U_{y x_U}\;{\mathcal{F}'_{y \in U}}^{-1})\;\mathcal{F}'_{y \in U}\;{\mathcal{F}'_{x \in U}}^{-1}\;(\mathcal{F}'_{x \in U}\;\upsilon^U_{x_U x}\;{\mathcal{F}_{x \in U}}^{-1})\;\mathcal{F}_{x \in U} \\
& = \upsilon^U_{y x_U}\;\upsilon^U_{x_U x}\; \\
& = \upsilon^U_{yx}.
\end{align*}
Thus, $\delta = \upsilon$, as claimed.
\subsection{Warping local systems}
Local systems are just a special kind of locally constant sheaves, so all the constructions of the previous sections can be applied to them. In this case, the definition of a deviation can be pared down a bit, because if $\mathcal{F}$ is a $G$ local system and $U$ is an $\mathcal{F}$-simple open set, an automorphism of $\mathcal{F}_U$ is just an element of $G$.

Warping a local system always produces another local system. To see why, take a $G$ local system $\mathcal{F}$ and warp it by some deviation. The warped sheaf $\mathcal{F}'$ is locally constant, and stalkwise isomorphic to $\mathcal{F}$ over the support of the deviation. Because the support is dense, it follows that every stalk of $\mathcal{F}'$ is isomorphic to a stalk of $\mathcal{F}$, and hence a $G$-torsor.
\section{Dividing translation surfaces}\label{dividing}
\subsection{Overview}
For working out the technical details of abelianization, it will be useful to embed the surface $\Sigma \smallsetminus \mathfrak{B}$ in a larger space $\divd{\Sigma}$, called the {\em divided surface}, whose local systems are naturally in correspondence with the local systems on $\Sigma \smallsetminus \mathfrak{B}$. On the divided surface, we can stand infinitesimally close to any critical leaf, streamlining our discussion of the abelianization process in Sections \ref{ab-princ} and \ref{ab-conv}.

Removing the critical leaves of $\Sigma$ from the divided surface yields a compact space $\frakd{\Sigma}$, called the {\em fractured surface}, which can locally be metrized in a very natural way. Its metric properties will play a crucial role in Section~\ref{ab-deliv}, where we prove that abelianization does the job it's meant to do.

There's a geometric perspective on the shear parameterization, outlined in Sections 6.3 \thru 6.4 of \cite{warping-geom} and Section~\ref{shear-params} of this paper, which leads naturally to the idea of a divided surface. See Section~\ref{collapsing} for details.
\subsection{A review of translation and half-translation surfaces}\label{tras-ha-tras}
\subsubsection{Translation surfaces}\label{tras-surfs}
A non-singular {\em translation surface} is a manifold whose charts are open subsets of $\R^2$ and whose transition maps are translations. Every translation surface comes with a bunch of geometric structures induced by the translation-invariant geometric structures on $\R^2$, which include:
\begin{itemize}
\item The flat metric.
\item The four cardinal directions: up, down, right, and left.
\item The vertical and horizontal foliations, whose leaves are vertical and horizontal lines. Both foliations can be oriented; we'll orient them upward and rightward, respectively.
\item The vertical flow, which moves points upward at unit speed. On a surface which is non-compact, as most non-singular translation surfaces are, this flow might not be defined everywhere at all times. In general, the flow at a given time will be only a {\em bicontinuous relation}, rather than a homeomorphism (see Appendix~\ref{rel-dyne} for details).
\end{itemize}
Whenever I refer to a foliation of a translation surface, I mean the vertical one, unless I say otherwise.

It's conventional, and convenient, to allow translation surfaces to have {\em conical singularities}, which look like this:
\begin{center}
\begin{tikzpicture}
\matrix[row sep=1cm, column sep=1.2cm]{
\notchdown{pwviolet}{$6$}{pwblue}{$5$} & \notchdown{pwgreen}{$4$}{pwgold}{$3$} & \notchdown{pworange}{$2$}{pwpink}{$1$} \\
\notchup{pwgreen}{$4$}{pwblue}{$5$} & \notchup{pworange}{$2$}{pwgold}{$3$} & \notchup{pwviolet}{$6$}{pwpink}{$1$} \\
};
\end{tikzpicture}
\end{center}
The like-numbered triangles are identified through translation. The triangles include the cross-marked center points, which get quotiented down to a single point by the identifications. In the case shown above, the total angle around the singularity is $6\pi$; in general, any even multiple of $\pi$ is possible. It's sometimes useful to mark a discrete set of ordinary points as ``singularities'' of cone angle $2\pi$.

A compact translation surface can have only finitely many singularities. Conformally, the vertical and horizontal foliations in the neighborhood of a singularity look like this:
\begin{center}
\begin{tikzpicture}
\matrix[row sep=0.2cm, column sep=0.5cm]{
\node {\includegraphics{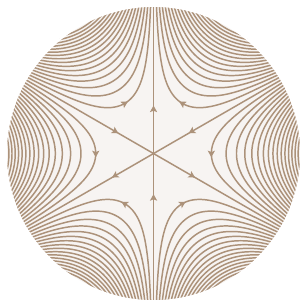}}; &
\node {\includegraphics{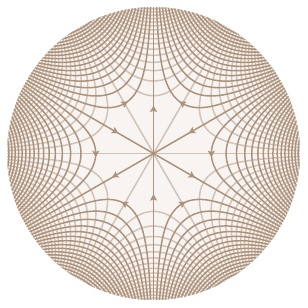}};
\\
\node {\small Vertical}; &
\node {\small Vertical and horizontal};
\\
};
\end{tikzpicture}
\end{center}
The vertical leaves that dive into the singularity are called {\em forward-critical}, and the ones that shoot out of the singularity are called {\em backward-critical}. A leaf which is critical in both directions is called a {\em saddle connection}. Critical leaves are spaced evenly around the singularity at angles of $\pi$.

The vertical flow on a singular translation surface is only defined away from the singularities. It acts by bicontinuous relations, making points on the critical leaves disappear as they fall into the singularities. When restricted to the complement of the critical leaves, the vertical flow acts by homeomorphisms.

A translation surface is said to be {\em minimal} if all its vertical leaves are dense. We'll see in Section~\ref{first-return} that a non-critical leaf which is dense in one direction must be dense in both directions. Thus, on a minimal translation surface, every leaf that's not forward-critical is dense in the forward direction, and every leaf that's not backward-critical is dense in the backward direction. A translation surface with no saddle connections is automatically minimal~\cite[proof of Theorem~1.8]{rat-flat}.

Away from the singularities, the topology of a translation surface has a basis consisting of {\em flow boxes}: open rectangles with vertical and horizontal sides. A compact translation surface can be covered by a finite collection of flow boxes and singularity charts. The special class of {\em well-cut} flow boxes, defined in Section~\ref{first-return}, will play an important role in this paper.

Because translations form a normal subgroup of $\on{Aff} \R^2$, you can modify a translation structure by composing an element of $\on{GL}_2 \R$ with all its charts. In particular, you can rotate a translation structure, tilting the vertical foliation. If you rotate a translation structure through a full circle, only countably many of the structures you pass through will have saddle connections~\cite[proof of Theorem~1.8]{rat-flat}. That means you can get rid of any saddle connections with an arbitrarily small rotation.
\subsubsection{First return maps}\label{first-return}
Let's say a {\em horizontal segment} on a translation surface is a subset that looks like an open, closed, or half-open horizontal line segment in some chart. Pick a point on a horizontal segment, and watch it as it's carried upward by the vertical flow. On a compact translation surface, unless it falls into a singularity, the point will eventually return to the segment it started on. This is an immediate corollary of \cite[Lemma~1.7]{rat-flat}, which for reference I'll restate here.
\begin{lemma}\label{return-lemma}
Let $Z$ be a closed horizontal segment on a compact translation surface, and let $p$ be one of its endpoints. Unless the vertical leaf through $p$ is forward-critical, the vertical flow will eventually carry $p$ back to $Z$.
\end{lemma}
The backward vertical flow, of course, has the same return property. As a consequence, the forward and backward orbits of a point on a non-critical leaf pass through the same collection of horizontal segments.

On any horizontal segment $Z$ in a compact translation surface, we can define a relation $\alpha$ that sends each point to the place where it first returns to $Z$ under the vertical flow. When fed a point that falls into a singularity before returning, $\alpha$ gives back nothing. We'll call $\alpha$ the {\em first return relation} on $Z$. The inverse relation $\alpha^{-1}$ sends each point to the place where it first returns to $Z$ under the backward vertical flow.

When restricted to the complement of the forward-critical leaves, $\alpha$ becomes a function, and is called the {\em first return map}. Similarly, $\alpha^{-1}$ becomes a function when restricted to the complement of the backward-critical leaves. On the complement of all the critical leaves, $\alpha$ and $\alpha^{-1}$ are inverse functions.

Because the vertical foliation only has a few kinds of local geometry, the first return relation only has a few kinds of local behavior. Under mild conditions, it belongs to the class of transformations called {\em interval exchanges}, which we'll hear more about in Section~\ref{divd-ex}~\cite[\S 3]{billiard-flows}. The conditions are the price we pay for defining first return relations, horizontal segments, and interval exchanges in a slightly non-standard way, whose advantages will become apparent in Section~\ref{divd-first-return}.

In our framework, the first return relation is only an interval exchange on a {\em well-cut} segment, which is a horizontal segment with the following properties:
\begin{itemize}
\item It looks like an open horizontal line segment bounded by critical leaves. (Both endpoints may lie on the same critical leaf.)
\item The forward vertical flow drops each forward-critical boundary point into a singularity without carrying it through the segment. Similarly, the backward vertical flow drops each backward-critical boundary point into a singularity without carrying it through the segment.
\end{itemize}
When we construct the forward and backward first return relations on a horizontal segment $Z$, this condition prevents the vertical flow $\psi^t$ from breaking $\psi^t Z$ across the boundaries of $Z$ in a way inconsistent with our definition of an interval exchange. We'll call a flow box {\em well-cut} if every horizontal slice across it is a well-cut segment.

When critical leaves are plentiful, well-cut segments are easy to find.
\begin{prop}\label{well-cut-pieces}
If $Z$ is an open horizontal segment bounded by critical leaves, every non-critical point on $Z$ is contained in a well-cut subsegment of $Z$.
\end{prop}
\begin{proof}
Carry each forward-critical boundary point of $Z$ along the forward vertical flow, marking each place it passes through $Z$ before falling into a singularity. Do the same with each backward-critical boundary point, using the backward vertical flow. Removing the marked points breaks $Z$ into a finite collection of open horizontal segments, which are all well-cut.
\end{proof}
\subsubsection{Half-translation surfaces}
A half-translation surface is the same thing as a translation surface, but with transition maps composed of both translations and half-turn rotations. This expansion of the structure group is pretty small, so half-translation surfaces have almost as much structure as translation surfaces do:
\begin{itemize}
\item The flat metric remains.
\item The vertical and horizontal foliations remain, but they're no longer canonically oriented. In fact, it's often not possible to give either foliation a consistent global orientation.
\item As a result, it's often not possible to define a global vertical flow.
\end{itemize}
The group generated by translations and half-turns has the translation group as a normal subgroup, so each of its elements is either a translation or a translation followed by a half-turn. I'll call the latter a {\em flip}.

In a translation surface, we saw that the angle around a conical singularity could be any even multiple of $\pi$. In a half-translation surface, any multiple of $\pi$ is possible. Here's a pattern for a singularity with cone angle $3\pi$:
\begin{center}
\begin{tikzpicture}
\matrix[row sep=1cm, column sep=1.2cm]{
& \notchdown{pworange}{$2$}{pwpink}{$1$} \\
\notchup{pworange}{$2$}{pwgold}{$3$} & \notchup{pwgold}{$3$}{pwpink}{1} \\
};
\end{tikzpicture}
\end{center}
The triangles labeled $3$ are identified by a flip. A singularity with cone angle $\pi$ can be constructed, informally, by using a flip to glue one of the notched square building blocks we've been using to itself:
\begin{center}
\begin{tikzpicture}
\notchup{pwpink}{$1$}{pwpink}{$1$}
\end{tikzpicture}
\end{center}

The vertical and horizontal foliations can't be oriented consistently around an odd singularity, but they can be oriented on a small enough neighborhood of any other point. In particular, on a half-translation surface, the definition of a flow box still makes sense, and the vertical and horizontal foliations can be oriented inside a flow box.

Every half-translation surface comes with a translation surface hovering over it as a branched double cover, with a projection map that preserves the half-translation structure. Away from the singularities, this {\em translation double cover} is built by making two copies of each flow box on the half-translation surface, one for each possible orientation of the vertical foliation. The transition maps of the half-translation surface induce transitions between the oriented copies in a natural way.

The translation double cover can be extended over the whole surface by duplicating the cut square pieces that make up the region around each singularity. In the picture below, duplicate pieces are labeled with the same numbers. One piece from each pair is drawn in dark ink, and the other in light ink.
\begin{center}
\begin{tikzpicture}
\matrix[row sep=1cm, column sep=1.2cm]{
\dupnotchdown{pwviolet}{$\bar{3}$}{pwblue}{$\bar{2}$} & \dupnotchdown{pwgreen}{$\bar{1}$}{pwgold}{$3$} & \notchdown{pworange}{$2$}{pwpink}{$1$} \\
\dupnotchup{pwgreen}{$\bar{1}$}{pwblue}{$\bar{2}$} & \notchup{pworange}{$2$}{pwgold}{$3$} & \notchup{pwviolet}{$\bar{3}$}{pwpink}{$1$} \\
};
\end{tikzpicture}
\end{center}
Around a singularity whose cone angle is an odd multiple of $\pi$, like the one shown here, the duplicated pieces fit together into a single connected region, whose covering map to the original region is branched over the singularity. The branch point is a conical singularity with twice the angle of the original. Around a singularity whose cone angle is an even multiple of $\pi$, the duplicated pieces form two disconnected copies of the original region.
\subsection{Dividing intervals}\label{divd-intvl}
\subsubsection{Construction of divided and fractured intervals}\label{divd-intvl-construct}
Dividing a translation surface is essentially a one-dimensional process, so let's start in the one-dimensional case. Let $I \subset \R$ be an open interval, and let $W$ be a subset of $I$ which is countable or smaller. To {\em divide} $I$ at $W$, first build a new set
\[ \divd{I} = \{\llane{w}, \median{w}, \rlane{w}\}_{w \in W} \sqcup \{s\}_{s \in I \smallsetminus W}. \]
There's an obvious map $\pi \maps \divd{I} \to I$ which sends $\llane{w}$, $\median{w}$, and $\rlane{w}$ to $w$ and each point in $I \smallsetminus W$ to itself. Order $\divd{I}$ so that $\pi$ is order-preserving and $\llane{w} < \median{w} < \rlane{w}$. Give $\divd{I}$ the topology generated by all the non-empty intervals $(a, b)$ except the ones that look like $(\llane{w}, b)$ or $(a, \rlane{w})$. This topology is coarser than the order topology, and non-Hausdorff: neither $\llane{w}$ nor $\rlane{w}$ can be separated from $\median{w}$ by open sets. The generating intervals described above form a basis for the topology, so I'll refer to them as ``basis intervals.''

As you might expect, $\pi$ turns out to be a quotient map. Going the other direction, let $\iota \maps I \to \divd{I}$ be the map that sends $w \in W$ to $\hat{w}$ and each point in $I \smallsetminus W$ to itself. Perhaps surprisingly, $\iota$ turns out to be an embedding. Both of these claims will be proven in Section~\ref{divd-props-1d}.

Define the {\em fractured interval} $\frakd{I}$ to be the complement of $\iota W$ in the divided interval $\divd{I}$. Although removing $\iota W$ changes the topology of $\divd{I}$ considerably, the restriction of $\pi$ to $\frakd{I}$ is still a quotient map. We'll prove this in Section~\ref{frakd-props-1d}.
\subsubsection{Construction of divided and fractured $1$-manifolds}
If you replace all the linear orders with cyclic ones, the instructions in Section~\ref{divd-intvl-construct} for dividing an open interval apply just as well to an oriented circle. Thus, component by component, we can divide any oriented $1$-manifold.

For simplicity, we'll take $I$ to be an open interval throughout Section~\ref{divd-intvl}, because that's the only case we need later in the paper. However, the results of Section~\ref{divd-intvl} should hold for any oriented $1$-manifold $I$.

Working on a circle will make the examples of Section~\ref{divd-ex} easier to describe. It will also give us a taste of the local reasoning we'll use to study divided translation surfaces.
\subsubsection{Examples from dynamics}\label{divd-ex}
Divided intervals arise naturally in dynamics, as a way of coding the trajectories of one-dimensional systems with ``break points.'' We'll see how this works through two examples. The first, a divided version of the doubling map, is meant to demonstrate the basic ideas in a simple setting. The second, a divided version of any interval exchange, will play a central role in this paper.

In both examples, we'll describe the dynamics using a partial map; you'll probably be able to guess from context what that means. To be precise, it means a coinjective, bicontinuous relation, in the terminology of Appendix~\ref{rel-dyne}.
\paragraph{The doubling map}
Let $\circgp$ be the the circle $\R/\Z$---the space of fractional parts of numbers. The {\em doubling map} is the partial map from $\circgp$ to itself that doubles each fractional part, returning nothing if the result is an integer. Geometrically, the doubling map breaks $\circgp$ into two pieces, $(0, \tfrac{1}{2})$ and $(\tfrac{1}{2}, 1)$, and stretches each piece evenly over the whole interval $(0, 1)$. This point of view is the motivation for returning nothing at half-integers. From yet another point of view, the doubling map removes the first digit of each fractional part's binary expansion, returning nothing if the result is $.00000\ldots$ or $.11111\ldots\;$.

Applying the doubling map over and over, let $W \subset \circgp$ be the set of points that eventually ``fall into a break,'' reaching a point where the map returns nothing. This turns out to be the set of rationals whose denominators are powers of two.

Divide $\circgp$ at $W$. Observe that $\frakd{\circgp}$ can be expressed as the interval $[\rlane{0}, \llane{1}]$. If you've read ahead to Section~\ref{frakd-props-1d}, you can immediately see from Corollary~\ref{frakd-cantor} that $\frakd{\circgp}$ is a Cantor set. We can see the same thing directly by identifying $\frakd{\circgp}$ with the space of one-sided binary sequences. A fractional part $s \in \circgp \smallsetminus W$ has a unique binary expansion, $\iota s$. A fractional part $w \in W$ has two binary expansions: $\llane{w}$, the one ending in ones, and $\rlane{w}$, the one ending in zeros. The quotient map $\pi \maps \frakd{\circgp} \to \circgp$ interprets each binary sequence as a fractional part.

As I hinted earlier, the doubling map lifts naturally to a map on the space of binary sequences---that is, to a map on $\frakd{\circgp}$. It then extends to a partial map on $\divd{\circgp}$, which acts on $\iota \circgp$ as usual and on $\frakd{\circgp}$ by shifting binary sequences. This is the only continuous partial map on $\divd{\circgp}$ which matches the doubling map on $\iota \circgp$ and returns something at every point in $\frakd{\circgp}$.
\paragraph{Interval exchanges}
An {\em interval exchange transformation} is a partial map from $\circgp$ to itself that works by splitting $\circgp$ into finitely many open intervals and shuffling them around:
\begin{center}
\begin{tikzpicture}[scale=5]
\exampleiet
\end{tikzpicture}
\end{center}
On the break points between the intervals, the map returns nothing. For convenience, let's assume $0$ is always a break point.

An interval exchange, unlike the doubling map, is injective, so its inverse relation is also a partial map. In fact, its inverse is another interval exchange, which unshuffles the pieces of $\circgp$.

Pick an interval exchange $\alpha$, and let $W \subset \circgp$ be the set of points that eventually fall into a break under iteration of $\alpha$ or its inverse. Just as before, divide $\circgp$ at $W$, and observe that $\frakd{\circgp}$ can be expressed as $[\rlane{0}, \llane{1}]$.

Suppose $W$ is dense in $\circgp$. Then, like before, Corollary~\ref{frakd-cantor} in Section~\ref{frakd-props-1d} will tell us that $\frakd{\circgp}$ is a Cantor set. We can see this directly by embedding $\frakd{\circgp}$ in a space of two-sided sequences. Let $\mathcal{A}$ be the intersections of the intervals exchanged by $\alpha$ and the ones exchanged by $\alpha^{-1}$. The orbit of a point $s \in \circgp \smallsetminus W$ under iteration of $\alpha$ and its inverse is infinite in both directions, so we can get a two-sided sequence $\iota s$ by keeping track of which intervals the orbit passes through.

The orbit of a point $w \in W$ ends when it falls into a break, but it can be continued in two natural ways. One is to extend the partial map $\alpha$ to a left-continuous map, so each point travels with the points to the left of it, and the intervals being shuffled become closed on the right. The orbit of $w$ becomes infinite in both directions, and keeping track of which intervals it goes through yields a two-sided sequence $\llane{w}$. The other way to continue the orbit of $w$ is to extend $\alpha$ to a right-continuous map, giving a different two-sided sequence $\rlane{w}$.

The sequence space $\mathcal{A}^\Z$ comes with a natural dynamical map: the {\em shift map}, which acts on a sequence by moving each letter one step earlier. Its action on the embedded copy of $\frakd{\circgp}$ is a lift of $\alpha$, so it extends $\alpha$ to a partial map on $\divd{\circgp}$. Like before, this extension is the only continuous partial map on $\divd{\circgp}$ which matches $\alpha$ on $\iota \circgp$ and returns something at every point in $\frakd{\circgp}$.

Our divided circle $\divd{\circgp}$ isn't the only one you can build from an interval exchange. In \cite{cantor-min-sys}, for instance, Gjerde and Johansen divide the circle only at the points that fall into a break under forward iteration of $\alpha$. They construct the fractured circle, lift $\alpha$ to it, and find an interesting model for the lifted dynamics.
\subsubsection{Properties of divided intervals}\label{divd-props-1d}
With those examples in mind, let's go back to studying a general interval $I$ divided at $W$. For convenience, let $\llane{\iota} \maps I \to \divd{I}$ be the map that sends $w \in W$ to $\llane{w}$ and each point in $I \smallsetminus W$ to itself. Define $\rlane{\iota}$ similarly.

A basis interval has a leftmost element if and only if it looks like $(\median{w}, b) = [\rlane{w}, b)$, and a rightmost element if and only if it looks like $(a, \median{w}) = (a, \llane{w}]$. It will often be useful to {\em trim} a basis interval by removing its leftmost and rightmost elements, if they exist. The trimmed version of an interval $(a, b) \subset \divd{I}$, denoted $\trim (a, b)$, can be written explicitly as $(\rlane{\iota}\pi a, \llane{\iota}\pi b)$. Notice that $\pi \trim (a, b) = \iota^{-1}(a, b) = (\pi a, \pi b)$ for any basis interval $(a, b)$, and that trimming a basis interval does not remove any points in the image of $\iota$. Conveniently, for any basis interval, $\pi^{-1} \iota^{-1} (a, b) = \trim (a, b)$.

With these tools in hand, let's prove the claims about $\pi$ and $\iota$ made in the previous section.
\begin{proof}[Proof that $\pi$ is a quotient map]
To see that $\pi$ is continuous, observe that the preimage of $(a, b) \subset I$ under $\pi$ is the basis interval $(\rlane{\iota}a, \llane{\iota}b)$.

To see that $\pi$ is a quotient map, pick any $S \subset I$ whose preimage under $\pi$ is open. We want to show $S$ is open. For any $s \in S$, the point $\iota s$ is in $\pi^{-1} S$, so there is a basis interval $H \subset \pi^{-1} S$ containing $\iota s$. Since $\trim H$ also contains $\iota s$, and $\pi$ sends trimmed basis intervals to open intervals, $\pi \trim H$ is an open subset of $S$ containing $s$.
\end{proof}
\begin{proof}[Proof that $\iota$ is an embedding]
To see that $\iota$ is continuous, recall that $\iota^{-1}(a, b) = (\pi a, \pi b)$ for any basis interval $(a, b)$.

To see that $\iota$ is an embedding, observe that the image under $\iota$ of an interval $(a, b) \subset I$ is the intersection of $(\iota a, \iota b)$ with $\iota I$.
\end{proof}

The continuity of $\iota$ is a way of saying that passing from $I$ to $\divd{I}$ spreads out the points of $W$, but it doesn't spread them out too much. Here are two more reflections of this idea.
\begin{prop}\label{undivd-dense-1d}
The embedding of $I$ in $\divd{I}$ is dense.
\end{prop}
\begin{proof}
It's enough to show that $\iota I$ intersects every basis interval. Suppose the basis interval $(a, b)$ doesn't intersect $\iota I$, so its preimage $(\pi a, \pi b)$ under $\iota$ is empty. Since $I$ is densely ordered, this means $\pi a = \pi b$, which is precluded by the rules defining basis intervals.
\end{proof}
\begin{prop}\label{divd-loc-conn-1d}
The divided interval $\divd{I}$ is locally connected.
\end{prop}
\begin{proof}
It's enough to show that every basis interval is connected. Recall that basis intervals are non-empty by definition. Let's say the basis interval $(a, b)$ is disconnected by two open subsets $U$ and $V$. Since $\iota I$ is dense in $\divd{I}$, the preimages of $U$ and $V$ under $\iota$ are non-empty, so they disconnect the preimage $(\pi a, \pi b)$ of $(a, b)$.
\end{proof}

For our purposes, the most important feature of $\divd{I}$ is that its local systems are naturally in correspondence with the local systems on $I$. This idea can be stated more precisely as follows.
\begin{thm}\label{loc-sys-equiv-1d}
For any group $G$, the direct image functors $\pi_*$ and $\iota_*$ give an equivalence between the groupoid of $G$ local systems on $\divd{I}$ and the groupoid of $G$ local systems on $I$.\footnote{Though they're stated for groupoids of local systems, Theorem~\ref{loc-sys-equiv-1d} and Lemma~\ref{trim-iso} hold for categories of locally constant sheaves into any fixed target category. The proofs are the same, keeping in mind our convention (from Section~\ref{loc-sys-conventions}) that the target category is a type of algebraic structure.}
\end{thm}
This identifies the $G$ character stacks of $\divd{I}$ and $I$.

The reason $\divd{I}$ has no more local systems than $I$, despite having more open subsets, is that a local system on $\divd{I}$ is determined entirely by its values on trimmed intervals.
\begin{lemma}\label{trim-iso}
If $\mathcal{F}$ is a local system on $\divd{I}$, the restriction $\mathcal{F}_{\trim H \subset H}$ is an isomorphism for any basis interval $H$.
\end{lemma}
\begin{proof}
If $H$ has neither a least element nor a greatest element, $\trim H = H$, so there's nothing to prove. Let's assume $H$ has a least element, but no greatest element; the remaining cases are essentially the same.

Since $\mathcal{F}$ is locally constant, we can pick a basis interval $A \subset H$ which contains the least element of $H$ and is small enough that $\left.\mathcal{F}\right|_A$ is constant. Since $A$ and $\trim A$ are both connected, $\mathcal{F}_{\trim A \subset A}$ is an isomorphism. The diagram
\[ \xymatrix{
& \ar[dl]_{\subset^{-1}} \ar[d]^= \mathcal{F}_{\trim A} & \ar[l]_\subset \ar[d]^= \mathcal{F}_{\trim H} \\
\mathcal{F}_A \ar[r]_\subset & \mathcal{F}_{\trim A} & \ar[l]^\subset \mathcal{F}_{\trim H}
} \]
commutes, so taking limits of its top and bottom rows gives a map $\mathcal{F}_{\trim H} \to \mathcal{F}_H$, which inverts $\mathcal{F}_{\trim H \subset H}$.
\end{proof}
\begin{proof}[Proof of Theorem~\ref{loc-sys-equiv-1d}]
There's a canonical natural isomorphism between $\pi_* \iota_*$ and the identity functor, because $\pi \iota$ is the identity map from $I$ to itself. Now, all we need is a natural isomorphism between $\iota_* \pi_*$ and the identity.

Pick any local system $\mathcal{F}$ on $\divd{I}$. For each basis interval $H$, recall that $\pi^{-1} \iota^{-1} H = \trim H$, so
\begin{align*}
(\iota_* \pi_* \mathcal{F})_H & = \mathcal{F}_{\pi^{-1} \iota^{-1} H} \\
& = \mathcal{F}_{\trim H}.
\end{align*}
Thus, the restriction $\mathcal{F}_{\trim H \subset H}$ gives a morphism from $\mathcal{F}_H$ to $(\iota_* \pi_* \mathcal{F})_H$, and Lemma~\ref{trim-iso} tells us this morphism is an isomorphism. For any basis interval $H' \subset H$, the diagram
\[ \xymatrix{
\mathcal{F}_{H'} \ar[d]_\subset & \ar[l]_\subset \ar[d]^\subset \mathcal{F}_H \\
\mathcal{F}_{\trim H'} & \ar[l]^\subset \mathcal{F}_{\trim H}
} \]
commutes because all the arrows are restrictions, so we've found a natural isomorphism from $\mathcal{F}$ to $\iota_* \pi _* \mathcal{F}$.
\end{proof}
\subsubsection{Properties of fractured intervals}\label{frakd-props-1d}
Let's start our discussion of fractured intervals with some tools for passing back and forth between $\frakd{I}$ and $\divd{I}$. Define $\nom{R} = R \cap \frakd{I}$ for any $R \subset \divd{I}$. We can {\em seal} the gaps in an open set $V \subset \frakd{I}$, yielding an open set $\seal V \subset \divd{I}$, by taking the union of all open $U \subset \divd{I}$ for which $\nom{U} = V$. Here's a more explicit construction of $\seal V$. To include $\llane{w}$ and $\rlane{w}$, an open set $U \subset \divd{I}$ must contain basis intervals $(a, \median{w})$ and $(\median{w}, b)$, implying that $U \cup \{\median{w}\}$ is open. So, if $V$ contains $\llane{w}$ and $\rlane{w}$, its sealed version contains $\median{w}$. The converse also holds, because every basis interval with $\median{w}$ in it contains $\llane{w}$ and $\rlane{w}$. Sealing an open set $V \subset \frakd{I}$ thus amounts to adding each $\median{w}$ for which $V$ contains $\llane{w}$ and $\rlane{w}$.

We can use these tools to prove our last claim about $\pi$ from Section~\ref{divd-intvl-construct}.
\begin{proof}[Proof that the restriction of $\pi$ to $\frakd{I}$ is a quotient map.]
Let $R = \pi^{-1}S$ for some $S \subset I$, and suppose $\nom{R}$ is open in $\frakd{I}$. Notice that $R$ contains $\median{w}$ if and only if $\nom{R}$ contains $\llane{w}$ and $\rlane{w}$. That means $R = \seal \nom{R}$, by our explicit construction above. In particular, $R$ is open. That means $S$ is open, because $\pi$ is a quotient map.
\end{proof}
One nice feature of $\frakd{I}$ is that, with $\median{w}$ out of the way, the points $\llane{w}$ and $\rlane{w}$ can be separated by open sets. The consequence is just what you'd expect.
\begin{prop}\label{frakd-hausdorff-1d}
The fractured interval $\frakd{I}$ is Hausdorff.
\end{prop}
\begin{proof}
Pick two points $s < t$ in $\frakd{I}$. If $\pi s \neq \pi t$, we can find disjoint neighborhoods of $\pi s$ and $\pi t$ in $I$ and pull them back to $\frakd{I}$. If $\pi s = \pi t$, then $s = \llane{w}$ and $t = \rlane{w}$ for some $w \in W$. Hence, $(-\infty, \llane{w}]$ and $[\rlane{w}, \infty)$ are disjoint neighborhoods of $s$ and $t$.
\end{proof}

When studying $\divd{I}$, we found it useful to work with basis intervals whose leftmost and rightmost elements had been removed. For studying $\frakd{I}$, it will be useful to go the opposite direction. Let's say a basis interval is {\em full} if it has both a leftmost element and a rightmost element. As we saw earlier, the full intervals in $\divd{I}$ are the ones that look like $[\rlane{a}, \llane{b}]$. The full intervals in $\frakd{I}$ are the same.
\begin{prop}\label{frakd-cpt-basis-1d}
In $\frakd{I}$, every full interval is compact.
\end{prop}
\begin{proof}
Consider a full interval $[\rlane{a}, \llane{b}]$. Let $W'$ be the subset of $W$ lying between $\rlane{a}$ and $\llane{b}$. Pick a function $\kappa \maps W' \to \R_+$ for which the sum $K = \sum_{w \in W'} \kappa w$ is finite, and let $\theta$ be the map from $[\rlane{a}, \llane{b}] \cap \frakd{I}$ to $\R$ given by the formula
\[ \theta s = \pi s + \sum_{\substack{w \in W' \\ \hat{w} < s}} \kappa w. \]
It's not hard to see that $\theta$ is a homeomorphism whose image is the set
\[ [a, b + K] \smallsetminus \bigcup_{w \in W'} (\theta \llane{w}, \theta \rlane{w}), \]
which is closed and bounded.
\end{proof}
\begin{prop}
In $\frakd{I}$, every full interval is clopen.
\end{prop}
\begin{proof}
In $\frakd{I}$, the complement of a full interval $[\rlane{a}, \llane{b}]$ is the union of the basis intervals $(-\infty, \llane{a}]$ and $[\rlane{b}, \infty)$.
\end{proof}
\begin{prop}\label{full-basis}
If $W$ is dense in $I$, the full intervals form a basis for $\frakd{I}$.
\end{prop}
\begin{proof}
Suppose $W$ is dense in $I$. Pick any point $s \in \frakd{I}$ and any basis interval $(a, b)$ containing it. If $\pi a \neq \pi s$, find a point of $W$ in the interval $(\pi a, \pi s)$ and call it $\alpha$. If $\pi a = \pi s$, observe that $a = \median{\alpha}$ and $s = \rlane{\alpha}$ for some $\alpha \in W$. One way or another, we've found a point $\alpha \in W$ with $a < \rlane{\alpha} \le s$. Using the same technique, we can find a point $\beta \in W$ with $s \le \llane{\beta} < b$. The full interval $[\rlane{\alpha}, \llane{\beta}]$ is a neighborhood of $s$ contained in $(a, b)$.
\end{proof}
\begin{cor}\label{frakd-cantor}
If $W$ is dense in $I$, every full interval in $\frakd{I}$ is a Cantor set.
\end{cor}
\begin{proof}
Suppose $W$ is dense in $I$. Because we require $W$ to be countable or smaller, the results above imply that $\frakd{I}$ is a Hausdorff space with a countable basis of clopen sets. Any full interval $H \subset \frakd{I}$ has the same properties, and in addition is compact. Therefore, $H$ is a Cantor set as long as it has no isolated points~\cite[Theorem~3]{set-theo-top}.

Intersecting $H$ with a basis interval in $\frakd{I}$ yields another basis interval. Since every basis interval contains more than one point, it follows that $H$ has no isolated points.
\end{proof}
\subsubsection{Metrization}
The topology of $I$ is induced by the metric that $I$ inherits from $\R$. If $W$ is dense in $I$, the topology of $\frakd{I}$ can be metrized too, and there's a simple way to do it. For the examples in Section~\ref{divd-ex}, the resulting metric is dynamically meaningful, as we'll see in Section~\ref{division-dynamics}.

For the rest of this section, suppose $W$ is dense in $I$. Let's say we've assigned each point in $W$ a natural number, its {\em grade}, and there are only finitely many points of each grade. Since $W$ is countable or smaller, this is always possible.

For the doubling map, the points in $W$ are the fractional parts whose binary expansions are eventually constant, and they're naturally graded by the position of the last digit before the constant tail. (This grading doesn't work at zero, which has no digits before the constant tail, but we won't need it there.) For an interval exchange, the points in $W$ are the points that will eventually fall into a break under forward or backward iteration, and they're naturally graded by how long it takes for that to happen. We'll fix a normalization by declaring the break points to have grade zero.

Pick a real number $\Delta > 1$, and define the {\em height} of a point in $w \in W$ to be $\Delta^{-\gr{w}}$. In $\divd{I}$, define the height of $\median{w} \in \iota W$ to be the height of $w$, and the height of any other point to be zero. The heights of some points in the divided circle for the doubling map are illustrated below.
\begin{center}
\begin{tikzpicture}[scale=9]
\landscape
\foreach \level / \size in {0 / \normalsize, 1 / \small, 2 / \tiny} {
	\pgfmathsetmacro{\altitude}{\mountainheight/pow(3,\level)}
	\draw[font=\size] (1+4/81,\altitude) node {$\Delta^{-\pgfmathprintnumber[fixed,precision=0]{\level}}$};
}
\begin{scope}[white]
\draw (1/3,-1/24) node {$\llane{\frac{1}{2}}$};
\draw (1/2,-1/24) node {$\median{\frac{1}{2}}$};
\draw (2/3,-1/24) node {$\rlane{\frac{1}{2}}$};
\begin{scope}[font=\small]
\draw (1/9,-1/24) node {$\llane{\frac{1}{4}}$};
\draw (3/18,-1/24) node {$\median{\frac{1}{4}}$};
\draw (2/9,-1/24) node {$\rlane{\frac{1}{4}}$};
\draw (7/9,-1/24) node {$\llane{\frac{3}{4}}$};
\draw (15/18,-1/24) node {$\frac{3}{4}$};
\draw (8/9,-1/24) node {$\rlane{\frac{3}{4}}$};
\end{scope}
\end{scope}
\end{tikzpicture}
\end{center}
Let's say the distance between two points $a, b \in \frakd{I}$ is the height of the highest point in $(a, b) \subset \divd{I}$. This defines a metric (in fact, an ultrametric) on $\frakd{I}$, which I'll call the {\em division metric} with steepness $\Delta$. The assumption that $W$ is dense in $I$ is essential here: it guarantees that distances between distinct points are positive.
\begin{prop}\label{metric-metrizes}
The division metric induces the topology of $\frakd{I}$.
\end{prop}
\begin{proof}
Let's see what the open balls of the division metric look like. Given a radius $r > 0$, let $W_{\ge r}$ be the set of points in $W$ with heights greater than or equal to $r$. This set is finite, because there are only finitely many points of each grade. Listing the points in $W_{\ge r}$ from left to right as $w_1, \ldots, w_n$, we can write down all the open balls of radius $r$:
\[ (-\infty, \median{w}_1), (\median{w}_1, \median{w}_2), \ldots, (\median{w}_{n-1}, \median{w}_n), (\median{w}_n, \infty). \]

From this description, it's clear that the open balls of the division metric are open subsets of $\frakd{I}$. It's also clear that every full interval is a union of open balls, because every full interval can be written as $(\median{a}, \median{b})$ for $a, b \in W$. By Proposition~\ref{full-basis}, the full intervals form a basis for $\frakd{I}$, so we're done.
\end{proof}
\subsubsection{Dynamical significance of the division metric}\label{division-dynamics}
In Section~\ref{divd-ex}, we embedded our fractured circle for the doubling map into the one-sided sequence space $\mathbf{2}^\N$. A choice of $\Delta > 1$ picks out a metric on $\mathbf{2}^\N$, defined as follows. Given two distinct sequences $x, y \in \mathbf{2}^\N$, let $m \in \N$ be the first position where they differ. The distance between $x$ and $y$ is $\Delta^{-m}$. This metric restricts to the division metric of steepness $\Delta$ on the divided circle, so let's call it by the same name. Then we can say that the fractured circle for the doubling map embeds isometrically into $\mathbf{2}^\N$, as long as we use the division metric of the same steepness on both sides.

We embedded our fractured interval exchange into a two-sided sequence space $\mathcal{A}^\Z$, which also comes with a division metric of steepness $\Delta$. Define the distance between two distinct sequences in $\mathcal{A}^\Z$ to be $\Delta^{-m}$, where this time $m$ is the absolute value of the first position where the sequences differ. Once again, the fractured interval exchange embeds isometrically into $\mathcal{A}^\Z$, as long as we use the division metric of the same steepness on both sides.

In both of our dynamical examples, the division metric tells you how long two points in the fractured circle travel together before they end up on opposite sides of a break. Nearby points move together for a long time, while the most distant points are separated immediately. A single step of the dynamics can take a pair of points at most one step closer to being separated, increasing the distance between then by at most a factor of $\Delta$. That means the dynamical map is Lipschitz with respect to the division metric.
\subsubsection{Lipschitz functions under the division metric}\label{lipschitz-functions}
Aside from the contrast between irreversible and reversible dynamics, our two dynamical examples from Section~\ref{divd-ex} have looked pretty similar so far. Here's a more subtle difference between them. For the doubling map, there are $2^{m+1} - 1$ points in $W$ with grade $m$ or less, because each nonzero point has two preimages. For an interval exchange, with break points $B \subset W$, there are at most $(m + 1)|B|$ points of grade $m$ or less, because each point has at most one preimage.\footnote{The break points of $\alpha^{-1}$ have no preimages. This only affects the number of points with grade $m$ or less if $\alpha$ is non-minimal~\cite[proof of Theorem~1.8]{rat-flat}, and in any case it won't bother us.}

Using the division metric, you can detect this difference in an interesting way. When the number of points with grade $m$ or less grows slowly with $m$, Lipschitz functions on $\frakd{I}$ are much more rigid than continuous ones. How slowly depends on the steepness of the division metric, but subexponential growth is slow enough at any steepness. Hence, Lipschitz functions on a fractured interval exchange are always rigid. In Section~\ref{ab-deliv}, their rigidity will ensure that our abelianization process does what it's supposed to do.

To see what kind of rigidity I mean, first consider a continuous function $f$ from $\frakd{I}$ into some metric space. Heuristically, $f$ has two sources of flexibility. One is the freedom to jump at the crack between two adjacent edge points $\llane{w}$ and $\rlane{w}$, for any $w \in W$. We can remove this ``flexibility at the edges'' by requiring the values of $f$ to match at adjacent edge points, in the sense that $f\llane{w} = f\rlane{w}$ for all $w \in W$. Topologically, we're requiring $f$ to factor through the quotient map $\pi \maps \frakd{I} \to I$. This reveals that $f$ still has some ``flexibility in the bulk''---the flexibility of a continuous function on $I$. Rigidity, as I've been using it here, means having no flexibility in the bulk.

To say what growing slowly means, it will be useful to think in terms of heights instead of grades. As we did in the proof of Proposition~\ref{metric-metrizes}, let $W_{\ge r}$ be the set of points in $W$ with heights greater than or equal to $r$.
\begin{thm}\label{no-bulk-flex}
As $r$ approaches zero, suppose $W_{\ge r}$ grows slowly enough that $r\left|W_{\ge r}\right|$ still approaches zero. Consider a Lipschitz function $f$ from $\frakd{I}$ into some metric space. If its values match at adjacent edge points, in the sense that $f\llane{w} = f\rlane{w}$ for all $w \in W$, then it's constant.
\end{thm}
\begin{proof}
Consider any two points $a, b \in \frakd{I}$. Pick a radius $r > 0$, and list the points in $W_{\ge r} \cap (a, b)$ from left to right as $w_1, \ldots, w_n$. By the triangle inequality,
\begin{align*}
d(fa, fb) & \le d(fa, f\llane{w}_1) + d(f\llane{w}_1, f\rlane{w}_1) + d(f\rlane{w}_1, f\llane{w}_2) + \ldots \\
& \qquad + d(f\rlane{w}_{n-1}, f\llane{w}_n) + d(f\llane{w}_n, f\rlane{w}_n) + d(f\rlane{w}_n, fb).
\end{align*}
Suppose the values of $f$ match at adjacent endpoints. That means the $d(f\llane{w}_k, f\rlane{w}_k)$ terms vanish, so we're left with the bound
\[ d(fa, fb) \le d(fa, f\llane{w}_1) + d(f\rlane{w}_1, f\llane{w}_2) + \ldots + d(f\rlane{w}_{n-1}, f\llane{w}_n) + d(f\rlane{w}_n, fb). \]
Since $f$ is Lipschitz, we can deduce that
\[ d(fa, fb) \le M \big[ d(a, \llane{w}_1) + d(\rlane{w}_1, \llane{w}_2) + \ldots + d(\rlane{w}_{n-1}, \llane{w}_n) + d(\rlane{w}_n, b) \big], \]
where $M$ is a Lipschitz constant for $f$. By definition, $w_1, \ldots, w_n$ are the only points of height $r$ or higher in $(a, b)$. That means the distances on the right-hand side are all less than $r$, so
\begin{align*}
d(fa, fb) & \le M(n+1)r \\
& \le M(\left|W_{\ge r}\right| + 1)r.
\end{align*}

The bound above holds for any radius $r$. As $r$ approaches zero, the right-hand side approaches zero by hypothesis, so $d(fa, fb)$ must be zero. Since $a, b \in \frakd{I}$ could be any points, that means $f$ is constant.
\end{proof}
\subsection{Dividing translation surfaces}
\subsubsection{Construction of divided and fractured surfaces}\label{divd-surf-construct}
Now, let's move up to the two-dimensional case. We'll use the notation from Section~\ref{running-notation} for the singularities and critical leaves of $\Sigma$. Cover $\Sigma \smallsetminus \mathfrak{B}$ with flow boxes. Each flow box can be identified with a rectangle $I \times L \subset \R^2$, where $I$ and $L$ are open intervals in $\R$. The critical leaves of $\Sigma$ intersect the flow box as vertical lines $\{w\} \times L$. There are only finitely many critical leaves, and each one passes through the flow box at most countably many times. Dividing $I$ at the positions of the critical leaves, we can produce a {\em divided flow box} $\divd{I} \times L$. In the divided flow box, each critical leaf $\{w\} \times L$ splits into a {\em left lane} $\{\llane{w}\} \times L$, a {\em median} $\{\median{w}\} \times L$, and a {\em right lane} $\{\rlane{w}\} \times L$.

The transition maps between the flow boxes induce transitions between the divided flow boxes in a natural way. Gluing the divided flow boxes together along these transitions yields a new space---a divided version of $\Sigma \smallsetminus \mathfrak{B}$.

We still need to decide what to do with the singularities. The region around a singularity $\mathfrak{b} \in \mathfrak{B}$ can be built from cut square pieces, as described in Section~\ref{tras-surfs}. Here's one of them, with the critical leaves marked. (Since the critical leaves are generally dense in $\Sigma$, I've only drawn finitely many of them thick enough to see.)
\begin{center}
\begin{tikzpicture}
\critnotchup
\end{tikzpicture}
\end{center}
The leaf through the center point is critical, of course, because the center point is $\mathfrak{b}$ itself. Here's a divided version of the same cut square:
\begin{center}
\begin{tikzpicture}
\matrix[column sep=1.2cm]{
\divdnotchup & \zoomnotchup{2mm} \\
};
\end{tikzpicture}
\end{center}
Zooming in, you can see that I've cut away $\median{\mathfrak{b}}$, but kept $\llane{\mathfrak{b}}$ and $\rlane{\mathfrak{b}}$. When a singularity is built from pieces that look like this one (and its half-rotation), the left and right lanes of the center leaf will connect up into lines running past the singularity, but the medians will remain disconnected rays that end at the singularity. Topologically, the region around a divided singularity looks like this:
\begin{center}
\begin{tikzpicture}
\matrix[column sep=0.5cm]{
\node {\includegraphics{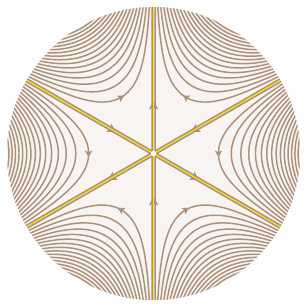}}; &
\useasboundingbox (-2.6, -2.6) rectangle (2.6, 2.6);
\zoomprongs{2}{2mm} \\
};
\end{tikzpicture}
\end{center}
The zoomed-in picture on the right shows how the lanes of the critical leaves join up around the singularity. For clarity, only the parts of the critical leaves adjacent to the singularity are shown.

Now that we've decided what to do with the singularities, we can extend our divided version of $\Sigma \smallsetminus \mathfrak{B}$ to a full {\em divided surface} $\divd{\Sigma}$. The quotient map discussed in the one-dimensional case extrudes naturally to a quotient map $\pi \maps \divd{\Sigma} \to \Sigma$. Away from the singularities, the embedding from the one-dimensional case also extrudes, yielding an embedding $\Sigma \smallsetminus \mathfrak{B} \to \divd{\Sigma}$. This embedding can't be extended over the singularities, because a singularity has no median point associated with it, and sending it to an associated lane point would break continuity.

We'll refer to the $\pi$-preimage of a leaf of $\Sigma$ as a {\em road}.\footnote{Our one-way roads are technically different from, but morally related to, the {\em two-way streets} of \cite{spec-nets}.} If $\mathcal{L}$ is a critical leaf of $\Sigma$, the {\em critical road} $\pi^{-1}\mathcal{L}$ splits into a left lane $\{\llane{w} : w \in \mathcal{L}\}$, a median $\{\hat{w} : w \in \mathcal{L}\}$, and a right lane $\{\rlane{w} : w \in \mathcal{L}\}$, as we saw locally at the beginning of the section. Each of the lane points associated with a singularity is adjacent to a lane of a forward-critical road and a lane of a backward-critical road. For convenience, we'll consider it an honorary member of both.

As in the one-dimensional case, define the {\em fractured surface} $\frakd{\Sigma}$ to be the complement of the critical leaves of $\Sigma$---or, more precisely, their images under $\iota$---in the divided surface $\divd{\Sigma}$. The divided flow boxes and divided singularity charts from which we constructed $\divd{\Sigma}$ pull back to an atlas of {\em fractured flow boxes} and {\em fractured singularity charts} on $\frakd{\Sigma}$.
\subsubsection{Properties of divided and fractured surfaces}\label{divd-props-2d}
Within each divided flow box $\divd{I} \times L$, we can find more flow boxes of the form $H \times L$, where $H \subset \divd{I}$ is a basis interval. Recall the notation $\nom{H} = H \cap \frakd{I}$ from Section~\ref{frakd-props-1d}. A flow box $U = H \times L$ will be called {\em full} if $H$ is full, and {\em well-cut} if $H$ is full and $\iota^{-1} U$ is well-cut. Define $\nom{U}$ as $U \cap \frakd{\Sigma}$, or equivalently $\nom{H} \times L$, and $\trim U$ as $(\trim H) \times L$. It's apparent from the analogous one-dimensional result that $\pi^{-1} \iota^{-1} U = \trim U$ for any flow box $U$.

Flow boxes form a basis for $\divd{\Sigma}$. If the critical leaves are dense in $\Sigma$, full flow boxes form a basis for $\frakd{\Sigma}$, as a consequence of Proposition~\ref{full-basis}. If full flow boxes form a basis for $\frakd{\Sigma}$, well-cut flow boxes do too, as a consequence of Proposition~\ref{well-cut-pieces}.

Propositions \ref{undivd-dense-1d}, \ref{divd-loc-conn-1d}, and \ref{frakd-hausdorff-1d} carry over from dimension one straightforwardly enough that I'll state them without proof.
\begin{prop}\label{undivd-dense-2d}
The embedding of $\Sigma$ in $\divd{\Sigma}$ is dense.
\end{prop}
\begin{prop}
The divided surface $\divd{\Sigma}$ is locally connected.
\end{prop}
\begin{prop}
The fractured surface $\frakd{\Sigma}$ is Hausdorff.
\end{prop}
Because $\Sigma$ is compact, Proposition~\ref{frakd-cpt-basis-1d} can be extended to a global result.
\begin{prop}
The fractured surface $\frakd{\Sigma}$ is compact.
\end{prop}
\begin{proof}
Because $\Sigma$ compact, it can be covered by a finite collection of flow boxes and singularity charts. Therefore, $\frakd{\Sigma}$ can be covered by a finite collection of fractured flow boxes and fractured singularity charts. Each fractured flow box $\frakd{I} \times L$ can be covered almost out to the edges by a ``full closed box'' of the form $\nom{H} \times C$, where $H \subset \divd{I}$ is full interval and $C \subset L$ is a closed interval. Similarly, each fractured singularity chart can be covered almost out to the edges by a finite collection of full closed boxes, taking crucial advantage of the fact that the non-compact medians of the critical rays have been removed.

With a little care, we can now cover $\frakd{\Sigma}$ with a finite collection of full closed boxes. Each full closed box is a product of two compact spaces, so we've covered $\frakd{\Sigma}$ with a finite collection of compact sets.
\end{proof}
Theorem~\ref{loc-sys-equiv-1d} carries over with a catch: dividing a translation surface opens up a tiny hole at each singularity, and a local system on the divided surface can have nontrivial holonomies around those holes.
\begin{thm}\label{loc-sys-equiv-2d}
For any group $G$, the direct image functors $\pi_*$ and $\iota_*$ give an equivalence between the groupoid of $G$ local systems on $\divd{\Sigma}$ and the groupoid of $G$ local systems on $\Sigma \smallsetminus \mathfrak{B}$.
\end{thm}
This identifies the $G$ character stacks of $\divd{\Sigma}$ and $\Sigma \smallsetminus \mathfrak{B}$.

The reason for the correspondence between local systems remains the same.
\begin{lemma}
If $\mathcal{F}$ is a local system on $\divd{\Sigma}$, the restriction $\mathcal{F}_{\trim U \subset U}$ is an isomorphism for any flow box $U$.
\end{lemma}
\begin{proof}
Write $U$ as $H \times L$ for some basis interval $H \subset \divd{I}$ and repeat the proof of Lemma~\ref{trim-iso}, replacing $A$ with $A \times L$.
\end{proof}
\begin{proof}[Proof of Theorem~\ref{loc-sys-equiv-2d}]
Replace
\[
\begin{gathered} \xymatrix{
\divd{I} \ar@/_/[d]_{\pi} \\
I \ar@/_/[u]_{\iota}
} \end{gathered}
\quad\text{with}\quad
\begin{gathered} \xymatrix{
\divd{\Sigma} \ar@/_/[d]_{\pi} \\
\Sigma \smallsetminus \mathfrak{B} \ar@/_/[u]_{\iota}
} \end{gathered}
\]
in the proof of Theorem~\ref{loc-sys-equiv-1d}, and change all the basis intervals to flow boxes.
\end{proof}
\subsection{Dynamics on divided surfaces}
\subsubsection{The vertical flow}\label{divd-flow}
Within a flow box $I \times L$ in $\Sigma$, the vertical flow is easy to describe:
\[ \psi^t \maps (s, \zeta) \mapsto (s, \zeta + t), \]
as long as $t$ is small enough that $\zeta + t$ is still in the interval $L$. This local description can be lifted directly to the divided flow box $\divd{I} \times L$ in $\divd{\Sigma}$; the only change is that $s$ will now be a point in $\divd{I}$ instead of in $I$.

The vertical flow on a singularity chart can be lifted in the same way. It's worth thinking carefully about how the lifted flow acts on the critical roads of $\divd{\Sigma}$. Pick a singularity chart on $\Sigma$, and look at a point $w$ on a forward-critical leaf that plunges into the singularity without leaving the chart. In $\divd{\Sigma}$, the point $w$ splits into the three points $\llane{w}, \median{w}, \rlane{w}$ on a forward-critical road. As time runs forward, the point $w$ falls into the singularity and disappears. Its lift $\median{w}$, on the median of the road, does the same. The points $\llane{w}$ and $\rlane{w}$, however, follow the left and right lanes past the singularity, peeling off in different directions:
\begin{center}
\begin{tikzpicture}
\matrix[column sep=0.6cm]{
\veerleft & \fallmiddle & \veerright \\
};
\end{tikzpicture}
\end{center}
On a backward-critical road, the story is the same, but told backward.

Just as the local vertical flows on flow boxes and singularity charts fit together into a global vertical flow on $\Sigma \smallsetminus \mathfrak{B}$, their lifts fit together into a global vertical flow on $\divd{\Sigma}$. For convenience, I'll refer to this vertical flow also as $\psi$. The vertical flows on $\divd{\Sigma}$ and $\Sigma \smallsetminus \mathfrak{B}$ commute with the embedding $\iota$. They don't quite commute with the quotient map $\pi$, because when a point on a critical leaf of $\Sigma \smallsetminus \mathfrak{B}$ falls into a singularity, its left- and right-lane lifts keep going. The forward vertical flows on $\Sigma \smallsetminus \mathfrak{W}^-$ and its $\pi$-preimage, however, do commute with $\pi$. The same goes for the backward vertical flows on $\Sigma \smallsetminus \mathfrak{W}^+$ and its $\pi$-preimage.

The vertical flow on $\divd{\Sigma}$, like the one on $\Sigma \smallsetminus \mathfrak{B}$, is a flow by bicontinuous relations. This should not be taken for granted: the topology of a divided interval was engineered to make it so. The medians of the critical leaves are the only parts of $\divd{\Sigma}$ that vanish into the singularities under the vertical flow. Removing them leaves an invariant subspace, $\frakd{\Sigma}$, on which the vertical flow acts by homeomorphisms. Thus, while the vertical flows on $\Sigma \smallsetminus \mathfrak{B}$ and $\divd{\Sigma}$ may look a bit ugly, the vertical flow on $\frakd{\Sigma}$ is remarkably well-behaved: it's a flow by homeomorphisms on a compact Hausdorff space.

The divided and fractured surfaces are foliated by the orbits of the vertical flow---or they would be, at least, if foliations were usually defined on more general spaces than manifolds. With that in mind, we'll sometimes refer to the orbits of the vertical flows on $\divd{\Sigma}$ and $\frakd{\Sigma}$ as vertical leaves.
\subsubsection{First return maps}\label{divd-first-return}
Let's say a {\em horizontal segment} on a divided surface is a subset that looks like a horizontal basis interval in some flow box. More precisely, it's a horizontal slice $H \times \{\zeta\}$ across a divided flow box $H \times L$. The quotient map $\pi$ projects horizontal segments on $\divd{\Sigma}$ down to horizontal segments on $\Sigma$, as defined in Section~\ref{first-return}. Lemma~\ref{return-lemma}, which made it sensible to talk about first return maps on a translation surface, has an analogue on the divided surface.
\begin{lemma}\label{divd-return-lemma}
Let $Z$ be a horizontal segment on $\divd{\Sigma}$, and let $p$ be a point in $Z$. Unless $p$ is on the median of a forward-critical road, the vertical flow will eventually carry $p$ back to $Z$.
\end{lemma}
The proof is edifying, but rather long, so I've postponed it to the end of the section. The most important consequence of this lemma is that we can define a first return relation on any horizontal segment in $\divd{\Sigma}$, just like we did for horizontal segments in $\Sigma$.

Suppose $Z$ is a horizontal slice across a well-cut flow box $U = H \times L$ in $\divd{\Sigma}$. Then $\iota^{-1} Z$ is a well-cut horizontal segment in $\Sigma$, so the first return relation on $\iota^{-1} Z$ is an interval exchange. The divided version of that interval exchange, constructed as in Section~\ref{divd-ex}, is precisely the first return relation on $Z$.

Identifying $Z$ and $\iota^{-1} Z$ with $H$ and $\iota^{-1} H$ in the obvious way, we can think of the first return relations as relations on $H$ and $\iota^{-1} H$. These relations don't depend on which slice across $H \times L$ we take, so we can say directly that a well-cut flow box $H \times L$ comes with first return relations on $H$ and $\iota^{-1} H$. The one on $H$ becomes a function when restricted to $\nom{H}$.

We'll be working with these first return relations a lot, so it will be useful to set down a pattern of notation for talking about them. For convenience, I'll refer to the first return relations on $H$ and $\iota^{-1} H$ both as $\alpha$. Let $W \subset \iota^{-1} H$ be the positions of the critical leaves, recalling that $\nom{H} = H \smallsetminus \iota W$. Within $W$, let $W^+$ and $W^-$ be the positions of the backward- and forward-critical leaves. If $\Sigma$ has no saddle connections, $W^+$ and $W^-$ are disjoint, and thus form a partition of $W$.

Let $B^+ \subset W^+$ be the places where the backward-critical leaves first pass through $\iota^{-1} U$ after shooting out of their singularities. Similarly, let $B^- \subset W^-$ be the places where the forward-critical leaves last pass through $\iota^{-1} U$ before diving into their singularities. The sets $B^+$ and $B^-$ are the break points of the backward and forward first return relations $\alpha^{-1}$ and $\alpha$ on $\iota^{-1} H$, as described in Section~\ref{divd-ex}.

\begin{proof}[Proof of Lemma~\ref{divd-return-lemma}]
If $p$ is in $\iota \Sigma$, we can just apply Lemma~\ref{return-lemma} to an appropriate closed subinterval of $\pi Z$, and we're done. The only time we have to do something less direct is when $p$ is in the left or right lane of a critical road.

Suppose $p = \rlane{w}$ for some $w \in \mathfrak{W} \smallsetminus \mathfrak{W}^-$. Because it's in the right lane of a critical road, $p$ can't be the rightmost element of a horizontal basis interval. We can therefore assume, without loss of generality, that $Z$ has no rightmost element. In this case, the only way for $\pi^{-1} \pi Z$ to contain more points than $Z$ is for $Z$ to have a leftmost point $\rlane{a}$, in which case $\pi^{-1} \pi Z = \{\llane{a}, \median{a}\} \cup Z$.

By Lemma~\ref{return-lemma}, $\psi^t w$ returns to $\pi Z$ at some time $t > 0$. Away from $\mathfrak{W}^-$ and its $\pi$-preimage, the forward vertical flows on $\divd{\Sigma}$ and $\Sigma$ commute, so $\pi \psi^t p \in \pi Z$. In other words, $\psi^t p \in \pi^{-1} \pi Z$. Points traveling along the vertical flow on $\divd{\Sigma}$ never change lanes, so $\psi^t p$ is in a right lane. Hence, knowing that $\psi^t p \in \pi^{-1} \pi Z$, we can conclude that $\psi^t p \in Z$, which is what we wanted to show. The same reasoning can be used when $p = \llane{w}$ for some $w \in \mathfrak{W} \smallsetminus \mathfrak{W}^-$.

On the other hand, suppose $p = \rlane{w}$ for some $w \in \mathfrak{W}^-$. In this case, we can assume without loss of generality that $p$ is the leftmost point of $Z$. Follow $\psi^t Z$ upward along the vertical flow until $\psi^t p$ passes a singularity, exiting the forward-critical road it started on and merging onto the adjacent backward-critical road. By this time, a few pieces of $\psi^t Z$ may have already hit singularities and peeled off to the right, but some piece of $\psi^t Z$ is still traveling with $\psi^t p$.
\begin{center}
\exits
\end{center}
The backward-critical road $\psi^t p$ is now following might also be forward-critical---a saddle connection. In that case, we can repeat the same maneuver. Because $\Sigma$ is compact, it has only finitely many forward-critical roads, so $\psi^t p$ will eventually end up on either a road that isn't forward-critical or a road that it's traveled before. In the latter case, because $\psi^t p$ can only merge onto a road at its very beginning, $\psi^t p$ will eventually return to a point it's passed through before. Since $\psi$ is a flow by bicontinuous relations, as defined in Appendix~\ref{rel-dyne}, it follows that $\psi^t p$ is defined for all $t \in \R$, and periodic. As a result, $\psi^t p$ will eventually return to its starting point, and thus to $Z$.

This leaves us with the case where $\psi^t p$ finally merges onto a road that isn't forward critical. In other words, at some time $t > 0$, we have $\pi \psi^t p \in \mathfrak{W} \smallsetminus \mathfrak{W}^-$. Let $Y$ be the piece of $\psi^t Z$ that's stuck with $\psi^t p$ all this time. We showed earlier that the vertical flow eventually carries $\psi^t p$ back to $Y$. The vertical flow is injective, so $\psi^t p$ must have passed through $\psi^{-t} Y$ on its way back to $Y$. Since $\psi^{-t} Y \subset Z$, we've proven that $\psi^t p$ eventually returns to $Z$.
\end{proof}
\subsubsection{Minimality}\label{divd-minimality}
The vertical flow on $\frakd{\Sigma}$ is very well-behaved, as I noted in Section~\ref{divd-flow}. It acts even more nicely when $\Sigma$ has no saddle connections.
\begin{prop}
If $\Sigma$ has no saddle connections, the forward and backward vertical flows on $\frakd{\Sigma}$ are both minimal as dynamical systems.
\end{prop}
\begin{proof}
The forward and backward cases are mirror images of each other, so let's focus on the forward vertical flow. Suppose $\Sigma$ has no saddle connections, recalling that this implies $\Sigma$ is minimal. We want to prove that the forward orbit $\mathcal{P}$ of any point $p \in \frakd{\Sigma}$ is dense.

Suppose $p = \iota q$ for some $q \in \Sigma \smallsetminus \mathfrak{W}$. Then we can just observe that the minimality of $\Sigma$ implies that the forward orbit of $q$ is dense in $\Sigma$. Since $\iota \Sigma$ is dense in $\frakd{\Sigma}$ by Proposition~\ref{undivd-dense-2d}, we're done.

Suppose $p$ is in the left or right lane of a backward-critical road. In other words, $p \in \{\llane{w}, \rlane{w}\}$ for some $w \in \mathfrak{W}^+$. Let $\mathcal{W}$ be the forward orbit of $w$. Because $\Sigma$ has no saddle connections, the backward-critical leaf $w$ lies on can't also be forward-critical, so the minimality of $\Sigma$ implies that $\mathcal{W}$ is dense in $\Sigma$.

Because $\Sigma$ is minimal, full flow boxes form a basis for $\frakd{\Sigma}$, as pointed out in Section~\ref{divd-props-2d}. Thus, to show that $\mathcal{P}$ is dense in $\frakd{\Sigma}$, we just have to show that it intersects every full flow box $V$. Note that $V$ is the intersection of $\frakd{\Sigma}$ with a full flow box $U$ in $\divd{\Sigma}$. Since $\iota$ is continuous, $\iota^{-1} U$ is an open subset of $\Sigma$, and therefore intersects the dense orbit $\mathcal{W}$. That means $\pi^{-1} \iota^{-1} U$ intersects $\iota \mathcal{W}$. Recall from Section~\ref{divd-props-2d} that $\pi^{-1} \iota^{-1} U = \trim U$. In $\divd{\Sigma}$, any open subset that contains a point in the median of a critical road also contains the corresponding points in the left and right lanes, so $\trim U$ intersects $\mathcal{P}$. Since $\trim U \subset U$, we've shown that $U$ intersects $\mathcal{P}$. Since $\mathcal{P} \subset \frakd{\Sigma}$, it follows that $V$ intersects $\mathcal{P}$. Since $V$ could have been any full flow box in $\frakd{\Sigma}$, we've proven that $\mathcal{P}$ is dense in $\frakd{\Sigma}$, under the assumption that $p$ is in the left or right lane of a backward-critical road.

Finally, suppose $p$ is in the left or right lane of a forward-critical road. The forward vertical flow will eventually carry $p$ past a singularity, where it will leave its current road and merge onto an adjacent backward-critical road. That means the forward orbit of $p$ contains the left or right lane of a backward-critical road. We just proved that the left and right lanes of a backward-critical road are both dense in $\frakd{\Sigma}$, so we're done.
\end{proof}
On the fractured surface, a minimal vertical flow induces minimal first return maps.
\begin{prop}
Let $H \times L$ be a well-cut flow box in $\divd{\Sigma}$. If the vertical flow on $\frakd{\Sigma}$ is minimal, the first return map on $\nom{H}$ is too.
\end{prop}
\begin{proof}
Suppose the first return map on $\nom{H}$ is not minimal. Find a closed invariant subset $C \subset \nom{H}$ which is neither empty nor all of $\nom{H}$. Let $\mathcal{C}$ be the orbit of $C \times L$ under the vertical flow. Notice that $\mathcal{C}$ can't intersect the open set $(\nom{H} \smallsetminus C) \times L$: if it did, the first return map on $\nom{H}$ would send some element of $C$ into $\nom{H} \smallsetminus C$. Since $\mathcal{C}$ is made of vertical orbits, we've shown that not every orbit of the vertical flow is dense in $\frakd{\Sigma}$. Hence, the vertical flow on $\frakd{\Sigma}$ is not minimal.
\end{proof}
\section{Warping local systems on divided surfaces}\label{warping-on-divd}
\subsection{Overview}
Many classic geometric constructions involve cutting, shifting, and regluing a local system on a manifold along something akin to a codimension-one submanifold. For example, a Fenchel-Nielsen twist of a hyperbolic surface shifts the $\on{PSL}_2 \R$ local system of charts along a closed geodesic. A complex earthquake of a pleated hyperbolic surface shifts the $\on{PSL}_2 \C$ local system of pleated path-isometries along a geodesic lamination~\cite[\S 5.3]{cp-structs}\cite[\S II.3]{convex-hulls}\cite[\S 1.3]{geom-of-qf}. Analogous processes act on $\on{PSL}_n \R$ local systems in higher Hitchin components~\cite{closed-edge-flows}\cite{anosov-cataclysms}.

The abelianization process described in this paper shifts an $\on{SL}_2 \C$ local system along the critical leaves of a translation surface $\Sigma$. It's most conveniently carried out by pushing the local system up to the divided surface $\divd{\Sigma}$ and warping it along a deviation supported on $\frakd{\Sigma}$. The special class of deviations used in this process will be the subject of this section.

Our discussion of warping only makes sense on a locally connected space, so the local connectedness of the divided surface is now playing an important role. The fact that we can warp local systems on the divided surface, like the equivalence of groupoids of local systems given by Theorem~\ref{loc-sys-equiv-2d}, reflects the general idea that local systems on the divided surface tend to be well-behaved.
\subsection{Critical roads in a flow box}
Consider a flow box $U = H \times L$ in $\divd{\Sigma}$. The median of each critical road passes through $U$ at most countably many times, intersecting it in a collection of vertical lines I'll call {\em dividers}. The dividers are naturally ordered by their positions in $H$. Given two points $y$ and $x$ in $\nom{U}$, with $y$ sitting to the left of $x$ along $H$, let's write $(y \mid x)^U$ to denote the ordered set of dividers in $U$ that lie between $y$ and $x$.
\subsection{Deviations defined by jumps, conceptually}\label{jump-concept}
Let $\mathcal{F}$ be a locally constant sheaf on $\divd{\Sigma}$. The $\mathcal{F}$-simple flow boxes form a basis for the topology of $\divd{\Sigma}$. To specify a {\em jump} in $\mathcal{F}$, we give for each divider $P$ in a simple flow box $U$ an automorphism $j_P$ of $\mathcal{F}_U$. These automorphisms have to fit together as follows:
\begin{itemize}
\item If the basis element $U$ contains the basis element $V$, and the divider $P$ in $U$ contains the divider $Q$ in $V$, the automorphisms $j_Q$ and $j_P$ commute with the restriction morphism $\mathcal{F}_{V \subset U}$.
\end{itemize}

Consider a simple flow box $U$ and a pair of points $y$ and $x$ in $\nom{U}$, with $y$ to the left of $x$. In the presence of a jump $j$, the ordered set $(y \mid x)^U$ of dividers between $y$ and $x$ becomes an ordered set of automorphisms of $\mathcal{F}_U$, which are just begging to be composed. There are probably infinitely many of them, so composing them may not make sense, but let's do it anyway, and call the result
\[ \delta^U_{yx} = \prod_{P \in (y \mid x)^U} j_P. \]
For good measure, define $\delta^U_{xy}$ to be the inverse of $\delta^U_{yx}$, so we don't have to worry about checking that $y$ is to the left of $x$ in $U$.

Our notation makes $\delta$ look like a deviation from $\mathcal{F}$ with support $\frakd{\Sigma}$, defined over the basis of $\mathcal{F}$-simple flow boxes. In fact, $\delta$ really will be a deviation as long as the compositions defining it make sense, and behave in the way you'd expect. We'll see this concretely in the next section, where we specialize to the case of jumps in local systems.
\subsection{Deviations defined by jumps, concretely}
Let's say $G$ is a Hausdorff topological group, and $\mathcal{F}$ is a $G$ local system. In this case, a jump in $\mathcal{F}$ assigns an element of $G$ to each divider. We can use the topology of $G$ to make sense of infinite ordered products, as described in Appendix~\ref{ord-prod}. Using the properties of these products, we can show that the construction in the previous section really does produce a deviation $\delta$ from a jump $j$, as long as all the products converge.
\subsubsection{The restriction property}
We want to prove that for any simple flow boxes $V \subset U$ and any two points $y, x \in V \cap \frakd{\Sigma}$, the automorphisms $\delta^U_{yx}$ and $\delta^V_{yx}$ commute with the restriction morphism $\mathcal{F}_{V \subset U}$. We can assume, without loss of generality, that $y$ is to the left of $x$. Jumps are required to commute with restriction, so
\[ \prod_{P \in (y \mid x)^U} j_P = \prod_{P \in (y \mid x)^V} \mathcal{F}_{V \subset U}^{-1}\;j_P\;\mathcal{F}_{V \subset U}. \]
Recalling that conjugation is a topological group automorphism, we apply Proposition~\ref{prod-equivar} and get
\[ \prod_{P \in (y \mid x)^U} j_P = \mathcal{F}_{V \subset U}^{-1} \left( \prod_{P \in (y \mid x)^V} j_P \right) \mathcal{F}_{V \subset U}, \]
which is what we wanted to show.
\subsubsection{The composition property}
We want to prove that $\delta^U_{zy} \delta^U_{yx} = \delta^U_{zx}$ for any simple flow box $U$ and any three points $z, y, x \in \nom{U}$. If $z, y, x$ happen to be ordered from left to right, we're trying to prove that
\[ \left( \prod_{P \in (z \mid y)^U} j_P \right)
\left( \prod_{P \in (y \mid x)^U} j_P \right)
= \prod_{P \in (z \mid x)^U} j_P. \]
Since, in the notation of Section~\ref{ord-prod}, $(z \mid y)^U \sqcup (y \mid x)^U = (z \mid x)^U$, this follows directly from Proposition~\ref{prod-compose}.

Now, suppose the three points are ordered differently. If the ordering from left to right is $y, z, x$, we can use the reasoning above to conclude that $\delta^U_{yz} \delta^U_{zx} = \delta^U_{yx}$, rewrite this as $(\delta^U_{zy})^{-1} \delta^U_{zx} = \delta^U_{yx}$, and rearrange to get the desired result. The other cases work similarly.
\section{Uniform hyperbolicity for $\on{SL}_2 \C$ dynamics}\label{unif-hyp}
\subsection{Motivation and notation}\label{unif-hyp-motive}
In Section~\ref{jump-concept}, we saw a way to turn a jump into a deviation by taking infinite products of jump automorphisms. To apply that construction to a given jump in a local system $\mathcal{E}$ on $\divd{\Sigma}$, we need some way of showing that all those infinite products converge. A convenient approach is to show that the jump automorphisms decay rapidly as you follow them out along the critical roads of $\divd{\Sigma}$.

{\em Uniform hyperbolicity} is a powerful and well-studied decay condition on the sections of a linear local system over a dynamical base. It will be a key player in our abelianization procedure, with a role going far beyond ensuring that jumps converge.

A bit of notation will streamline our reasoning about growth and decay. For positive functions $f$ and $g$, we'll write $f \lesssim g$ to say that $f$ is bounded by a constant multiple of $g$. When $f$ and $g$ depend on several parameters, we might say that $f \lesssim g$ over a parameter $t$ to specify that we're thinking of $f$ and $g$ as functions of $t$, with the other parameters held fixed.
\subsection{The global version}\label{global-unif-hyp}
Fix a linear $\on{SL}_2 \C$ local system $\mathcal{E}$ on $\divd{\Sigma}$. Parallel transport along the vertical flow $\psi \maps \R \times \frakd{\Sigma} \to \frakd{\Sigma}$ induces a flow $\Psi$ on the stalks of $\mathcal{E}$. At time $t$, the parallel transport flow gives an isomorphism $\Psi^t \maps \mathcal{E}_x \to \mathcal{E}_{\psi^t x}$ for each $x \in \frakd{\Sigma}$. Pick an inner product on the stalks of $\mathcal{E}$ over $\frakd{\Sigma}$ which is continuous in the sense that, for any two vectors $u, v \in \mathcal{E}_U$ over an open set $U \subset \divd{\Sigma}$, the inner product of $u$ and $v$ in the stalk $\mathcal{E}_x$ varies continuously as a function of $x \in \frakd{\Sigma}$. Because $\frakd{\Sigma}$ is compact, it doesn't matter which inner product we pick.

Saying $\mathcal{E}$ is {\em globally uniformly hyperbolic} means that for every $x \in \frakd{\Sigma}$, the dynamics of $\Psi$ split $\mathcal{E}_x$ into a direct sum of two one-dimensional subspaces $\mathcal{E}^+_x$ and $\mathcal{E}^-_x$, called the {\em forward-stable} and {\em backward-stable} lines, respectively. These lines are characterized by the following properties:
\begin{enumerate}
\item The parallel transport map $\Psi^t$ sends $\mathcal{E}^\pm_x$ to $\mathcal{E}^\pm_{\psi^t x}$.
\item There is a constant $K > 0$ such that
\[ \|\Psi^{\pm t} v\| \lesssim e^{-Kt} \|v\| \]
over all $x \in \frakd{\Sigma}$, $v \in \mathcal{E}^\pm_x$, and $t \in [0, \infty)$.
\end{enumerate}
We'll call the constant $K$ a {\em bounding exponent} for $\mathcal{E}$. Bounding exponents are not by any means unique: if $K$ is a bounding exponent for a uniformly hyperbolic local system, every positive number less than $K$ is a bounding exponent too.

Let's collect the forward- and backward-stable lines into a pair of functions $\mathcal{E}^\pm$, which assign to each point $x \in \frakd{\Sigma}$ the line $\mathcal{E}^\pm_x \subset \mathcal{E}_x$. We'll call these functions the {\em stable distributions} of $\mathcal{E}$.

The decay condition in the definition of uniform hyperbolicity only talks about forward parallel transport, but we can turn it around to get an equivalent growth condition on backward parallel transport. Our argument depends crucially on the fact that $x$ and $v$ are treated as variables in the decay condition, rather than parameters that can be held fixed. That detail is the ``uniform'' part of ``uniform hyperbolicity.''
\begin{prop}\label{global-soft-turnaround}
For a given value of $K$ and choice of sign, Condition~2 in the definition of global uniform hyperbolicity holds if and only if
\[ e^{Kt} \|v\| \lesssim \|\Psi^{\mp t} v\| \]
over all $x \in \frakd{\Sigma}$, $v \in \mathcal{E}^\pm_x$, and $t \in [0, \infty)$.
\end{prop}
\begin{proof}
Suppose Condition~2 holds. Then there's a constant $C$ such that
\[ \|\Psi^{\pm t} v\| \le Ce^{-Kt} \|v\| \]
for all $x \in \frakd{\Sigma}$, $v \in \mathcal{E}^\pm_x$, and $t \in [0, \infty)$. In particular,
\[ \|\Psi^{\pm t} \Psi^{\mp t} v'\| \le Ce^{-Kt} \|\Psi^{\mp t} v'\| \]
for all $x \in \frakd{\Sigma}$, $v' \in \mathcal{E}^\pm_{\psi^{\pm t} x}$, and $t \in [0, \infty)$. Simplifying the inequality and rewriting the quantifier over $x$, we see that
\[ \|v'\| \le Ce^{-Kt} \|\Psi^{\mp t} v'\| \]
for all $x' \in \frakd{\Sigma}$, $v' \in \mathcal{E}^\pm_{x'}$, and $t \in [0, \infty)$. This implies the condition in the proposition.

We've now shown that Condition~2 implies the condition in the proposition. The same kind of reasoning can be used to prove the reverse implication.
\end{proof}
\subsection{The local version}\label{local-unif-hyp}
Consider a simple, well-cut flow box $U = H \times L$ in $\divd{\Sigma}$, and let $E = \mathcal{E}_U$. Let $\alpha \maps \nom{H} \to \nom{H}$ be the first return map discussed in Section~\ref{divd-first-return}. For each $h \in \nom{H}$, parallel transport along the leaf through $\{h\} \times L$ gives an automorphism $A_h$ of $E$, defined by the commutative square
\[ \xymatrix@C=20mm{
E \ar[d]_{\mathcal{E}_{\{h\} \times L \subset U}} \ar@{.>}[r]^{A_h} & E \ar[d]^{\mathcal{E}_{\{\alpha h\} \times L \subset U}} \\
\mathcal{E}_{\{h\} \times L} \ar[r]_{\substack{\text{parallel} \\ \text{transport}}} & \mathcal{E}_{\{\alpha h\} \times L}
} \]
These automorphisms form a cocycle over $\alpha$, which we'll call the {\em parallel transport cocycle}. Write $A^n_h$ for the parallel transport along $n$ iterations of $\alpha$, starting at $h$. Pick an inner product on $E$ (it doesn't matter which one).

Saying $\mathcal{E}$ is {\em locally uniformly hyperbolic} with respect to $U$ means that for every $h \in \nom{H}$, the dynamics of $A$ split $E$ into a direct sum of two one-dimensional subspaces $E^+_h$ and $E^-_h$, called the {\em forward-stable} and {\em backward-stable} lines, respectively. These lines are characterized by the following properties:
\begin{enumerate}
\item The parallel transport map $A_h$ sends $E^\pm_h$ to $E^\pm_{\alpha h}$.
\item There is a constant $K > 0$ such that
\[ \|A_h^{\pm n} v\| \lesssim e^{-Kn} \|v\| \]
over all $h \in \nom{H}$, $v \in E^\pm_h$, and $n \in \N$.
\end{enumerate}
This is a standard definition of uniform hyperbolicity for dynamical cocycles~\cite[\S 2.2]{lyapunov-lectures}, specialized to the task at hand. We'll call $K$ a {\em bounding exponent} for $A$.

Writing $\Proj E$ for the projective space of $E$, let's collect the forward- and backward-stable lines into a pair of functions $E^\pm \maps \nom{H} \to \Proj E$. We'll call these functions the {\em stable distributions} of $A$.

We can turn the decay condition around to get an equivalent growth condition on backward parallel transport, just like we did for the global version of uniform hyperbolicity.
\begin{prop}\label{local-soft-turnaround}
For a given value of $K$ and choice of sign, Condition~2 in the definition of local uniform hyperbolicity holds if and only if
\[ e^{Kn} \|v\| \lesssim \|A_h^{\mp n} v\| \]
over all $h \in \nom{H}$, $v \in E^\pm_h$, and $n \in \N$.
\end{prop}
\begin{proof}
Essentially the same as for Proposition~\ref{global-soft-turnaround}.
\end{proof}
\subsection{The two versions are typically equivalent}
Suppose the vertical flow on $\frakd{\Sigma}$ is minimal. Then, for any simple, well-cut flow box $U$ in $\divd{\Sigma}$, global uniform hyperbolicity is equivalent to local uniform hyperbolicity with respect to $U$. That means we can drop the distinction between them, and just talk about {\em uniform hyperbolicity}.

The fact above follows from a more general result, which doesn't require the vertical flow on $\frakd{\Sigma}$ to be minimal.
\begin{prop}\label{unif-hyp-descent}
Take any simple, well-cut flow box $U$ in $\divd{\Sigma}$.
\begin{itemize}
\item If $\mathcal{E}$ is globally uniformly hyperbolic, then it's locally uniformly hyperbolic with respect to $U$.
\item Suppose $U$ intersects every leaf of $\frakd{\Sigma}$. If $\mathcal{E}$ is locally uniformly hyperbolic with respect to $U$, then it's globally uniformly hyperbolic.
\end{itemize}
\end{prop}
\begin{proof}
We'll use all the notation and equipment of Sections \ref{global-unif-hyp} and \ref{local-unif-hyp}. Pick a horizontal slice $Z = H \times \{\zeta\}$ across $U$. For each $h \in \nom{H}$, we can identify $E$ with $\mathcal{E}_{(h,\zeta)}$ through the stalk restriction $\mathcal{E}_{(h,\zeta) \in U}$. We'll make this identification implicitly throughout the arguments that follow.

The equipment we're bringing in from earlier includes an inner product on $E$ and a continuous inner product on the stalks of $\mathcal{E}$ over $\frakd{\Sigma}$. Adjust the latter so that the inner products on $E$ and $\mathcal{E}_{(h,\zeta)}$ match for all $h \in \nom{H}$. You can make the adjustment using a continuous function $\frakd{\Sigma} \to [0, 1]$ which is one on $Z$ and zero outside of $U$. Functions like this exist because $U$ is well-cut, and hence full.
\paragraph{Global to local}
Suppose $\mathcal{E}$ is globally uniformly hyperbolic, with bounding exponent $\kappa$. Then there's a constant $C$ such that
\[ \|\Psi^{\pm t} v\| \le Ce^{-\kappa t} \|v\| \]
for all $x \in \frakd{\Sigma}$, $v \in \mathcal{E}^\pm_x$, and $t \in [0, \infty)$. Let $\tau$ be the minimum forward return time for a point on $Z$, noting that it's also the minimum backward return time.

Let $E^\pm_h \subset E$ be the line that restricts to the stable line $\mathcal{E}^\pm_{(h,\zeta)} \subset \mathcal{E}_{(h,\zeta)}$. We're going to show that $E^\pm$ are stable distributions for the parallel transport cocycle $A$. It's not hard to see that $A_h$ sends $E^\pm_h$ to $E^\pm_{\alpha h}$, so we just need to verify the decay condition.

Pick a point $h \in \nom{H}$, a number $n \in \N$, and a direction to follow the vertical flow. Let $t$ be the time it takes for the vertical flow to carry $(h, \zeta)$ back to $Z$ for the $n$th time, flowing in the chosen direction. Observe that $A^{\pm n}_h = \Psi^{\pm t}_{(h, \zeta)}$, with the sign determined by our choice of direction. By the global decay condition,
\[ \|A^{\pm n}_h v\| \le Ce^{-\kappa t} \|v\| \]
for all $v \in E^\pm_h$. Since $\tau n \le t$, it follows that
\[ \|A^{\pm n}_h v\| \le Ce^{-\kappa \tau n} \|v\| \]
for all $v \in E^\pm_h$. The constant $C$ doesn't depend on our choice of $h$, $n$, or flow direction, so we've shown that
\[ \|A^{\pm n}_h v\| \lesssim e^{-\kappa \tau n} \|v\| \]
over all $h \in \nom{H}$, $n \in \N$, and $v \in E^\pm_h$. Hence, $\mathcal{E}$ is locally uniformly hyperbolic with respect to $U$. Its parallel transport cocycle has stable distributions $E^\pm$ and bounding exponent $\kappa \tau$.

\paragraph{Local to global}
Suppose $\mathcal{E}$ is locally uniformly hyperbolic with respect to $U$, with parallel transport cocycle $A$. Suppose $U$ intersects every leaf of $\frakd{\Sigma}$. For each point $x \in \frakd{\Sigma}$, the leaf through $x$ must pass through some point $(h, \zeta) \in Z$. Pick up one of the stable lines $E^\pm_h$ and carry it to $x$ along the parallel transport flow $\Psi$, defining a line $\mathcal{E}^\pm_x \subset \mathcal{E}_x$. You get the same line over $x$ no matter which intersection with $Z$ you come from, because the stable distribution $E^\pm$ is invariant under $A$. By construction, the distributions $\mathcal{E}^\pm$ are invariant under $\Psi$.

We're going to show that $\mathcal{E}^\pm$ are stable distributions for $\mathcal{E}$. First, let's fix some constants. Pick a bounding exponent $K$ for $A$. From the definition of a bounding exponent, we get a constant $C > 0$ such that
\[ \|A_h^{\pm n} v\| \le C e^{-Kn} \|v\| \]
for all $h \in \nom{H}$, $v \in E^\pm_h$, and $n \in \N$. Let $T$ be the maximum forward return time for a point on $Z$, noting that it's also the maximum backward return time. Because $\frakd{\Sigma}$ is compact, and both the parallel transport flow and the stalkwise inner product on $\mathcal{E}$ are continuous, $\|\Psi^s_x\|$ is bounded by some constant $M$ as $x$ varies over $\frakd{\Sigma}$ and $s$ varies over the interval $[0, T]$.

Starting at any point $x \in \frakd{\Sigma}$, let the forward or backward vertical flow carry us along for time $t \in [0, \infty)$. Our trip can be broken up into three parts. First, we flow for some time $s \in [0, T]$ until we hit $Z$. Then we return to $Z$ some number $n \in \N$ of times. Finally, we flow for some additional time $s' \in [0, T]$ after our last encounter with $Z$. If we never hit $Z$, set $s$ and $n$ to zero. The parallel transport flow over our trip has a corresponding decomposition
\[ \Psi^{\pm t}_x = \Psi^{\pm s'}_{(\alpha^n h, \zeta)}\;A^{\pm n}_h\;\Psi^{\pm s}_x, \]
where $(h, \zeta) \in \nom{H} \times \{\zeta\}$ is the point where we first hit $Z$, and the sign depends on which flow direction we picked.

For any $v \in \mathcal{E}^\pm_x$, the constants we defined earlier give the bounds
\begin{align*}
\|\Psi^{\pm t}_x v\|
& = \|\Psi^{\pm s'}_{(\alpha^n h, \zeta)}\;A^{\pm n}_h\;\Psi^{\pm s}_x v\| \\
& \le M \|A^{\pm n}_h\;\Psi^{\pm s}_x v\| \\
& \le M Ce^{-Kn}\|\Psi^{\pm s}_x v\| \\
& \le M^2 Ce^{-Kn} \|v\|.
\end{align*}
We know from the way we decomposed our trip that $t \le T(n+2)$, so
\begin{align*}
\|\Psi^{\pm t}_x v\| & \le M^2 Ce^{-K(t/T - 2)} \|v\| \\
& = M^2 e^{2K} Ce^{-(K/T)t} \|v\|.
\end{align*}
The constant $M^2 e^{2K} C$ doesn't depend on our choice of $x$, $t$, or flow direction, so we've shown that
\[ \|\Psi^{\pm t}_x v\| \lesssim e^{-(K/T)t} \|v\| \]
over all $x \in \frakd{\Sigma}$, $v \in \mathcal{E}^\pm_x$, and $t \in [0, \infty)$. Hence, $\mathcal{E}$ is globally uniformly hyperbolic, with stable distributions $\mathcal{E}^\pm$ and bounding exponent $K/T$.
\end{proof}
\subsection{The stable distributions are typically Lipschitz}\label{stable-lipschitz}
Suppose $\mathcal{E}$ is locally uniformly hyperbolic with respect to a simple, well-cut flow box $U = H \times L$, still using all the notation and equipment of Section~\ref{local-unif-hyp}. Since we picked an inner product on $E$, we can say the distance between two lines in $E$ is the sine of the angle between them. This puts a metric on the projective space $\Proj E$, which I'll call the {\em sine metric}. I'll write $d_\angle(u, v)$ to mean the distance in $\Proj E$ between the lines generated by $u, v \in E$.

Suppose the critical leaves are dense in $U$, so we can put a division metric on $\nom{H}$. Then the stable distributions $E^\pm \maps \nom{H} \to \Proj E$ become maps from one metric space to another, and we can ask about their regularity. For a shallow enough choice of division metric, they turn out to be Lipschitz continuous.
\begin{prop}\label{transport-stable-lipschitz}
If $K$ is a bounding exponent for the parallel transport cocycle $A$, the stable distributions of $A$ are Lipschitz continuous with respect to the sine metric on $\Proj E$ and the division metric on $\nom{H}$ of steepness $e^{2K}$.
\end{prop}
This result isn't special to our setting; it extends to any uniformly hyperbolic ``Markov cocycle'' over a two-sided sequence space. I learned it from a paper by Damanik, Fillman, Lukic, and Yessen, who prove it in passing while constructing the stable distributions~\cite{cmv-matrices}. We'll treat it as a special case of the version for Markov cocycles, which is proven in Appendix~\ref{markov-stable-lines}.
\begin{proof}[Reduction of Proposition~\ref{transport-stable-lipschitz} to Proposition~\ref{markov-stable-lipschitz}]
Embed $\nom{H}$ isometrically into a two-sided sequence space as we did in Sections \ref{divd-ex} and \ref{division-dynamics}. Our embedding turns the parllel transport cocycle into a Markov cocycle, so Proposition~\ref{markov-stable-lipschitz} applies.
\end{proof}
\subsection{Extending over medians}\label{unif-hyp-medians}
Let's suppose in this section that $\Sigma$ has no saddle connections. Not much would change if there were saddle connections, as long as we stayed away from them, but our notation will be more coherent this way. As a bonus, we won't have to distinguish between global and local uniform hyperbolicity.

Consider a point $w$ on a critical leaf of $\Sigma$. Every neighborhood of $\median{w}$ contains $\llane{w}$ and $\rlane{w}$, so we can take a colimit over neighborhoods $U$ in the diagram
\[ \xymatrix@C15mm{
\mathcal{E}_{\llane{w}} & \ar[l]_{\mathcal{E}_{\llane{w} \subset U}} \mathcal{E}_U \ar[r]^{\mathcal{E}_{\rlane{w} \subset U}} & \mathcal{E}_{\rlane{w}}
} \]
to get isomorphisms
\[ \xymatrix@C15mm{
\mathcal{E}_{\llane{w}} & \ar[l] \mathcal{E}_{\median{w}} \ar[r] & \mathcal{E}_{\rlane{w}}
} \]
identifying the three stalks. I'll refer to all three as $\mathcal{E}_w$, writing $\llane{v}$, $\median{v}$, or $\rlane{v}$ when I want to think of a vector $v \in \mathcal{E}_w$ as an element of one or the other.

Let's say $w$ is on a backward-critical leaf. The left and right lanes of this leaf are never separated by the forward vertical flow, so $\Psi^t \llane{v} = \Psi^t \rlane{v}$ for all $v \in \mathcal{E}_w$ as long as $t \ge 0$. From the continuity of the inner product and the compactness of $\frakd{\Sigma}$, you can show that the difference between $\log \|\Psi^t \llane{v}\|$ and $\log \|\Psi^t \rlane{v}\|$ stays bounded as $t$ varies over $[0, \infty)$.\footnote{Here's a sketch of the proof. Cover $\divd{\Sigma}$ with a finite collection $\mathcal{U}$ of simple open sets. For each $U \in \mathcal{U}$, the closure of $\nom{U}$ in $\frakd{\Sigma}$ is compact. Hence, the difference between $\log \|v\|_x$ and $\log \|v\|_y$ stays bounded as $x$ and $y$ vary over $\nom{U}$ and $v$ varies over $\mathcal{E}_U$. In particular, the difference between $\log \|v\|_\llane{w}$ and $\log \|v\|_\rlane{w}$ is bounded over all $w \in \mathfrak{W} \cap \pi U$. Now, just combine the bounds over all $U \in \mathcal{U}$.} Thus, if $\mathcal{E}$ is uniformly hyperbolic, the forward-stable lines $\mathcal{E}^+_{\llane{w}}$ and $\mathcal{E}^+_{\rlane{w}}$ are the same, and we can refer to both as $\mathcal{E}^+_w$. The backward-stable lines at $\llane{w}$ and $\rlane{w}$, on the other hand, are typically different. If $w$ is on a forward-critical leaf instead of a backward-critical one, we can use the same reasoning in the other direction to define $\mathcal{E}^-_w$.

Working in a well-cut flow box $H \times L \subset \divd{\Sigma}$, it will be useful to describe our extension of $\Psi^t$ in terms of the first return map. As we did in Section~\ref{divd-first-return}, let $W \subset \pi H$ be the positions of the critical leaves, and partition it into the backward- and forward-critical sets $W^+$ and $W^-$. By our previous reasoning, $A_{\llane{w}} = A_{\rlane{w}}$ for all $w \in W^+$. Defining $A_{\median{w}}$ to be equal to both, we can extend the forward parallel transport cocycle $A$ over the medians of all backward-critical points. If $\mathcal{E}$ is uniformly hyperbolic, the forward-stable lines $E^+_{\llane{w}}$ and $E^+_{\rlane{w}}$ of $E = \mathcal{E}_U$ match, so we can define $E^+_{\median{w}}$ to be equal to both. The backward cocycle $A^{-1}$ extends over the medians of all forward-critical points in the same way, allowing us to define $E^-_{\median{w}}$ for all $w \in W^-$.
\subsection{Constructing uniformly hyperbolic local systems}\label{construct-unif-hyp}
Our abelianization process can only be carried out when $\mathcal{E}$ is uniformly hyperbolic, so it will be nice to have a way of constructing uniformly hyperbolic $\on{SL}_2 \C$ local systems on $\divd{\Sigma}$. Thanks to Proposition~\ref{unif-hyp-descent}, we can do the construction locally, in any well-cut flow box $H \times L \subset \divd{\Sigma}$ that intersects every leaf of $\frakd{\Sigma}$.

Recall that $\alpha \maps \nom{H} \to \nom{H}$ is the fractured version of an interval exchange. The parallel transport cocycle over $\alpha$ is constant on each of the exchanged intervals. Following the terminology of Appendix~\ref{markov-stable-lines}, I'll call this kind of cocycle a {\em Markov cocycle}. A local system on $\divd{\Sigma}$ is determined up to isomorphism by the parallel transport cocycle it induces over $\alpha$. Conversely, any Markov cocycle over $\alpha$ is the parallel transport cocycle of some local system on $\divd{\Sigma}$.

To get a sense of why the claims above are true, first recall that Theorem~\ref{loc-sys-equiv-2d} lets us pass from local systems on $\divd{\Sigma}$ to ones on $\Sigma \smallsetminus \mathfrak{B}$. Our conditions on $H \times L$ ensure that we can express $\Sigma \smallsetminus \mathfrak{B}$ as a suspension of the first return relation on $\iota^{-1} H$~\cite[\S 1.2.1]{lyap-teich}. Any local system on $\Sigma \smallsetminus \mathfrak{B}$ can be trivialized over the flow boxes that make up the suspension, with transition morphisms given by the parallel transport cocycle. Conversely, the transition morphisms given by an interval cocycle can be used to construct a local system over $\Sigma \smallsetminus \mathfrak{B}$.

Now all we need is a way of constructing uniformly hyperbolic Markov cocycles over $\alpha \maps \nom{H} \to \nom{H}$. Here's a method based on the ``cone field criterion,'' a topological characterization of uniform hyperbolicity~\cite[\S 2.1]{pos-exp-dense}. Fix a two-dimensional complex vector space $E$ with a volume form. Define $\mathcal{A}$, like we did in Section~\ref{divd-ex}, as the intersections of the intervals exchanged by $\alpha$ and the ones exchanged by $\alpha^{-1}$. Over each interval $J \in \mathcal{A}$, choose an open set $\Omega_J \subset \Proj E$ which is neither empty nor dense. The union
\[ \Omega = \bigcup_{J \in \mathcal{A}} J \times \Omega_J \]
is a cone field in $\nom{H} \times \Proj E$. I'll call it a {\em Markov cone field} over $\alpha$.

Following \cite{mult-uni-erg}, let's say an $\on{SL}(E)$ cocycle over $\alpha$ is {\em positive} with respect to $\Omega$ if its action maps $\overline{\Omega}$ into $\Omega$. (Here, $\overline{\blankbox}$ means closure.) By the cone field criterion, any positive cocycle is uniformly hyperbolic. For a Markov cocycle, positivity with respect to a Markov cone field is particularly easy to check, since both the cocycle and the cone field are constant on each interval in $\mathcal{A}$.

For a minimal interval exchange, the Markov cone field criterion works in both directions. Section~\ref{divd-minimality} offers convenient ways to ensure that $\alpha$ is minimal.
\begin{prop}
Suppose $\alpha$ is minimal. Then an $\on{SL}(E)$ cocycle over $\alpha$ is uniformly hyperbolic if and only if it's positive with respect to a Markov cone field over some power of $\alpha$.
\end{prop}
\begin{proof}
The ``if'' direction follows directly from the cone field criterion, as we just saw. For the ``only if'' direction, define $\mathcal{A}^n$ for $n \in \N_+$ as the intersections of the intervals exchanged by $\alpha^n$ and $\alpha^{-n}$. Because $\alpha$ is minimal, the intervals in $\mathcal{A}^1 \cup \mathcal{A}^2 \cup \mathcal{A}^3 \cup \ldots$ form a basis for the topology of $\nom{H}$.

Take a uniformly hyperbolic $\on{SL}(E)$ cocycle over $\alpha$, and call it $M$. By the cone field criterion, $M$ is positive with respect to some cone field $\Omega \subset \nom{H} \times \Proj E$. Write $\Omega$ as a union of the form
\[ \bigcup_{n \in \N_+} \left[ \bigcup_{J \in \mathcal{A}^n} J \times \Omega^n_J \right], \]
where each $\Omega^n_J$ is an open subset of $\Proj E$. For convenience, let
\[ \Omega^n = \bigcup_{J \in \mathcal{A}^n} J \times \Omega^n_J. \]
Since each element of $\mathcal{A}^{n+1}$ lies within an element of $\mathcal{A}^n$, we can arrange for
\[ \Omega^1 \subset \Omega^2 \subset \Omega^3 \subset \ldots \]
by enlarging the subsets $\Omega^{n+1}_J$ round by round. Notice that $\Omega^n$ is a Markov cone field over $\alpha^n$ as long as all the subsets $\Omega^n_J$ are nonempty.

Recall that $M$ is positive with respect to $\Omega$, which means $M\overline{\Omega} \subset \Omega$. That means $\Omega^1, \Omega^2, \Omega^3, \ldots$ form an open cover of $M\overline{\Omega}$. The space $\nom{H} \times \Proj E$ is compact, so $\overline{\Omega}$ is compact. The action of $M$ on $\nom{H} \times \Proj E$ is a homeomorphism, so $M\overline{\Omega}$ is compact too. It follows that $M\overline{\Omega} \subset \Omega^m$ for some $m \in \N_+$.

One consequence is that $M\Omega \subset \Omega^m$. Since $M\Omega$ is a cone field, all the subsets $\Omega^m_J$ must be nonempty, so $\Omega^m$ is a Markov cone field over $\alpha^m$. Another consequence is that $M\overline{\Omega^m} \subset \Omega^m$. In other words, $M$ is positive with respect to $\Omega^m$.
\end{proof}
\section{Abelianization in principle}\label{ab-princ}
\subsection{Overview}\label{ab-overview}
Now that we have the tools we need, we can turn again to our goal of extending abelianization to surfaces without punctures. At this point, it will be useful to take a closer look at the original description of abelianization, which is implicit in \cite[\S 10]{spec-nets}, but first appears explicitly in \cite[\S 4]{fenchel-nielsen}. For simplicity, we'll only talk about abelianization using translation structures, leaving aside Gaiotto, Moore, and Neitzke's more powerful and general {\em spectral networks}.
\subsubsection{Setting}
Our review takes place on a translation surface $\Sigma'$ which is compact except for a finite set of {\em punctures}. A puncture is a region homeomorphic to a punctured disk, with a translation structure taken from a certain family of shapes. This definition of a puncture is analogous to our earlier definition of a singularity. For simplicity, let's consider a translation surface whose punctures all have the simplest shape: a half-infinite cylinder. A complete list of puncture shapes, and an explanation of where they come from, can be found in Appendix~\ref{punk-shapes}.

Let's assume $\Sigma'$ has no saddle connections and at least one singularity.\footnote{A ``singularity'' of cone angle $2\pi$ counts.} In this case, if you follow a vertical leaf in some direction, your fate is easy to describe. If you're following a critical leaf in the critical direction, you will by definition fall into a singularity after a finite amount of time. Otherwise, you'll end up falling forever into a puncture; in our case, that means spiraling down the long end of a half-infinite cylinder.

Every non-critical leaf on $\Sigma'$ is thus associated with two punctures, not necessarily distinct: the punctures its ends spiral into. If you remove the critical leaves $\mathfrak{W}$, the surface $\Sigma'$ falls apart into infinite vertical strips of leaves that all go into the same punctures. Each strip is bounded by four critical leaves, joined at two singularities, as illustrated below. Gaiotto, Hollands, Moore, and Neitzke compactify the surface by filling in the punctures, so the closure of each strip becomes a quadrilateral with singularities as two of its vertices and punctures as the other two.
\begin{center}
\begin{tikzpicture}
\matrix[row sep=4mm, column sep=2cm]{
\foliate{pwbeige}{30}{60} (-1, -3) rectangle (1, 3);
\draw[pworange, line width=1.5pt] (-1, -3) -- (-1, 3);
\draw[pworange, line width=1.5pt] (1, -3) -- (1, 3);
\sing{(-1, 2/3)}{pworange}
\sing{(1, -2/3)}{pworange}
\fill[white, path fading=south] (-1.1, 3) rectangle (1.1, 2.8);
\fill[white, path fading=north] (1.1, -3) rectangle (-1.1, -2.8);
\draw[white] (-1.1, 3) -- (1.1, 3);
\draw[white] (-1.1, -3) -- (1.1, -3); &
\node {\includegraphics{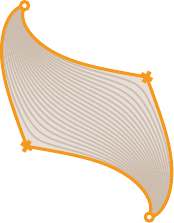}}; \\
\node {\small Strip}; &
\node {\small Compactified strip}; \\
};
\end{tikzpicture}
\end{center}
\subsubsection{Framings}
To abelianize a linear $\on{SL}_2 \C$ local system $\mathcal{E}$ on $\Sigma'$, we first need to give it a bit of extra structure, called a {\em framing}~\cite{fenchel-nielsen} (or {\em flag data}, in \cite{spec-nets}). A framing specifies a projectively flat section of $\mathcal{E}$ on a neighborhood of each puncture. For reasons that will become apparent later, I'll refer to the specified sections as {\em stable lines}. Framings always exist, because an operator on a finite-dimensional complex vector space always has at least one eigenvector.

The framing gives a pair of lines in every stalk of $\mathcal{E}$ over a non-critical leaf. One, which I'll call the {\em forward-stable} line, is gotten by following the forward vertical flow until you fall into a puncture, grabbing the stable line, and carrying it back by parallel transport. The {\em backward-stable} line is gotten in the same way by following the backward vertical flow. The forward- and backward-stable lines fit together into sections of $\mathcal{E}$ over every strip of $\Sigma' \smallsetminus \mathfrak{W}$.

The framing also gives a line in every stalk of $\mathcal{E}$ over a critical leaf---the stable line from the puncture the critical leaf falls into. This line matches the backward- or forward-stable lines in the strips on either side, depending on whether the leaf is forward- or backward-critical. Hence, as you cross the boundary between two strips, one of the stable lines stays fixed, although the other can change.
\begin{center}
\begin{tikzpicture}
\matrix[row sep=4mm]{
\node {\includegraphics[width=0.75\textwidth]{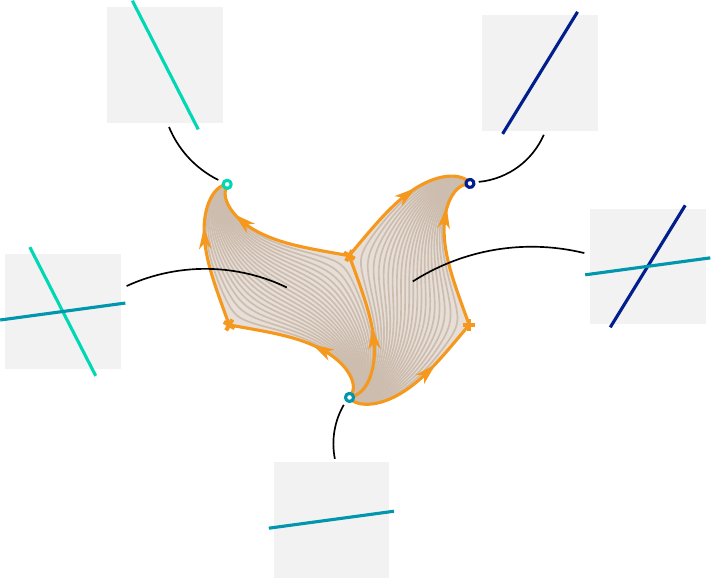}}; \\
\node {\small The stable lines over two adjacent strips}; \\
};
\end{tikzpicture}
\end{center}
\subsubsection{The WKB framing}
Although the leaves of $\Sigma'$ can plunge straight forward or backward into a puncture, they typically spiral in. This picks out a special framing, called the {\em WKB framing}, for each local system with hyperbolic or parabolic holonomies around the punctures~\cite[\S 6.5]{wkb}. When you're using abelianization to compute shear-bend coordinates, as discussed in Section~\ref{spec-coords}, the WKB framing is the one you need.

There are two projectively flat sections around each hyperbolic puncture. The stable line of the WKB framing is the one that contracts when you circle the puncture in the winding direction of the flow toward the puncture. Around a parabolic puncture, there's only one projectively flat section, and therefore only one possible choice of stable line.
\subsubsection{Abelianization}\label{punkd-ab}
For a generic local system, the forward and backward-stable lines over each strip are complementary, diagonalizing the structure group of $\mathcal{E}$ over $\Sigma' \smallsetminus \mathfrak{W}$. The changes in the stable lines at the boundaries between strips generally keep this diagonalization from extending over all of $\Sigma'$. At each critical leaf, however, there's a unique element of $\on{SL}_2 \C$ that sends the stable lines on one side to the stable lines on the other, acting by the identity on the line that stays fixed.
\begin{center}
\includegraphics[width=0.75\textwidth]{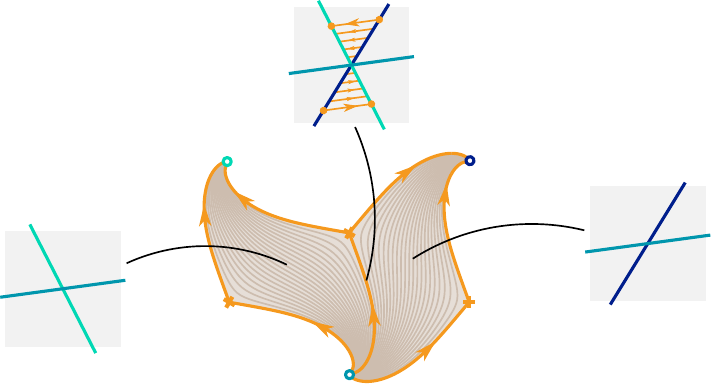}
\end{center}
By cutting $\mathcal{E}$ along the critical leaves, shifting it by this automorphism, and gluing it back together, we can match up the stable lines across the boundaries of the strips, so the diagonalization they give becomes global.
\begin{center}
\begin{tikzpicture}
\matrix{
\node {\includegraphics[width=0.95\textwidth]{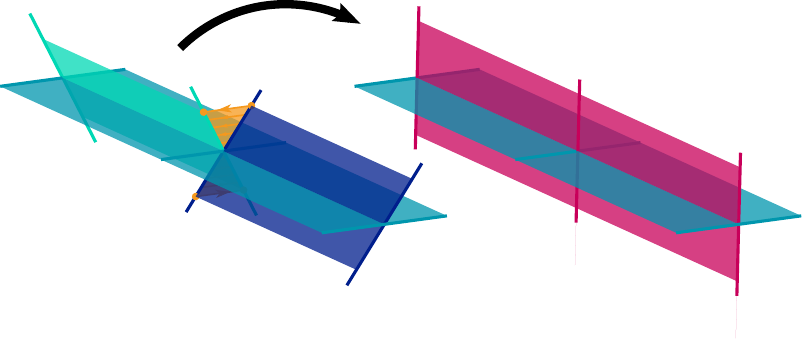}}; \\
\node {\small Cutting and gluing to match up the stable lines}; \\
};
\end{tikzpicture}
\end{center}
That's abelianization.
\subsubsection{Abelianization without punctures}\label{no-punks}
If we want to carry out the process above on a surface without punctures, there are two questions we have to settle. One is conceptual: what should play the role of the framing? If you've read Section~\ref{unif-hyp}, our suggestive terminology may have already told you the answer. When $\mathcal{E}$ is uniformly hyperbolic, it comes with complementary forward- and backward-stable lines over every non-critical leaf, which can be used as the forward- and backward-stable lines in the process above. The discussion in Section~\ref{unif-hyp-medians} amounts to saying that one of the stable lines stays fixed when you cross a critical leaf, so we can define the automorphisms over the critical leaves just as we did before.

The stable lines of a uniformly hyperbolic local system on a compact surface are analogous to the stable lines of a WKB-framed local system on a punctured surface, with hyperbolic holonomies around the punctures. In both cases, the stable lines are the ones that contract exponentially as you follow the vertical flow.

The other question is just a technical difficulty. On a compact translation surface with no saddle connections, every vertical leaf is dense, so how do we think about shifting $\mathcal{E}$ along the automorphisms over the critical leaves? How do we know the process is well-defined? How do we know the resulting local system has the diagonal subgroup of $\on{SL}_2 \C$ as its structure group? The three parts of this question are answered in Sections \ref{slithering}, \ref{ab-conv}, and \ref{ab-deliv}, respectively.
\subsection{Running assumptions}\label{running-assumptions}
The compact translation surface $\Sigma$ introduced in Section~\ref{running-notation} will, of course, stay with us. We'll discuss the abelianization of a fixed linear $\on{SL}_2 \C$ local system $\mathcal{E}$ on $\divd{\Sigma}$. Abelianization will yield a diagonalizable $\on{SL}_2 \C$ local system $\mathcal{F}$ and a stalkwise isomorphism $\Upsilon \maps \mathcal{E} \to \mathcal{F}$, supported on $\frakd{\Sigma}$.

From now on, we'll assume the following things:
\begin{itemize}
\item The surface $\Sigma$ has no saddle connections.\footnote{\label{minimality-note}This implies that $\Sigma$ is minimal, as I mentioned in Section~\ref{tras-surfs}.}
\item The local system $\mathcal{E}$ is uniformly hyperbolic.
\end{itemize}
All it takes to get rid of any saddle connections is an arbitrarily small rotation of the translation structure on $\Sigma$, as discussed in Section~\ref{tras-surfs}. Once you've fixed a saddle-connection-free translation structure, you can use the results of Section~\ref{construct-unif-hyp} to find lots of uniformly hyperbolic local systems.
\subsection{The slithering jump}\label{slithering}
\subsubsection{Overview}
Working on $\divd{\Sigma}$ gives us a convenient way to talk about the stable lines on either side of a critical leaf: using the identification in Section~\ref{unif-hyp-medians}, we can compare the stable lines over the left and right lanes. The automorphisms that match up the stable lines across the median can be encoded as a jump in the local system $\mathcal{E}$, as defined in Section~\ref{warping-on-divd}. This jump contains essentially the same information as {\em slithering maps} defined by Bonahon and Dreyer in \cite{hit-chars}, so we'll call it the {\em slithering jump}. We abelianize $\mathcal{E}$ by warping it along the deviation defined by the slithering jump. More explicitly, we abelianize $\mathcal{E}$ by carrying out the following steps:
\begin{enumerate}
\item Compute the slithering jump, using the formulas in Section~\ref{slithering-formulas}.
\item\label{shearing-devi} Turn the slithering jump into a deviation, as described in Section~\ref{jump-concept}.
\item Warp $\mathcal{E}$ along the deviation, as described in Section~\ref{warping-sheaves}.
\end{enumerate}
We'll prove in Sections \ref{ab-conv} and \ref{ab-deliv} that these instructions have the desired effect, as long as the conditions in Section~\ref{running-assumptions} are satisfied.
\subsubsection{Definition}
Consider a point $w$ on a backward-critical leaf of $\Sigma$. Because $\mathcal{E}_w$ is two-dimensional, and $\on{SL}_2 \C$ is the group of volume-preserving linear maps, there's a unique automorphism $s_w$ of $\mathcal{E}_w$ that sends $\mathcal{E}^-_{\rlane{w}}$ to $\mathcal{E}^-_{\llane{w}}$ and is the identity on $\mathcal{E}^+_w$. This induces an automorphism of $\mathcal{E}_U$ for any simple flow box $U$ containing $\median{w}$. If $w$ is on a forward-critical leaf instead of a backward-critical one, we can define $s_w$ in the same way, with the roles of the backward-stable and forward-stable lines reversed.

Given a divider $P$ in a simple flow box $U$, it's not hard to see that $s_w$ induces the same automorphism of $\mathcal{E}_U$ for every $w \in P$. Call this automorphism $s_P$. As $P$ varies over all dividers in all simple flow boxes, the automorphisms $s_P$ fit together into a jump $s$ in the local system $\mathcal{E}$---the {\em slithering jump}.
\subsubsection{Formulas}\label{slithering-formulas}
Assuming, for convenience, that $w$ is on a backward-critical leaf, we can get an explicit expression for $s_w$ by choosing representatives
\begin{align*}
u' & \in \mathcal{E}^-_{\llane{w}} &
v & \in \mathcal{E}^+_w &
u & \in \mathcal{E}^-_{\rlane{w}}.
\end{align*}
Observe that $\{v, u'\}$ and $\{v, u\}$ are ordered bases for $\mathcal{E}_w$. By rescaling $u'$ and $u$, we can ensure that both ordered bases have unit volume. The transformation $s_w$ is then given by\footnote{Given a pre-existing basis for $\mathcal{E}_w$, you can find $s_w$ by solving the matrix equation
\[ \left[ \begin{array}{c|c} \multirow{2}{*}{$v$} & \multirow{2}{*}{$u'$} \\ & \end{array} \right]
= s_w \left[ \begin{array}{c|c} \multirow{2}{*}{$v$} & \multirow{2}{*}{$u$} \\ & \end{array} \right], \]
which I have found convenient for numerical work.}
\begin{align*}
v & \mapsto v &
u & \mapsto u'.
\end{align*}
A quick calculation with the volume form $D$ gives the relation
\[ u' = u + D(u', u)\,v, \]
revealing that $s_w$ is a shear transformation whose off-diagonal part $s_w - 1$ is given by
\begin{align*}
v & \mapsto 0 &
u & \mapsto D(u',u)v.
\end{align*}
When $w$ is on a forward-critical leaf, similar expressions can be obtained.
\subsubsection{Flow invariance}\label{flow-invariance}
Suppose $w$ is on a backward-critical leaf. Because $s_w$ is defined in terms of the stable lines $\mathcal{E}^\pm_w$ and the volume form on $\mathcal{E}_w$, which are all invariant under the vertical flow, the diagram
\[ \xymatrix@C=15mm{
\mathcal{E}_{\psi^t w} & \ar[l]_{s_{\psi^t w}} \mathcal{E}_{\psi^t w} \\
\ar[u]^{\Psi^t} \mathcal{E}_w & \ar[l]^{s_w} \mathcal{E}_w \ar[u]_{\Psi^t}
} \]
commutes for all positive times $t$. If $w$ is on a forward-critical leaf, the same is true for all negative times.

This flow invariance property is not unique to the slithering jump. In fact, it holds for all jumps, as a direct consequence of the defining consistency condition. Its introduction has been delayed until now only for convenience.
\section{Abelianization converges}\label{ab-conv}
\subsection{Overview}\label{ab-conv-overview}
To show that the slithering jump defines a deviation $\sigma$, as discussed in Section~\ref{warping-on-divd}, we need to verify that the infinite product defining the automorphism $\sigma^U_{yx}$ converges for every simple flow box $U \subset \divd{\Sigma}$ and every pair of points $y, x \in \nom{U}$. Because jumps commute with restrictions, it's enough to check for convergence on a set of simple flow boxes that cover $\divd{\Sigma}$. We'll use the simple, well-cut flow boxes for this purpose.

Consider a simple, well-cut flow box $U = H \times L$ in $\divd{\Sigma}$, and let $E = \mathcal{E}_U$. Pick an inner product on $E$, so we can use the decay conditions given by the uniform hyperbolicity of $\mathcal{E}$. Denote the parallel transport cocycle and its stable distributions as we did in Section~\ref{local-unif-hyp}, extending them over medians as discussed in Section~\ref{unif-hyp-medians}. As usual, let $W \subset \pi H$ be the positions of the critical leaves, recalling that $\nom{H} = H \smallsetminus \iota W$. Label each divider $\{\median{w}\} \times L$ in $U$ by the point $w \in W$ it sits above. As we did in Section~\ref{divd-first-return}, let $B^+ \subset W^+$ and $B^- \subset W^-$ be the break points of $\alpha^{-1}$ and $\alpha$, respectively.

We'll keep using the shorthand $\blankbox \lesssim \blankbox$ to say that one positive function is bounded by a constant multiple of another, as we started doing in Section~\ref{unif-hyp-motive}.
\subsection{Bounding the jump}\label{bound-jump}
For any $b \in B^+$, we can choose representatives
\begin{align*}
u' & \in E^-_{\llane{b}} &
v & \in E^+_{\median{b}} &
u & \in E^-_{\rlane{b}}
\end{align*}
for which $D(v, u')$ and $D(v, u)$ are $1$ and conclude that $s_b - 1$ is given by
\begin{align*}
v & \mapsto 0 &
u & \mapsto D(u',u)\,v,
\end{align*}
as described in the previous section. Applying the flow invariance of jumps, we see that $s_{\alpha^n b} - 1$ is given by
\begin{align*}
A_{\median{b}}^n v & \mapsto 0 &
A_{\median{b}}^n u & \mapsto D(u',u)\,A_{\median{b}}^n v
\end{align*}
for all positive $n$.

We see from the formula above that $D(u',u)\,A_{\median{b}}^n v$ spans the image of $s_{\alpha^n b} - 1$. We also learn that the shortest vector $s_{\alpha^n b} - 1$ sends to $D(u',u)\,A_{\median{b}}^n v$ is the orthogonal projection of $A_{\median{b}}^n u$ onto the orthogonal complement of $A_{\median{b}}^n v$. From this, we can calculate the operator norm of $s_{\alpha^n b} - 1$:
\[ \|s_{\alpha^n b} - 1\| = \frac{|D(u',u)|}{d_\angle(A_{\median{b}}^n v, A_{\median{b}}^n u)}\,\frac{\|A_{\median{b}}^n v\|}{\|A_{\median{b}}^n u\|}, \]
where $d_\angle$ is the sine metric from Section~\ref{stable-lipschitz}. Rearranging a bit, we get
\[ \|s_{\alpha^n b} - 1\|
= \frac{|D(u',u)|}{d_\angle(A_{\median{b}}^n v, A_{\median{b}}^n u)}\,\frac{\|v\|}{\|u\|}
\cdot \frac{\|A_{\median{b}}^n v\|}{\|v\|}
\cdot \frac{\|u\|}{\|A_{\median{b}}^n u\|}. \]
Because $\mathcal{E}$ is uniformly hyperbolic, the first-return cocycle $A$ is uniformly hyperbolic too, as a consequence of Proposition~\ref{unif-hyp-descent}. Pick a bounding exponent $K > 0$ for $A$. The stable lines of $A$ vary continuously (Proposition~\ref{transport-stable-lipschitz}),\footnote{We need a division metric on $\nom{H}$ to apply Proposition~\ref{transport-stable-lipschitz}, so we're using the fact that the critical leaves of $\Sigma$ are dense in $U$. See footnote~\ref{division-pseudometric}.} and $\nom{H}$ is compact (Proposition~\ref{frakd-cpt-basis-1d}), so $d_\angle(A_{\median{b}}^n v, A_{\median{b}}^n u)$ is bounded away from zero. That and the uniform hyperbolicity of $A$ tell us that
\[ \|s_{\alpha^n b} - 1\| \lesssim e^{-2Kn} \]
over all $b \in B^+$ and $n \in \N$.

Applying the same reasoning in the other direction, we see more generally that
\[ \|s_{\alpha^{\pm n} b} - 1\| \lesssim e^{-2Kn} \]
over all $b \in B^\pm$ and $n \in \N$. Exchanging the roles of $u$ and $u'$, we see that the inverses of the jump automorphisms obey the same bound.
\subsection{Showing its product converges}\label{jumps-converge}
Recall that $\sigma$ is the deviation we're hoping will be defined by the slithering jump. Pick any two points $y, x \in \nom{U}$, with $y$ to the left of $x$. Since we're labeling the dividers in $U$ by points of $W$, we can think of $(y \mid x)^U$ as a subset of $W$, and write
\[ \sigma^U_{yx} = \prod_{w \in (y \mid x)^U} s_w. \]
Proposition~\ref{ord-prod-conv} in Appendix~\ref{ord-prod} tells us that this product converges if the sum
\[ C_{yx} = \sum_{w \in (y \mid x)^U} \|s_w - 1\| \]
does. (I've given the sum a name because its value, as a function of $y$ and $x$, will be useful to us later.)

Let's say every point in $(y \mid x)^U$ takes at least $n$ iterations of $\alpha$ or $\alpha^{-1}$ to hit a break point. Then the set
\[ \bigcup_{m \ge n} \left[ \vphantom{\bigcup} \{\alpha^m b : b \in B^+\} \cup \{\alpha^{-m} b : b \in B^-\} \right] \]
contains all the points in $(y \mid x)^U$, so
\[ C_{yx} \le \sum_{m \ge n} \left[ \sum_{b \in B^+} \|s_{\alpha^m b} - 1\| + \sum_{b \in B^-} \|s_{\alpha^{-m} b} - 1\| \right]. \]
Applying the bound from the previous section, we see that
\begin{align*}
C_{yx} & \lesssim \sum_{m \ge n} \left[ \sum_{b \in B^+} e^{-2Km} + \sum_{b \in B^-} e^{-2Km} \right] \\
& \lesssim \sum_{m \ge n} e^{-2Km}.
\end{align*}
Hence, the sum defining $C_{yx}$ converges.

Summing the geometric series, we learn that $C_{yx} \lesssim e^{-2Kn}$. But $n$ is the grade of the highest point in $(y \mid x)^U$, so $e^{-2Kn}$ is the distance between $y$ and $x$ in the division metric with steepness $e^{2K}$! The implied constant multiple in the bound above doesn't depend on $y$ and $x$, so we've proven that $C_{yx} \lesssim d(y, x)$ over all $y, x \in \nom{U}$ with $y$ to the left of $x$.

For symmetry, define
\[ C_{yx} = \sum_{w \in (x \mid y)^U} \|s_w^{-1} - 1\| \]
for $y$ to the right of $x$. Since the jump automorphisms and their inverses decay at the same rate, we can show in general that $C_{yx} \lesssim d(y, x)$ over all $y, x \in \nom{U}$.
\section{Abelianization delivers}\label{ab-deliv}
\subsection{Overview}
Now we know the slithering jump defines a deviation $\sigma$, so we can warp $\mathcal{E}$ along this deviation to produce a new local system $\mathcal{F}$ and a stalkwise isomorphism $\Upsilon \maps \mathcal{E} \to \mathcal{F}$, supported on $\frakd{\Sigma}$. By design, $\Upsilon$ matches up the stable lines of $\mathcal{E}$ across the medians of $\divd{\Sigma}$, sending corresponding stable lines in $\mathcal{E}_\llane{w}$ and $\mathcal{E}_\rlane{w}$ to the same line in $\mathcal{F}_\median{w}$ for all $w \in \mathfrak{W}$. To diagonalize the structure group of $\mathcal{F}$, we need to prove that $\Upsilon$ matches up the stable lines on larger scales. For any simple flow box $U \subset \divd{\Sigma}$, we have to show that $\Upsilon$ sends the corresponding stable lines in all the stalks of $\mathcal{E}$ over $\nom{U}$ to the stalk restrictions of a single line in $\mathcal{F}_U$. Because of the way deviations restrict, it's enough to prove the desired result on a set of simple flow boxes that cover $\divd{\Sigma}$, just like in Section~\ref{ab-conv}. Once again, we'll use the simple, well-cut flow boxes for this purpose.

We'll keep all the notation from Section~\ref{ab-conv}, and add to it the shorthand $F = \mathcal{F}_U$. To make the geometry facts from Appendix~\ref{euclidean} available, scale the inner product on $E$ so that the unit square has unit volume. To make the results from Sections \ref{ab-conv} and \ref{lipschitz-functions} available, pick a bounding exponent $K > 0$ for $A$, and give $\nom{H}$ the division metric\footnote{\label{division-pseudometric}For our definition of the division metric to yield a true metric, rather than a pseudometric, we need $W$ to be dense in $\pi H$. Equivalently, we need the critical leaves of $\Sigma$ to be dense in $U$. Footnote~\ref{minimality-note} guarantees this.} with steepness $e^{2K}$.

The argument we're about to do is somewhat technical, so let's first recall how it works over a punctured surface $\Sigma'$, where it's very straightforward. Think of the stable lines of $\mathcal{E}$ as lines in $E$ parameterized by the points of $\nom{U}$, and think of their images in $\mathcal{F}$ under $\Upsilon$ as lines in $F$. The stable lines are constant in $E$ away from the critical leaves of $\Sigma'$, and the slithering jump only disturbs $\mathcal{E}$ at the critical leaves. Hence, the images under $\Upsilon$ of the stable lines are constant in $F$ away from the critical leaves. By design, the images of the stable lines under $\Upsilon$ are also constant across the critical leaves, so they must be constant everywhere. In other words, $\Upsilon$ matches up the images of the stable lines in $F$ all across $\nom{U}$.

This argument hinges on the fact that the images of the stable lines are ``rigid,'' in the sense that they're locally constant away from the critical leaves. On the unpunctured surface $\Sigma$, it's too hard to get away from the critical leaves for that notion of rigidity to make sense. Fortunately, as we saw in Section~\ref{lipschitz-functions}, the fractured surface $\frakd{\Sigma}$ comes with its own notion of rigidity: being locally Lipschitz in the horizontal direction. Theorem~\ref{no-bulk-flex} lets us plug this notion of rigidity into the argument we used over $\Sigma'$.

We know that the stable lines over $\frakd{\Sigma}$ are rigid (Section~\ref{stable-lipschitz}). We'll show that the images of the stable lines are rigid as well (Sections \ref{after-ab} \thru \ref{still-lipschitz}). We'll then see, with no fuss, that the images of the stable lines are constant (Section~\ref{ab-const}).
\subsection{The stable distributions after abelianization}\label{after-ab}
Recall that warping $\mathcal{E}$ along $\sigma$ has given us a new local system $\mathcal{F}$ and a stalkwise isomorphism $\Upsilon \maps \mathcal{E} \to \mathcal{F}$, supported on $\frakd{\Sigma}$. We're using the shorthand $F = \mathcal{F}_U$. Because $U$ is simple, the stalk restrictions of $\mathcal{E}$ and $\mathcal{F}$ identify $\mathcal{E}_p$ with $E$ and $\mathcal{F}_p$ with $F$ for every $p \in \nom{U}$. We can thus view $\Upsilon$ as a map from $\nom{U}$ to $\on{SL}(E, F)$. Because $\sigma$ comes from a jump, $\Upsilon$ is constant along the vertical leaves of $\frakd{\Sigma}$, so in fact we can treat $\Upsilon$ as a map from $\nom{H}$ to $\on{SL}(E, F)$. This section takes place entirely within the flow box $U$, so we'll abbreviate $\sigma^U_{yx}$ as $\sigma_{yx}$ and $(y \mid x)^U$ as $(y \mid x)$.

Just as parallel transport in $\mathcal{E}$ along the vertical flow gave the linear cocycle $A \maps \nom{H} \to \on{SL}(E)$, parallel transport in $\mathcal{F}$ along the vertical flow gives a linear cocycle $\nom{H} \to \on{SL}(F)$. Define $F^\pm_h \subset F$ as the images of the lines $E^\pm_h$ under $\Upsilon_h$. Like $E^\pm$, the distributions $F^\pm$ are invariant under the parallel transport cocycle.\footnote{As a matter of fact, $\mathcal{F}$ should be uniformly hyperbolic, with $F^\pm$ as its stable distributions. We don't need to know that, though.} Let's put an inner product on $F$ by declaring $\Upsilon_a$, for some arbitrary $a \in \nom{H}$, to be an isometry. We then get a sine metric on $\Proj F$, and we can ask whether the functions $F^\pm$ are Lipschitz.\footnote{In light of the previous footnote, you might hope to show that $F^\pm$ are Lipschitz the same way we showed that $E^\pm$ are Lipschitz, by applying Proposition~\ref{markov-stable-lipschitz}. The snag is that the parallel transport cocycle for $F$ typically won't be a Markov cocycle.}

Recalling that $\Upsilon_h = \Upsilon_a \sigma_{ah}$, we see that $F^\pm_h = \Upsilon_a \sigma_{ah} E^\pm_h$ for all $h \in \nom{H}$. Since $\Upsilon_a$ is, by definition, an isometry,
\[ d_\angle(F^\pm_y, F^\pm_x) = d_\angle(\sigma_{ay} E^\pm_y, \sigma_{ax} E^\pm_x) \]
for all $y, x \in \nom{H}$. We might therefore be able to prove that $F^\pm$ are Lipschitz by looking at how $\sigma^U$ affects distances in $\mathbf{P}E$.
\subsection{The abelianized stable distributions are still Lipschitz}\label{still-lipschitz}
\subsubsection{The deviation between nearby points is close to the identity}
Remember the sum $C_{yx}$ we used in Section~\ref{jumps-converge}? We'll soon see that $\|\sigma_{yx} - 1\| \lesssim C_{yx}$ over all $y, x \in \nom{H}$. Combining this with the bound $C_{yx} \lesssim d(y, x)$ proven at the end of Section~\ref{jumps-converge}, we'll learn that
\[ \|\sigma_{yx} - 1\| \lesssim d(y, x) \]
over all $y, x \in \nom{H}$.

Let's get down to business. Recall that
\begin{align*}
\sigma_{yx} & = \prod_{w \in (y \mid x)} s_w & C_{yx} & = \sum_{w \in (y \mid x)^U} \|s_w - 1\|,
\end{align*}
assuming for convenience that $y$ lies to the left of $x$. Knowing that $C_{yx}$ converges, we can use Proposition~\ref{ord-prod-near-id} to bound $\|\sigma_{yx} - 1\|$. The norm of a volume-preserving transformation is always at least one, so the bound simplifies to
\[ \|\sigma_{yx} - 1\| \le C_{yx} \left(\prod_{w \in (y \mid x)} \|s_w\|\right). \]
Applying Proposition~\ref{ord-prod-bound}, we learn that
\[ \|\sigma_{yx} - 1\| \le C_{yx} \exp C_{yx}. \]
Since $C_{yx} \lesssim d(y, x)$, and distances in $\nom{H}$ are bounded, we can bound $\exp C_{yx}$ by a constant. Hence, $\|\sigma_{yx} - 1\| \lesssim C_{yx}$. Although we've been assuming that $y$ lies to the left of $x$, similar reasoning leads to the same conclusion when $y$ lies to the right. It follows, as explained at the beginning of the section, that
\[ \|\sigma_{yx} - 1\| \lesssim d(y, x) \]
over all $y, x \in \nom{H}$.
\subsubsection{The abelianized stable distributions are Lipschitz}
We know from Section~\ref{stable-lipschitz} that $E^\pm$ are Lipschitz. For any $y, x \in \nom{H}$, as discussed in Section~\ref{after-ab},
\begin{align*}
d_\angle(F^\pm_y, F^\pm_x) & = d_\angle(\sigma_{ay} E^\pm_y, \sigma_{ax} E^\pm_x) \\
& = d_\angle(\sigma_{ay} E^\pm_y, \sigma_{ay} \sigma_{yx} E^\pm_x).
\end{align*}
Proposition~\ref{expansion-bound} gives
\[ d_\angle(F^\pm_y, F^\pm_x) \le \|\sigma_{ya}\|^2\;d_\angle(E^\pm_y, \sigma_{yx} E^\pm_x). \]
Because distances in $\nom{H}$ are bounded, the result of the previous section ensures that $\|\sigma_{ya}\|$ is bounded as well. Therefore,
\begin{align*}
d_\angle(F^\pm_y, F^\pm_x) & \lesssim d_\angle(E^\pm_y, \sigma_{yx} E^\pm_x) \\
& \le d_\angle(E^\pm_y, E^\pm_x) + d_\angle(E^\pm_x, \sigma_{yx} E^\pm_x)
\end{align*}
over all $y, x \in \nom{H}$. Proposition~\ref{motion-bound} combines with the bound from the previous section to show that
\begin{align*}
d_\angle(E^\pm_x, \sigma_{yx} E^\pm_x) & \le \|\sigma_{yx} - 1\| \\
& \lesssim d(y, x).
\end{align*}
Meanwhile, the Lipschitz continuity of $E^\pm$ gives
\[ d_\angle(E^\pm_y, E^\pm_x) \lesssim d(y, x). \]
Therefore, altogether,
\[ d_\angle(F^\pm_y, F^\pm_x) \lesssim d(y, x) \]
over all $y, x \in \nom{H}$. In other words, the abelianized stable distributions $F^\pm$ are Lipschitz.
\subsection{The abelianized stable distributions are constant}\label{ab-const}
By construction, the values of the functions $F^\pm \maps \nom{H} \to \mathbf{P}F$ match at adjacent edge points, in the sense of Theorem~\ref{no-bulk-flex}. Since we just saw that $F^\pm$ are Lipschitz, Theorem~\ref{no-bulk-flex} tells us that $F^\pm$ are constant.

Globally, this means the stable distributions $\mathcal{F}^\pm$ are constant with respect to the local system $\mathcal{F}$. Thus, they reduce the structure group of $\mathcal{F}$ to the diagonal subgroup of $\on{SL}_2 \C$.
\section{A quick example}\label{toy-example}
\subsection{Overview}
Now that we've proven abelianization works, let's see an example of what it does. The calculations in this section aren't rigorous, but I'll try to indicate what it would take to make them rigorous.
\subsection{Setting the scene}
\subsubsection{A translation surface}\label{exmp-transl-surf}
Construct a torus with a translation structure by gluing the opposite sides of a parallelogram, inserting a singularity of cone angle $2\pi$ at the corner. For concreteness, let's fix the base of the parallelogram to be horizontal with length one, and set the height to be one as well. This leaves only one degree of freedom in the translation structure: the slope parameter $m$ labeled in the drawing below.
\begin{center}
\begin{tikzpicture}
\spiralparallelogram
\end{tikzpicture}
\end{center}
The torus has one forward-critical leaf and one backward-critical leaf. The drawing follows the critical leaves a little ways out from the singularity, so you can get an idea of how they wind around the surface. We'll assume $m$ is irrational, ensuring that neither critical leaf is a saddle connection.

The details of the computation depend on which way the parallelogram is leaning---a first hint of something like the cluster variety structure mentioned in Section~\ref{shear-intro}. To match the drawings, we'll show the work for the left-leaning case.
\subsubsection{A variety of local systems}
An $\on{SL}_2 \C$ local system on the torus minus the singularity is specified, up to isomorphism, by the group elements $A, B \in \on{SL}_2 \C$ that describe the parallel transport across the sides of the parallelogram, as shown in the drawing. Any values of $A$ and $B$ are possible.

Let's restrict ourselves to the dense open subset of the character stack in which $B$ has distinct eigenvalues. In this region, we can hit every isomorphism class using group elements of the form
\begin{align*}
A & = \left[ \begin{array}{cc} \mu & \rho \\ \rho & \nu \end{array} \right] &
B & = \left[ \begin{array}{cc} \lambda & \cdot \\ \cdot & \tfrac{1}{\lambda} \end{array} \right],
\end{align*}
where $|\lambda| < 1$. The ordinary points---the isomorphism classes of irreducible local systems---are the ones with $\mu \nu \neq 1$. Restricting further to the dense open subset in which $\mu \nu \notin (-\infty, 1]$, we can make $\rho$ a holomorphic function of $\mu$ and $\nu$ by noting that $\det A = 1$ and imposing the additional constraint $\Re \rho > 0$. This gives a holomorphic parameterization of a dense open subset of the character variety by the three variables $\mu$, $\nu$, and $\lambda$, which vary over the domain
\begin{align*}
\mu \nu \notin (-\infty, 1] & & |\lambda| < 1.
\end{align*}
\subsubsection{Reduction to an interval exchange}
Let $Z$ be the horizontal segment running across the middle of the parallelogram. Under the vertical flow, $Z$ sweeps out a simple flow box that covers almost the whole torus. The vertical edge of the flow box is non-critical, so we can compute the abelianized local system just by looking at the parallel transport cocycle over $Z$. For this purpose, we'll mostly carry on with the notation from Section~\ref{ab-conv}.
\begin{center}
\begin{tikzpicture}
\returnparallelogram
\end{tikzpicture}
\end{center}

Identify $Z$ with $(-1, 0)$. The first return relation $\alpha$ has a single break point, $b = -\tfrac{m}{2}$. Its inverse $\alpha^{-1}$ has break point $c = -1 + \tfrac{m}{2}$. Because $Z$ isn't well-cut, the break points aren't the only points where $\alpha$ and $\alpha^{-1}$ return nothing: $-m$ and $-1 + m$ vanish under the actions of $\alpha$ and $\alpha^{-1}$ as well. We aren't calling the latter break points because they don't lie on critical leaves.

The forward parallel transport cocycle is constant on the intervals
\begin{align*}
\left(-1, -m \vphantom{\tfrac{m}{2}}\right) & &  \left(-m, -\tfrac{m}{2}\right) & & \left(-\tfrac{m}{2}, 0\right),
\end{align*}
where it has the values
\begin{align*}
B & & A^{-1}B & & BA^{-1},
\end{align*}
respectively. You might be able to use the Markov cone field criterion from Section~\ref{construct-unif-hyp} to figure out which values of $\mu, \nu$ and $\lambda$ make the parallel transport cocycle uniformly hyperbolic, but it doesn't seem straightforward. Let's just assume from now on that we've chosen a uniformly hyperbolic local system.
\subsection{Abelianization}
\subsubsection{Approximation}
There isn't an obvious way to compute the abelianized local system exactly, but there is a pretty obvious way to approximate it when $m$ is tiny. As before, we'll show the work for the left-leaning case.
\subsubsection{The slithering jumps at the break points}
As $m$ approaches zero, the sequence of $\on{SL}_2 \C$ elements generated by repeatedly applying the forward parallel transport cocycle to $\rlane{b}$ approaches
\[ \ldots B, B, B, B, BA^{-1}, \]
in the sense that it takes more and more iterations to deviate from this sequence. Using Proposition~\ref{stable-line-appx}, you can deduce from this that the forward-stable line $E^+_{\rlane{b}}$ approaches the line spanned by
\[ A \left[ \begin{array}{c} 1 \\ 0 \end{array} \right] =
\left[ \begin{array}{c} \mu \\ \rho \end{array} \right] \sim
\left[ \begin{array}{c} 1 \\ \tfrac{\rho}{\mu} \end{array} \right] \]
Similarly, applying the forward parallel transport cocycle to $\llane{b}$ yields a sequence approaching
\[ \ldots B, B, B, B, A^{-1}B \]
as $m$ goes to zero, so $E^+_{\llane{b}}$ approaches the span of
\[ B^{-1}A \left[ \begin{array}{c} 1 \\ 0 \end{array} \right] =
\left[ \begin{array}{c} \tfrac{1}{\lambda}\,\mu \\ \lambda\,\rho \end{array} \right] \sim
\left[ \begin{array}{c} \hphantom{\lambda^2}\,1 \\ \lambda^2\,\tfrac{\rho}{\mu} \end{array} \right]. \]
On the other hand, applying the backward parallel transport cocycle to $b$ yields a sequence approaching
\[ \ldots B^{-1}, B^{-1}, B^{-1}, B^{-1}, B^{-1}, \]
so $E^-_{b}$ goes to the span of
\[ \left[ \begin{array}{c} 0 \\ 1 \end{array} \right]. \]
The slithering jump $s_{b}$ therefore approaches
\[ \left[ \begin{array}{cc}
\hphantom{(\lambda^2 - 1)} 1 & \cdot \\
(\lambda^2 - 1)\tfrac{\rho}{\mu} & 1
\end{array} \right] \]
as $m$ goes to zero. A similar computation shows that $s_c$ approaches
\[ \left[ \begin{array}{cc}
1 & (\lambda^2 - 1)\tfrac{\rho}{\mu} \\
\cdot & \hphantom{(\lambda^2 - 1)} 1
\end{array} \right] \]
in the same limit.
\subsubsection{The slithering jumps at all the critical points}
The slithering jumps at all the critical points can be deduced from the ones at the break points using the flow-invariance property discussed in Section~\ref{flow-invariance}. For each $n \ge 0$, the jump $s_{\alpha^{-n} b}$ at the forward-critical point $\alpha^{-n} b$ goes to
\[ B^{-n} s_b B^n =
\left[ \begin{array}{cc}
\hphantom{\lambda^{2n}(\lambda^2 - 1)} 1 & \cdot \\
\lambda^{2n}(\lambda^2 - 1)\tfrac{\rho}{\mu} & 1
\end{array} \right] \]
as $m$ goes to zero. The jump $s_{\alpha^n c}$ at the backward-critical point $\alpha^n c$ goes to
\[ B^n s_c B^{-n} = \left[ \begin{array}{cc}
1 & \lambda^{2n}(\lambda^2 - 1)\tfrac{\rho}{\mu} \\
\cdot & \hphantom{\lambda^{2n}(\lambda^2 - 1)} 1
\end{array} \right] \]
in the same limit.
\subsubsection{The slithering deviation}
Looking back at the drawing in Section~\ref{exmp-transl-surf}, you can see that forward-critical points $\alpha^{-n} b$ march from right to left across $Z$ as $n$ grows, while the backward-critical points $\alpha^n c$ march from left to right. As $m$ goes to zero, it takes longer and longer for the parades to meet at $-\tfrac{1}{2} \in Z$. If we can find a common bounding exponent for all the parallel transport cocycles we pass through, the bound from Section~\ref{bound-jump} will hold uniformly as $m$ goes to zero, so we shouldn't have to worry too much about the later jumps. We can therefore compute as though the parades never meet.

In this approximation, the deviation to $-\tfrac{1}{2}$ from $0$ is given by the product
\[ \cdots \quad s_{\alpha^{-3} b} \quad s_{\alpha^{-2} b} \quad s_{\alpha^{-1} b} \quad s_b, \]
and the deviation to $-1$ from $-\tfrac{1}{2}$ is given by
\[ s_c \quad s_{\alpha^1 c} \quad s_{\alpha^2 c} \quad s_{\alpha^3 c} \quad \cdots. \]
As $m$ goes to zero, the jumps $s_{\alpha^{-n} b}$ go to commuting shears, so you should be able to show that their product approaches
\[ \left[ \begin{array}{cc}
1 \hphantom{(\lambda^2 - 1) \sum_{n = 0}^\infty \lambda^{2n}} & \cdot \\
\tfrac{\rho}{\mu} (\lambda^2 - 1) \sum_{n = 0}^\infty \lambda^{2n} & 1
\end{array} \right] =
\left[ \begin{array}{cc}
\hphantom{-} 1 & \cdot \\
-\tfrac{\rho}{\mu} & 1
\end{array} \right] \]
Similarly, the product of the jumps $s_{\alpha^n c}$ should approach
\[ \left[ \begin{array}{cc}
1 & \tfrac{\rho}{\mu}(\lambda^2 - 1) \sum_{n = 0}^\infty \lambda^{2n} \\
\cdot & 1 \hphantom{(\lambda^2 - 1) \sum_{n = 0}^\infty \lambda^{2n}}
\end{array} \right] =
\left[ \begin{array}{cc}
1 & -\tfrac{\rho}{\mu} \\
\cdot & \hphantom{-} 1
\end{array} \right] \]

Now we know enough to approximate the holonomy $A_\text{ab}$ of the abelianized local system around the loop that starts at $-\tfrac{1}{2}$, runs left to $-1$, wraps around to $0$, and runs left back to $-\tfrac{1}{2}$. As $m$ goes to zero, this holonomy approaches
\[ \left[ \begin{array}{cc}
\hphantom{-} 1 & \cdot \\
-\tfrac{\rho}{\mu} & 1
\end{array} \right]
A
\left[ \begin{array}{cc}
1 & -\tfrac{\rho}{\mu} \\
\cdot & \hphantom{-} 1
\end{array} \right] =
\left[ \begin{array}{cc}
\mu & \cdot \\
\cdot & \tfrac{1}{\mu}
\end{array} \right]. \]
As expected, the abelianized holonomy preserves the stable lines $E^+_{-1/2}$ and $E^-_{-1/2}$, which approach
\[ \left[ \begin{array}{c} 1 \\ 0 \end{array} \right]
\quad \text{and} \quad
\left[ \begin{array}{c} 0 \\ 1 \end{array} \right]
\]
as $m$ goes to zero. In the limit, abelianization has no effect on the holonomy around a vertical loop, so $B_\text{ab}$ goes to $B$ as $m$ goes to zero.
\subsection{Limiting coordinates}
We've learned that on the translation torus constructed from a left-leaning parallelogram with slope parameter $m$, the abelianization of the local system
\begin{align*}
A & = \left[ \begin{array}{cc} \mu & \rho \\ \rho & \nu \end{array} \right] &
B & = \left[ \begin{array}{cc} \lambda & \cdot \\ \cdot & \tfrac{1}{\lambda} \end{array} \right]
\end{align*}
approaches
\begin{align*}
A_\text{ab} & = \left[ \begin{array}{cc} \mu & \cdot \\ \cdot & \tfrac{1}{\mu} \end{array} \right] &
B_\text{ab} & = \left[ \begin{array}{cc} \lambda & \cdot \\ \cdot & \tfrac{1}{\lambda} \end{array} \right]
\end{align*}
as $m$ goes to zero. For a right-leaning parallelogram, the analogous calculation shows that the abelianization approaches
\begin{align*}
A_\text{ab} & = \left[ \begin{array}{cc} \tfrac{1}{\nu} & \cdot \\ \cdot & \nu \end{array} \right] &
B_\text{ab} & = \left[ \begin{array}{cc} \lambda & \cdot \\ \cdot & \tfrac{1}{\lambda} \end{array} \right]
\end{align*}
as $m$ goes to zero.

As expected, the abelianized local system diagonalizes along the forward- and backward-stable lines of the original. Using the embedding
\begin{align*}
\C^\times & \to \on{SL}_2 \C \\
\xi & \mapsto \left[ \begin{array}{cc} \xi & \cdot \\ \cdot & \tfrac{1}{\xi} \end{array} \right],
\end{align*}
we can see it as a $\C^\times$ local system. Its limiting holonomies are
\begin{align*}
A_\text{ab} & = \mu &
B_\text{ab} & = \lambda
\end{align*}
in the left-leaning case, and
\begin{align*}
A_\text{ab} & = \tfrac{1}{\nu} &
B_\text{ab} & = \lambda
\end{align*}
in the right-leaning case. Looking at both limits, we can recover the holomorphic coordinates $\mu$, $\nu$, and $\lambda$ that we've been using to parameterize a dense open subset of the $\on{SL}_2 \C$ character variety.
\section{The shear parameterization}\label{shear-params}
\subsection{Overview}
In Section~\ref{toy-example}, we approximately abelianized a local system on a punctured torus by direct computation. We wrote down the slithering jumps and composed them, mostly in the right order, to find an approximation of the slithering deviation. Then we used that deviation to construct a diagonalizable local system.

In this section, we'll abelianize the local system of charts on a compact hyperbolic surface by working in the opposite direction. We'll use a construction from hyperbolic geometry to produce a diagonalizable $\on{SL}_2 \R$ local system, and a densely defined stalkwise isomorphism to the new local system from the local system of charts. Then we'll show that the deviation of that stalkwise isomorphism is the slithering deviation coming from a certain translation structure.

This example is important because it demonstrates the relationship between abelianization and the shear parameterization, summarized in Theorem~\ref{shear-access}. I'll show in Section~\ref{ab-vs-shears} that the abelianized local system of charts neatly encodes the shear parameters of the hyperbolic surface.

I expect this section to be most interesting to readers already familiar with the shear parameterization, so I won't pause to define technical terms. I'll try, though, to provide a few entry-level references.
\subsection{Collapsing a hyperbolic surface}\label{collapsing}
Take a compact hyperbolic surface $C_\text{hyp}$ and equip it with a measured maximal geodesic lamination $\mathcal{V}$, being sure to choose a lamination with no leaves that are closed geodesics~\cite{glam}. Build the corresponding horocyclic measured foliation $\mathcal{H}$ according to the plans in \cite[\S 3.1]{length-convexity}. Using the process Gupta outlines in \cite[\S 3.1]{asympt-of-grafting}, collapse the hyperbolic surface $C_\text{hyp}$ to a half-translation surface $C_\text{flat}$, producing a quotient map $G \maps C_\text{hyp} \to C_\text{flat}$ that sends $\mathcal{V}$ and $\mathcal{H}$ to the vertical and horizontal foliations of $C_\text{flat}$.\footnote{In \cite[\S 6.4]{warping-geom}, I call this process ``deflation,'' unaware of Gupta's prior terminology. Gupta mentions that $\mathcal{H}$ can be modified to cover the central regions of the triangles complementary to $\mathcal{V}$, but there's no need to do that before collapsing. Indeed, Gupta later mentions ``the central (unfoliated) region of each ideal triangle'' while describing the collapsing process. In \cite{asympt-of-grafting}, it doesn't matter, but here it will be best to leave $\mathcal{H}$ unmodified.} Under $G$, each non-edge leaf of $\mathcal{V}$ maps isometrically to a non-critical vertical leaf of $C_\text{flat}$, and each edge leaf maps isometrically to a pair of critical leaves joined at a singularity, as pictured below.\footnote{By the ``edge leaves,'' I mean the leaves bounding the ideal triangles that make up the complement of $\mathcal{V}$.}
\begin{center}
\begin{tikzpicture}[nodes={align=center}]
\matrix[row sep=2cm, column sep=2cm, nodes={outer sep=2mm}]{
\node (go) {\includegraphics{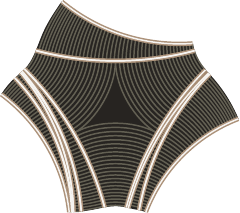}}; &
\node (gone) {\includegraphics{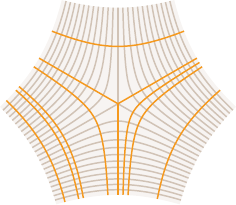}};
\\
};
\path[
  commutative diagrams/.cd, every arrow, every label,
  nodes={fill=white, inner sep=1mm, outer sep=0.5mm}
]
  (go) edge[swap] node {$G$} (gone);
\end{tikzpicture}
\end{center}
At the center of any ideal triangle, you can find an open triangular region bounded by horocycles around the vertices. I'll call this region the {\em contact triangle}, after its Euclidean analogue. The map $G$ sends each contact triangle in the complement of $\mathcal{V}$ to a singularity of $C_\text{flat}$.

Our procedure for dividing a translation surface, introduced in Section~\ref{dividing}, can be applied just as well to a half-translation surface. Let $\divd{C}_\text{flat}$ be the divided version of $C_\text{flat}$, and let $g \maps \mathcal{V} \to \divd{C}_\text{flat}$ be the unique continuous map that commutes with $G \maps C_\text{hyp} \to C_\text{flat}$ in the diagram below.
\begin{center}
\begin{tikzpicture}[nodes={align=center}]
\matrix[row sep=2cm, column sep=2cm, nodes={outer sep=2mm}]{
\node (lam) {\includegraphics{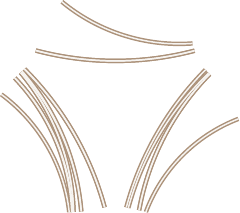}}; &
\node (going) {\includegraphics{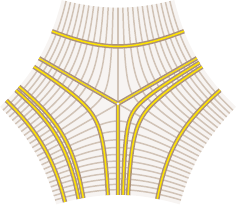}};
\\
\node (go) {\includegraphics{figures/fattened.pdf}}; &
\node (gone) {\includegraphics{figures/deflated.pdf}};
\\
};
\path[
  commutative diagrams/.cd, every arrow, every label,
  nodes={fill=white, inner sep=1mm, outer sep=0.5mm}
]
  (lam) edge node {$g$} (going)
  (going) edge node {$\pi$} (gone)
  (lam) edge[right hook->] node {} (go)
  (go) edge[swap] node {$G$} (gone);
\end{tikzpicture}
\end{center}
The divided surface $\divd{C}_\text{flat}$ is what you might imagine the collapsing hyperbolic surface $C_\text{hyp}$ looks like in the instant before it flattens out. The lanes of the critical leaves are the edges of the complementary triangles, about to fuse together. The medians that separate them are the last vestiges of the interiors of the triangles.
\subsection{The twisted local system of spin charts}
\subsubsection{Construction}\label{construct-spin-charts}
As I mentioned in Section~\ref{context}, the local isometries from $C_\text{hyp}$ to the hyperbolic plane form a $\on{PSL}_2 \R$ local system, which I'll call the ``local system of charts.'' It lifts, in a canonical way, to a twisted $\on{SL}_2 \R$ local system, which I'll call the ``twisted local system of spin charts''~\cite[\S 23]{curves-on-surfaces}\cite[\S 3.3.1]{warping-geom}. The twisted local system of spin charts will play a starring role in this section, so I'd like to describe it in some detail, building on the treatment in \cite{warping-geom}.

Throughout this section, we'll model the hyperbolic plane $\Hyp$ as the upper half-plane in $\Proj\C^2$. This identifies its ideal boundary with $\Proj \R^2$, and its isometry group with $\on{PSL}_2 \R$. Fix a base point in $U\Hyp$: the unit tangent vector at $(i, 1)$ pointing along the geodesic $(i\R_+, 1)$ in the $(i, 0)$ direction. The action of the isometry group on the base point identifies $U\Hyp$ with $\on{PSL}_2 \R$. We can then view $\on{SL}_2 \R$ as a circle bundle over $\Hyp$, with each fiber wrapped twice around the corresponding fiber of $U\Hyp$. I'll call it the ``unit spinor bundle'' of $\Hyp$, and write it as $W\Hyp$. Each unit spinor has a negative, which projects to the same unit tangent vector. The identity element of $\on{SL}_2 \R$, which projects to the base point of $U\Hyp$, will serve as a base point for $W\Hyp$. The action of the isometry group $\on{PSL}_2 \R$ on $U\Hyp$ lifts to an action of the ``spin isometry group'' $\on{SL}_2 \R$ on $W\Hyp$.

For each pair of unit tangent vectors $u \in UC_\text{hyp}$ and $v \in U\mathbb{H}$, there's a local isometry $C_\text{hyp} \to \mathbb{H}$ whose derivative sends $u$ to $v$, and it's unique up to restriction. The trivial $\on{PSL}_2 \R$ bundle $\underline{U\mathbb{H}} \to UC_\text{hyp}$, which I'll call $M$, thus parameterizes germs of charts on $C_\text{hyp}$. We can give $M$ a flat connection by declaring the derivative of each chart to be a flat section. For each $u \in UC_\text{hyp}$, let $m_u$ be the germ of the chart that sends $u$ to the base point of $U\Hyp$, defining a global section $m \maps UC_\text{hyp} \to M$ which is smooth but not flat.

Let $E$ be the trivial $\on{SL}_2 \R$ bundle $\underline{W\Hyp} \to UC_\text{hyp}$, equipped with the lift of the flat connection on $M$. Its flat sections, which I'll call ``spin charts,'' form a twisted local system on $C_\text{hyp}$. Up to restriction, there are two spin charts that send each $u \in UC_\text{hyp}$ to the base point in $U\Hyp$. Their germs appear in the fiber $E_u$ as the base point of $W\Hyp$ and its negative. I'll call the base point $e_u$. The section $e \maps UC_\text{hyp} \to E$ is a lift of the section $m \maps UC_\text{hyp} \to M$.
\subsubsection{Linearization}\label{linear-spin-charts}
In Section~\ref{showing-ab}, we'll need to view the twisted local system of spin charts as a linear local system. For this purpose, let $L$ be the trivial vector bundle $\underline{\R}^2 \to UC_\text{hyp}$. Beneath all our evocative language, $E$ is just the trivial bundle $\underline{\on{SL}_2 \R} \to UC_\text{hyp}$, so it's naturally identified with the bundle of linear maps $\on{SL}(L, \underline{\R^2})$. Place the flat connection on $L$ that gives $\on{SL}(L, \underline{\R}^2)$ the flat connection of $E$. Spin charts now correspond to flat sections of $\on{SL}(L, \underline{\R}^2)$. In the language of Section~\ref{lin-loc-sys}, we've identified the twisted local system of spin charts with the local system of flat $\on{SL}_2 \R$ structured frames in $L$.

To get a geometric understanding of $L$, consider paths on $C_\text{hyp}$ up to homotopies fixing their start and end points. Under the map that sends each path to its start point, they form a bundle whose fibers are universal covers of $C_\text{hyp}$. This bundle of paths comes with a natural flat connection, whose flat sections are families of paths that share the same end point. Pull it back to a flat bundle $P \to UC_\text{hyp}$. By extending the germs of charts parameterized by $M_u$ to global charts on the universal cover $P_u$, we can identify $M$ with the bundle of isometries $\on{Isom}(P, \Hyp)$, observing that its flat connection matches the one induced by the flat connection on $P$.

The section $m \maps UC_\text{hyp} \to \on{Isom}(P, \Hyp)$ trivializes $P$, and identifies the ideal boundary of each fiber $P_u$ with $\Proj\R^2$. Recalling that $L$ is $\underline{\R}^2$, we can interpret $\Proj L_u$ as the ideal boundary of $P_u$. The lines in $L_u$ generated by $(1, 0)$ and $(0, 1)$ are the forward and backward boundary points of the geodesic along $u$. Call them $L^+_u$ and $L^-_u$, respectively.

The isometries of $P_u$ act by inverse precomposition\footnote{Plain precomposition is a right action, which would conflict with the convention used throughout the paper that automorphisms act from the left. Inverting before precomposing resolves the conflict.} on $\on{Isom}(P_u, \Hyp)$, giving us a concrete model for the automorphisms of $M_u$. The volume-preserving linear automorphisms of $L_u$ act similarly on $\on{SL}(L_u, \R^2)$, providing a representation of the automorphisms of $E_u$. The abstract double covering $\Aut E_u \to \Aut M_u$ then appears concretely as the projection $\on{SL}(L_u) \to \on{PSL}(L_u)$, which tells us how $\on{SL}(L_u)$ acts on the ideal boundary of $P_u$.
\subsection{Untwisting the local system of spin charts}\label{untwisting}
The contact triangles form an open subset $B \subset C_\text{hyp}$. Their complement is a topological manifold with boundary. The unit tangent vectors perpendicular to the leaves of $\mathcal{H}$ form a double cover $\Sigma_\text{hyp} \to C_\text{hyp} \smallsetminus B$, embedded in $UC_\text{hyp}$ as a topological submanifold with boundary. The flat bundle $E$ pulls back to a flat bundle on $\Sigma_\text{hyp}$. Its sheaf of flat sections is an ordinary $\on{SL}_2 \R$ local system on $\Sigma_\text{hyp}$, which I'll call the ``untwisted local system of spin charts.''

Let's find the holonomy of the untwisted local system of spin charts around a boundary component of $\Sigma_\text{hyp}$. The loop shown below encloses a boundary component, traveling first on one sheet of $\Sigma_\text{hyp}$ and then on the other.
\begin{center}
\begin{tikzpicture}[nodes={align=center}]
\matrix[row sep=2cm, column sep=1cm, nodes={outer sep=2mm}]{
\node {\includegraphics{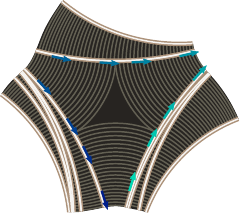}}; &
\node {\includegraphics{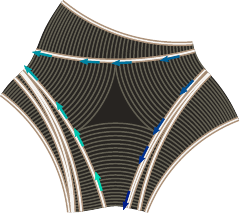}};
\\
};
\end{tikzpicture}
\end{center}
In $UC_\text{hyp}$, it contracts to a loop that runs once around a fiber.
\begin{center}
\includegraphics{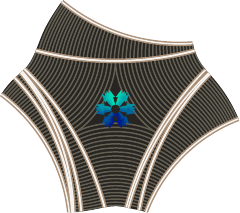}
\end{center}
The twisted local system of spin charts has holonomy $-1$ around the contracted loop, as the definition of a twisted local system demands. Hence, the untwisted local system of spin charts has holonomy $-1$ around each boundary component of $\Sigma_\text{hyp}$.

The resemblance between the double cover $\Sigma_\text{hyp} \to C_\text{hyp} \smallsetminus B$ and the translation double cover $\Sigma_\text{flat} \to C_\text{flat}$ is very strong. It suggests, for the quotient map $G \maps C_\text{hyp} \to C_\text{flat}$, a preferred lift $\Gamma \maps \Sigma_\text{hyp} \to \Sigma_\text{flat}$, whose action is illustrated below.
\begin{center}
\begin{tikzpicture}[nodes={align=center}]
\matrix[row sep=2cm, column sep=2cm, nodes={outer sep=2mm}]{
\node (go) {\includegraphics{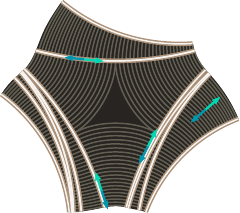}}; &
\node (gone) {\includegraphics{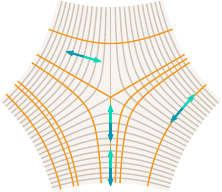}};
\\
};
\path[
  commutative diagrams/.cd, every arrow, every label,
  nodes={fill=white, inner sep=1mm, outer sep=0.5mm}
]
  (go) edge[swap] node {$\Gamma$} (gone);
\end{tikzpicture}
\end{center}
Because the untwisted local system of spin charts has non-identity holonomy around each boundary component of $\Sigma_\text{hyp}$, its pushforward to $\Sigma_\text{flat}$ isn't quite a local system: the stalks over the singularities $\mathfrak{B} \subset \Sigma_\text{flat}$ are empty. The pushforward does, however, restrict to a local system on $\Sigma_\text{flat} \smallsetminus \mathfrak{B}$. From there, it ascends to $\divd{\Sigma}_\text{flat}$ via the equivalence in Theorem~\ref{loc-sys-equiv-2d}. We've now repackaged the twisted local system of spin charts on $C_\text{hyp}$ as an ordinary local system $\mathcal{E}$ on $\divd{\Sigma}_\text{flat}$.

Just as $G \maps C_\text{hyp} \to C_\text{flat}$ commutes with a unique continuous map $g \maps \mathcal{V} \to \divd{C}_\text{flat}$, its lift $\Gamma \maps \Sigma_\text{hyp} \to \Sigma_\text{flat}$ commutes with a unique continuous map $\gamma \maps U\mathcal{V} \to \divd{\Sigma}_\text{flat}$, where $U\mathcal{V} \subset \Sigma_\text{hyp}$ consists of the vectors tangent to the leaves of $\mathcal{V}$. The map $\gamma$ is a homeomorphism onto its image, $\frakd{\Sigma}_\text{flat}$.

As I mentioned in Section~\ref{divd-props-2d}, dividing a translation surface opens up a tiny hole at each singularity. The local system $\mathcal{E}$ has holonomy $-1$ around each of these holes. You can see this by laying down a lily path, as defined in Appendix~\ref{lily-pads}, along the sequence of points in $\divd{\Sigma}_\text{flat}$ shown below.
\begin{center}
\begin{tikzpicture}[nodes={align=center}]
\matrix[row sep=2cm, column sep=1cm, nodes={outer sep=2mm}]{
\node {\includegraphics{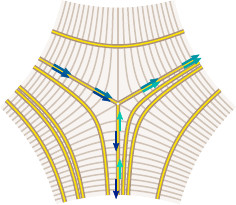}}; &
\node {\includegraphics{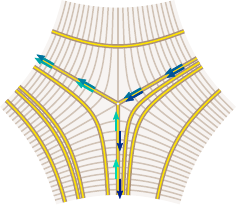}};
\\
};
\end{tikzpicture}
\end{center}
If you choose the lily path well, its preimage under $\gamma$ follows the loop we took earlier around a boundary component of $\Sigma_\text{hyp}$.
\subsection{Identifying germs of charts}\label{identifying-germs}
The local charts on the translation surface $\Sigma_\text{flat}$ are maps to $\R^2$, which form an $\R^2$ local system. The vertical components of local charts are maps to $\R$, which form an $\R$ local system. Let's identify $\R$ with the geodesic $(i\R_+, 1) \subset \mathbb{H}$ by following our base point for $U\mathbb{H}$ along the geodesic flow. The translations of $\R$ are represented by isometries of $\mathbb{H}$ through the embedding
\begin{align*}
\R & \to \on{SL}_2 \R \\
t & \mapsto \left[ \begin{array}{cc} e^{t/2} & \cdot \\ \cdot & e^{-t/2} \end{array} \right].
\end{align*}
Now we can view the vertical charts on $\Sigma_\text{flat}$ as maps to $\R \subset \mathbb{H}$, which form an $\on{SL}_2 \R$ local system whose structure group reduces to the diagonal subgroup.\footnote{From an abstract perspective, we're carrying out the local system version of the usual induced bundle construction~\cite{princ-buns}. From a $K$-set $X$ and a group embedding $K \hookrightarrow H$, we can induce an $H$-set $X_H = H \times_K X$. If $X$ is a $K$ torsor, $X_H$ will be an $H$ torsor, so we can use this operation to induce an $H$ local system from a $K$ local system.} Push this local system up to $\divd{\Sigma}_\text{flat}$ through the equivalence in Theorem~\ref{loc-sys-equiv-2d}, and call the result $\mathcal{F}$.

Let $\theta \maps \frakd{\Sigma}_\text{flat} \to U\mathcal{V}$ be the inverse of $\gamma$. At any point $p \in \frakd{\Sigma}_\text{flat}$, our construction of $\mathcal{E}$ suggests an identification of the stalk $\mathcal{E}_p$ with the fiber $E_{\theta p}$. As we saw in Section~\ref{construct-spin-charts}, $E_{\theta p}$ has a distinguished element $e_{\theta p}$---one of the two spin chart germs that sends $\theta p$ to the base point in $U\mathbb{H}$. The stalk $\mathcal{F}_p$ has a similar distinguished element: the germ $f_{\pi p}$ of the vertical chart that sends the upward unit tangent vector at $\pi p$ to the base point in $U\Hyp$, where $\pi \maps \divd{\Sigma}_\text{flat} \to \Sigma_\text{flat}$ is the usual projection. It seems natural to identify $\mathcal{E}_p$ with $\mathcal{F}_p$ through the isomorphism of $\on{SL}_2 \R$ torsors that sends $e_{\theta p}$ to $f_{\pi p}$. This defines a stalkwise isomorphism $\Upsilon \maps \mathcal{E} \to \mathcal{F}$, supported on $\frakd{\Sigma}_\text{flat}$.
\subsection{Showing we've abelianized the untwisted spin charts}\label{showing-ab}
At this point, we have a diagonalizable $\on{SL}_2 \R$ local system $\mathcal{F}$ and a densely defined stalkwise isomorphism $\Upsilon \maps \mathcal{E} \to \mathcal{F}$. Now we'll show that $\mathcal{F}$ and $\Upsilon$ form an abelianization of $\mathcal{E}$. In other words, we'll show that the deviation of $\Upsilon$ is the slithering deviation for $\mathcal{E}$.
\paragraph{The deviation of $\Upsilon$ across a critical road}
First, let's find the deviation of $\Upsilon$ across the median of a critical road. Let $\upsilon$ be the deviation of $\Upsilon$. Pick a point $w$ on a forward-critical leaf of $\Sigma_\text{flat}$, and a simple flow box $U \subset \divd{\Sigma}_\text{flat}$ containing $\hat{w}$. We'd like to work out $\upsilon^U_{\llane{w}\rlane{w}}$. Luckily, we already have all the ingredients for this calculation in place. We know from the definition of $\Upsilon$ that $\Upsilon_\llane{w}$ sends $e_{\theta \llane{w}}$ to $f_w \in \mathcal{F}_\llane{w}$, while $\Upsilon_\rlane{w}$ sends $e_{\theta \rlane{w}}$ to $f_w \in \mathcal{F}_\rlane{w}$. Though these two copies of $f_w$ live in different fibers of $\mathcal{F}$, they extend to the same element of $\mathcal{F}_U$. Hence, looking back to the definition of the deviation in Sections \ref{deviations-intuition} \thru \ref{dev-loc-const}, we can characterize $\upsilon^U_{\llane{w}\rlane{w}}$ as the automorphism of $\mathcal{E}_U$ which sends the spin chart with germ $e_{\theta \rlane{w}}$ to the one with germ $e_{\theta \llane{w}}$.

To compare $\upsilon^U_{\llane{w}\rlane{w}}$ with the slithering automorphism $s_w$ defined in Section~\ref{slithering}, we need to view $\mathcal{E}$ as a linear local system. Take the flat vector bundle $L$ and the bundle of paths $P$ from Section~\ref{linear-spin-charts} and turn their local systems of flat sections into local systems $\mathcal{L}$ and $\mathcal{P}$ on $\divd{\Sigma}_\text{flat}$, reprising the process we used to turn $E$ into $\mathcal{E}$. Now we can identify $\mathcal{E}_U$ with the torsor of frames $\on{SL}(\mathcal{L}_U, \R^2)$.

To get a geometric understanding of the vector space $\mathcal{L}_U$, let's unpack the definition of $\mathcal{P}_U$. Let $\iota \maps \Sigma_\text{flat} \to \divd{\Sigma}_\text{flat}$ be the usual embedding. The preimage $\Gamma^{-1} \iota^{-1} U$ is an open subset of $\Sigma_\text{hyp}$ that doesn't touch the boundary. Its image $V$ under the covering map $\Sigma_\text{hyp} \to C_\text{hyp} \smallsetminus B$ is therefore an open subset of $C_\text{hyp}$. Shrink the flow box $U$ until the projection $\divd{\Sigma}_\text{flat} \to \divd{C}_\text{flat}$ is injective on it. Then $\Gamma^{-1} \iota^{-1} U$ covers $V$ homeomorphically. Because $\iota^{-1} U$ is a flow box, bounded by vertical and horizontal leaves, $V$ is bounded as shown below by leaves of $\mathcal{V}$ and $\mathcal{H}$.
\begin{center}
\begin{tikzpicture}[nodes={align=center}]
\matrix[row sep=2cm, column sep=2cm, nodes={outer sep=2mm}]{
\node (fat) {\includegraphics[width=5cm]{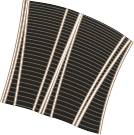}}; &
\node (divd) {\includegraphics[width=4cm]{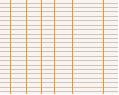}};
\\
};
\path[
  commutative diagrams/.cd, every arrow, every label,
  nodes={fill=white, inner sep=1mm, outer sep=0.5mm}
]
  (fat) edge[swap] node {$G$} (divd);
\end{tikzpicture}
\end{center}
Recalling the definition of the flat connection on $P$, we can interpet $\mathcal{P}_U$ as the space of paths on $C_\text{hyp}$ starting in $V$, up to homotopies fixing their end points and keeping their start points in $V$. From this point of view, $\mathcal{P}_U$ is a universal cover of $C_\text{hyp}$, with a natural inclusion $V \hookrightarrow \mathcal{P}_U$. Our interpretation of $\Proj L_u$ as the ideal boundary of $P_u$ is consistent with the flat connections on $L$ and $P$, so we can interpret $\Proj \mathcal{L}_U$ as the ideal boundary of $\mathcal{P}_U$.

When we look at $\mathcal{E}_U$ as $\on{SL}(\mathcal{L}_U, \R^2)$, the torsor automorphisms of $\mathcal{E}_U$ appear as the volume-preserving linear automorphisms of $\mathcal{L}_U$. We can find the element of $\on{SL}(\mathcal{L}_U)$ that represents $\upsilon^U_{\llane{w}\rlane{w}}$ by watching how $\upsilon^U_{\llane{w}\rlane{w}}$ acts on the ideal boundary of $\mathcal{P}_U$. The automorphism of $\on{SL}(\mathcal{L}_U, \R^2)$ that sends $e_{\theta \rlane{w}}$ to $e_{\theta \llane{w}}$ projects to the automorphism of $\on{Isom}(\mathcal{P}_U, \Hyp)$ that sends $m_{\theta \rlane{w}}$ to $m_{\theta \llane{w}}$. The corresponding isometry of $\mathcal{P}_U$, which acts on $\on{Isom}(\mathcal{P}_U, \Hyp)$ by inverse precomposition, sends $\theta \rlane{w}$ to $\theta \llane{w}$. The action of $\upsilon^U_{\llane{w}\rlane{w}}$ on $\Proj \mathcal{L}_U$ must therefore send $L^\pm_{\theta \rlane{w}}$ to $L^\pm_{\theta \llane{w}}$.

The geodesics along $\theta \rlane{w}$ and $\theta \llane{w}$ form the sides of an ideal triangle in $\mathcal{V}$, as shown below. They therefore share the same backward boundary point.
\begin{center}
\begin{tikzpicture}[nodes={align=center}]
\matrix[row sep=2cm, column sep=2cm, nodes={outer sep=2mm}]{
\node (fat) {\includegraphics[width=5cm]{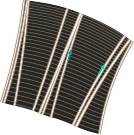}}; &
\node (divd) {\includegraphics[width=4cm]{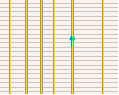}};
\\
};
\path[
  commutative diagrams/.cd, every arrow, every label,
  nodes={fill=white, inner sep=1mm, outer sep=0.5mm}
]
  (fat) edge[swap] node {$\gamma$} (divd);
\end{tikzpicture}
\end{center}
In other words, $L^-_{\theta \rlane{w}}$ and $L^-_{\theta \llane{w}}$ extend to the same line in $\mathcal{L}_U$, which I'll call $\mathcal{L}^-_w$. Now we know that $\upsilon^U_{\llane{w}\rlane{w}}$ sends $L^+_{\theta \rlane{w}}$ to $L^+_{\theta \llane{w}}$ and preserves $\mathcal{L}^-_w$. In fact, because $\theta \rlane{w}$ and $\theta \llane{w}$ lie on the same horocycle through the boundary point $\mathcal{L}^-_w$ represents, $\upsilon^U_{\llane{w}\rlane{w}}$ must act on $\mathcal{L}^-_w$ by $\pm 1$. This determines the action of $\upsilon^U_{\llane{w}\rlane{w}}$ on $\mathcal{L}_U$ up to sign.

To pin $\upsilon^U_{\llane{w}\rlane{w}}$ down precisely, consider a unit tangent vector $v$ sliding from $\theta \rlane{w}$ to $\theta \llane{w}$ along the horocycle mentioned above, staying perpendicular to the horocycle at all times. It traces out the arc $\Gamma^{-1} w$ in $UV$. As it moves, the automorphism of $\mathcal{E}_U$ sending $e_{\theta \rlane{w}}$ to $e_v$ changes continuously from the identity to $\upsilon^U_{\llane{w}\rlane{w}}$, acting on $\mathcal{L}^-_w$ by $\pm 1$ at all times. The identity, of course, acts on $\mathcal{L}^-_w$ by $1$, so $\upsilon^U_{\llane{w}\rlane{w}}$ must do the same.
\paragraph{Comparison with the slithering jump}
When we introduced the slithering jump $s$ in Section~\ref{slithering}, we defined $s_w$ by its action on the dynamically determined stable lines $\mathcal{E}^+_\llane{w}, \mathcal{E}^-_w, \mathcal{E}^+_\rlane{w}$. Let's find those lines in $\mathcal{L}_U$. Let $A$ be the parallel transport cocycle of $\mathcal{E}$ over $U$. Pick a point $p \in U$ not lying on a critical leaf, and apply $A^1_p, A^2_p, A^3_p, \ldots$ to the spin chart on $V$ whose germ is $e_{\theta p}$. This yields a sequence of spin charts whose images in $\Hyp$ hop forward along the geodesic generated by the base point of $U\Hyp$. If we view the automorphisms of $\mathcal{E}_U$ as elements of $\on{SL}(\mathcal{L}_U)$, acting on $\on{SL}(\mathcal{L}_U, \R^2)$ by inverse precomposition, we can say equivalently that the actions of $(A^1_p)^{-1}, (A^2_p)^{-1}, (A^3_p)^{-1}, \ldots$ on $\mathcal{P}_U$ skip $V$ forward along the geodesic generated by $\theta p$. The line $L^+_{\theta p}$ representing the forward boundary point of the geodesic therefore contracts exponentially under the actions of $A^1_p, A^2_p, A^3_p, \ldots$\,. That means $\mathcal{E}^+_p$ is $L^+_{\theta p}$. Generalizing this reasoning to the critical leaves, we see that $\mathcal{E}^+_\llane{w}, \mathcal{E}^-_w, \mathcal{E}^+_\rlane{w}$ are the lines $L^+_{\theta \llane{w}}$, $\mathcal{L}^-_w$, $L^+_{\theta \rlane{w}}$, which stand for the vertices of the ideal triangle pictured above.

By definition, in light of the above, $s_w$ sends $L^+_{\theta \rlane{w}}$ to $L^+_{\theta \llane{w}}$, acting by $1$ on $\mathcal{L}^-_w$. Earlier, we characterized $\upsilon^U_{\llane{w}\rlane{w}}$ in the same way, proving that $\upsilon^U_{\llane{w}\rlane{w}} = s_w$ for any point $w$ on a forward-critical leaf of $\Sigma_\text{flat}$. On a backward-critical leaf, the same arguments work, and the same conclusion holds.
\paragraph{Going off-road}
So far, we've shown that $\upsilon$ matches the slithering deviation $\sigma$ when you look at points on opposite sides of a critical road. Now we need to extend that result to all pairs of points in $\nom{U}$. Since $\upsilon$ and $\sigma$ are both the identity for pairs of points lying on the same vertical leaf, we just need to show that $\upsilon^U_{qp} = \sigma^U_{qp}$ for any two points $q, p \in \nom{U}$ lying on the same horizontal slice $Z \subset U$, with $q$ to the left of $p$. Recall from Section~\ref{jumps-converge}---referring back to Section~\ref{jump-concept} as needed---that $\sigma^U_{qp}$ is defined as the product
\[ \prod_{\median{w} \in (q \mid p)^U} s_w \]
over the points between $q$ and $p$ where the medians of the critical roads pierce the horizontal slice $Z$.

We know from the construction of $\Sigma_\text{flat}$ that $\theta$ sends points on the same horizontal slice of $\nom{U}$ to unit tangent vectors located on the same leaf of the horocyclic foliation. Consider a unit tangent vector sliding from $\theta p$ to $\theta q$ along the horocyclic foliation, staying perpendicular to the foliation at all times. It traces out a path $Y$ in $UV$. As we noted earlier, the action of $s_w$ on $\mathcal{P}_U$ connects the endpoints of the arc $\Gamma^{-1} w$, which lies in $Y$. If you put together all the arcs $\Gamma^{-1} w$ for $\median{w} \in (q \mid p)^U$, you get a full-measure subset of $Y$, suggesting that the product defining $\sigma^U_{qp}$ ought to converge to $\upsilon^U_{qp}$. We'll complete our proof that $\sigma = \upsilon$ by formalizing that argument.

To show that
\[ \upsilon^U_{qp} = \prod_{\median{w} \in (q \mid p)^U} s_w, \]
we need to show that if the product on the right-hand side is restricted to a finite subset $S \subset (q \mid p)^U$ of the medians separating $q$ from $p$, we can get it arbitrarily close to the left-hand side by enlarging $S$. For convenience, I'll refer to the arcs $\Gamma^{-1} w$ with $\hat{w} \in S$ as ``the bites.'' Removing the bites from $Y$ leaves a finite sequence of curves, which I'll call ``the scraps.''
\begin{center}
\begin{tikzpicture}
\matrix[row sep=0.2cm, column sep=0.5cm]{
\node {\includegraphics[width=5cm]{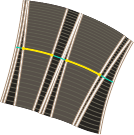}}; &
\node {\includegraphics[width=5cm]{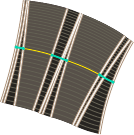}};
\\
\node {\small The bites, projected to $V$}; &
\node {\small The scraps, projected to $V$};
\\
};
\end{tikzpicture}
\end{center}
One way to remove the bites is to literally cut the ideal wedges containing them out of $\mathcal{P}_U$. You can glue $\mathcal{P}_U$ back together by applying the automorphisms $\{s_w^{-1}\}_{\hat{w} \in S}$ in reverse order, with $s_w^{-1}$ acting only on the part of $\mathcal{P}_U$ that comes after the wedge containing $\Gamma^{-1} w$.
\begin{center}
\begin{tikzpicture}
\matrix[row sep=0.2cm, column sep=0.5cm]{
\node {\includegraphics[width=5cm]{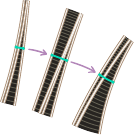}};
\\
\node {\small The actions of $s_w^{-1}$ for $\hat{w} \in S$};
\\
};
\end{tikzpicture}
\end{center}
This regluing combines the scraps into a new path $Y_S \subset UV$, which runs from $p$ to a new end point $q_S$. You should be able to convince yourself that the automorphism
\[ y_S = \left(\prod_{\median{w} \in S} s_w\right)^{-1} \upsilon^U_{qp}. \]
sends $e_p$ to $e_{q_S}$.

We want to prove we can bring $y_S$ arbitrarily close to the identity by enlarging $S$. We can do it with the help of a well-chosen metric on $\on{SL}(\mathcal{L}_U)$. Take the Sasaki metric on $U\mathcal{P}_U$~\cite{tangent-metrics}, pull it back along the action of $\on{Isom}(\mathcal{P}_U)$ on $p$, and then lift it to a metric on $\on{SL}(\mathcal{L}_U)$. The distance of $y_S$ from the identity in this metric is bounded by twice the arc length of $Y_S$, because $Y_S$ has unit curvature almost everywhere. The arc length of $Y_S$ is the total arc length of the scraps. We can therefore bring $y_S$ as close as we want to the identity by taking more bites. It follows, by our reasoning above, that the product defining $\sigma^U_{qp}$ converges to $\upsilon^U_{qp}$. This completes our proof that $\sigma = \upsilon$.
\subsection{Abelianization and the shear parameterization}\label{ab-vs-shears}
In Sections \ref{collapsing} \thru \ref{untwisting}, we used a measured maximal geodesic lamination on a hyperbolic surface $C_\text{hyp}$ to ``untwist'' the twisted local system of spin charts, repackaging its geometric information in an $\on{SL}_2 \R$ local system $\mathcal{E}$ on a divided surface $\divd{\Sigma}_\text{flat}$. In Sections \ref{identifying-germs} \thru \ref{showing-ab}, we found a stalkwise isomorphism between $\mathcal{E}$ and the local system $\mathcal{F}$ of vertical charts on $\Sigma_\text{flat}$, and showed that it abelianizes $\mathcal{E}$. We'll now see that $\mathcal{F}$ neatly encodes the shear parameters of $C_\text{hyp}$, so abelianization encompasses the shear parameterization.

First, let's review the shear parameterization, as described by Papadopoulos and Th\'{e}ret~\cite{teich-thurs}. A shear parameter is a number associated with a {\em horogeodesic curve} on $C_\text{hyp}$: a continuous curve built from finitely many segments of the leaves of $\mathcal{H}$ and $\mathcal{V}$. If you walk along a horogeodesic curve, you'll make a quarter-turn each time you cross from one segment to another. We'll say a horogeodesic curve is {\em non-backtracking} if it never makes two left turns or two right turns in a row. A non-backtracking horogeodesic curve has a {\em signed length}, obtained by measuring the lengths of its geodesic segments and summing them with the signs shown below.
\begin{center}
\begin{tikzpicture}
\matrix[row sep=0.2cm, column sep=1cm]{
\node {\includegraphics[width=5cm]{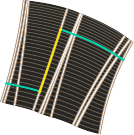}}; &
\node {\includegraphics[width=5cm]{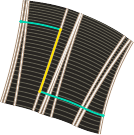}};
\\
\node {\small Positive (turn left on entry)}; &
\node {\small Negative (turn right on entry)};
\\
};
\end{tikzpicture}
\end{center}
Making the convention that a homotopy of non-backtracking horogeodesic curves must travel through the family of such curves, we can say that signed length is invariant under homotopy. A {\em shear parameter} is the signed length of a non-backtracking horogeodesic curve whose endpoints are the corners of contact triangles.

A non-backtracking horogeodesic curve is determined up to homotopy by its endpoints and the ordered set of leaves of $\mathcal{V}$ it passes through. Since $\mathcal{V}$ can be specified in a purely topological way, without reference to the hyperbolic structure of $C_\text{hyp}$~\cite[\S 5.3.1]{warping-geom}, the shear parameters can be labeled under the same restrictions. That means you can think of the shear parameters as variables whose values depend on the hyperbolic structure of $C_\text{hyp}$. In fact, various subsets of the shear parameters form coordinate systems for the space of hyperbolic structures.

To extract the values of the shear parameters from $\mathcal{F}$, first recall that Gupta's collapsing map $G \maps C_\text{hyp} \to C_\text{flat}$ sends the leaves of $\mathcal{H}$ and $\mathcal{V}$ to the horizontal and vertical leaves of $C_\text{flat}$, respectively. Hence, $G$ projects each horogeodesic curve on $C_\text{hyp}$ to what I'll call a {\em taxicab curve}: a continuous curve built from finitely many segments of the horizontal and vertical leaves of $C_\text{flat}$. The definitions of {\em non-backtracking} and {\em signed length} adapt straightforwardly to taxicab curves, and projection along $G$ preserves both properties. Since $G$ collapses each contact triangle of $C_\text{hyp}$ to a singularity of $C_\text{flat}$, we can say a shear parameter is the signed length of a taxicab curve whose endpoints are singularities.

A taxicab curve $Z$ on $C_\text{flat}$ has two lifts to $\Sigma_\text{flat}$. If $Z$ is non-backtracking, and it has at least one horizontal segment, you can orient both lifts so that their horizontal segments are oriented rightward. With this choice of orientation, a vertical segment of $Z$ is positively signed if its lift is oriented upward, and negatively signed if its lift is oriented downward. Hence, if you walk forward along either lift of $Z$, the net vertical distance you travel will be the signed length of $Z$.

If $Z$ has its endpoints at singularities, where the sheets of $\Sigma_\text{flat}$ meet, its lifts join up into an oriented loop on $\Sigma_\text{flat}$. By the reasoning above, if you measure the holonomy around this loop of the $\R$ local system of vertical charts on $\Sigma_\text{flat}$, you get twice the signed length of $Z$. The shear parameters of $C_\text{hyp}$ are thus encoded in the holonomies of the local system of vertical charts on $\Sigma_\text{flat}$.\footnote{For a different picture of how the holonomies of vertical charts on $\Sigma_\text{flat}$ encode the signed lengths of taxicab curves connecting singularities of $C_\text{flat}$, see the discussion in Section~6.3.2 of \cite{warping-geom}, which carries over to compact surfaces without much fuss.} The abelianized local system $\mathcal{F}$ is a straightforward repackaging of the local system of vertical charts, so its holonomies encode the shear parameters as well.
\subsection{A numerical note}\label{numerical-note}
Once you've made your way through the subtleties of convincing yourself that it works, abelianization is a very practical numerical tool for finding the shear parameters of hyperbolic surfaces. Computationally, abelianization amounts to:
\begin{enumerate}
\item Listing the critical points of an interval exchange transformation.
\item Finding the forward- and backward-stable lines of a Markov $\on{SL}_2 \C$ cocycle over that interval exchange.
\item Computing the jump automorphism at each critical point.
\item Composing the jumps in order.
\end{enumerate}
These steps are all well-suited to numerical approximation, and they're implemented in a software package developed in conjunction with this paper~\cite{whorl}. In the application described here, the interval exchange encodes the topology of the surface $C_\text{hyp}$ and its measured maximal geodesic lamination $\mathcal{V}$, and the Markov cocycle encodes the hyperbolic structure.
\appendix
\section{Technical tools for warping local systems}\label{warping-tech}
\subsection{The lily pad lemma}\label{lily-pads}
\begin{lemma}
Suppose $\mathcal{U}$ is an open cover of a connected space. For any two points $a$ and $b$ in the space, there is a finite sequence of elements of $\mathcal{U}$ in which the first element contains $a$, the last element contains $b$, and every element intersects the next one.
\end{lemma}
\begin{proof}
Let's call a finite sequence of elements of $\mathcal{U}$ a {\em lily path} if every element intersects the next one. We'll say two points $a$ and $b$ can be ``connected by a lily path'' if there's a lily path whose first element contains $a$ and whose last element contains $b$.

A lily path connecting $a$ to $b$ also connects $a$ to every other point in the last element of the path, which is an open neighborhood of $b$. Hence, the set of points that can be connected to $a$ by a lily path is open.

On the other hand, suppose $b$ can't be connected to $a$ by a lily path. Since $\mathcal{U}$ is a cover, there's some $U \in \mathcal{U}$ containing $b$. If a point in $U$ could be connected to $a$ by a lily path, adding $U$ to the end of that path would give a lily path connecting $a$ to $b$. Hence, no point in $U$ can be connected to $a$ by a lily path. Therefore, the set of points that can't be connected to $a$ by a lily path is also open.
\end{proof}
\subsection{Collapsing downward-directed colimits}\label{iso-colim}
\begin{prop}
Suppose $\Lambda$ is a downward-directed set, $C$ is some category, and $F \maps \Lambda \to C$ is a diagram in which all the arrows are isomorphisms. Then $F$ has a colimit, and the defining arrows from the diagram to its colimit are isomorphisms.
\end{prop}
\begin{proof}
Pick any object $s$ of $\Lambda$. For any other object $t$, we can get an isomorphism $f_t \maps F(t) \to F(s)$ by picking a common lower bound $\check{t}$ of $s$ and $t$ and taking the composition $F(\check{t} \leftarrow s)^{-1} F(\check{t} \leftarrow t)$. The isomorphism we get doesn't depend on our choice of $\check{t}$.

Now, pick any object $c$ of $C$. Suppose that for each object $t$ of $\Lambda$, we have an arrow $\phi_t \maps F(t) \to c$, and these arrows commute with the arrows of the diagram $F$. Observing that
\begin{align*}
\phi_s f_t & = \phi_s F(\check{t} \leftarrow s)^{-1} F(\check{t} \leftarrow t) \\
& = \phi_{\check{t}} F(\check{t} \leftarrow t) \\
& = \phi_t,
\end{align*}
we see that the object $F(s)$, equipped with the isomorphisms $f_t \maps F(t) \to F(s)$, is a colimit of the diagram $F$.
\end{proof}
\section{Relational dynamics}\label{rel-dyne}
Some dynamical systems, including vertical flows on singular translation surfaces, interval exchanges, and even the humble doubling map, are discontinuous if you insist on defining them at every point. (These particular examples are discussed in Sections \ref{tras-surfs} and \ref{divd-ex}.) You can recover a kind of continuity, however, if you describe the dynamics using relations rather than maps, allowing points to get lost as they fall into singularities, breaks between intervals, or whatever.

Consider a relation $\phi$ between topological spaces $Y$ and $X$. Define
\begin{align*}
\phi A = \{y \in Y : y \mathbin{\phi} a \text{ for some } a \in A\} \\
B \phi = \{x \in X : b \mathbin{\phi} x \text{ for some } b \in B\}
\end{align*}
for subsets $A \subset X$ and $B \subset Y$. For convenience, we'll relax the distinction between singletons and points, denoting $\phi\{x\}$, for example, by $\phi x$. If $\phi$ is a function $Y \leftarrow X$, then $\phi x \in Y$ is the value of $\phi$ at a point $x \in X$, and $B\phi \subset X$ is the preimage of a subset $B \subset Y$.

Define $\phi$ to be {\em injective} if $y \phi$ is a singleton for all $y \in Y$, {\em coinjective} if $\phi x$ is a singleton for all $x \in X$, and {\em biinjective} if it's both injective and coinjective. In less baroque language, a coinjective relation is just a partially defined function.

Define $\phi$ to be {\em continuous} if $V \phi$ is open whenever $V \subset Y$ is open, {\em cocontinuous} if $\phi U$ is open whenever $U \subset X$ is open, and {\em bicontinuous} if it's both continuous and cocontinuous. If $\phi$ is a function, ``continuous'' means what it usually means, and ``cocontinuous'' means ``open.'' Local systems can be pushed forward and backward along a bicontinuous relation.

A {\em flow by bicontinuous relations} on $X$ can be defined as a relation $\psi$ between $X$ and $\R \times X$ with the following properties:
\begin{itemize}
\item As a whole, $\psi$ is continuous and coinjective.
\item For all $t \in \R$, the relation $\psi^t = \blankbox \mathbin{\psi} (t, \blankbox)$ is bicontinuous and biinjective.
\item For all $t, s \in [0, \infty)$, we have $\psi^t \psi^s = \psi^{t + s}$ and $\psi^{-t} \psi^{-s} = \psi^{-t - s}$.
\end{itemize}
This kind of partially defined flow acts a lot like an ordinary flow by homeomorphisms. In particular, if $X$ carries a local system $\mathcal{E}$, it gives a parallel transport morphism $\mathcal{E}_{\psi^t U} \leftarrow \mathcal{E}_U$ for every open subset $U$ of $X$ and every time $t$.
\section{Infinite ordered products}\label{ord-prod}
\subsection{Definition}
Suppose $M$ is a Hausdorff topological monoid, like the one formed by the endomorphisms or automorphisms of a finite-dimensional vector space, and $A$ is a totally ordered set. Given a function $m \maps A \to M$, we'd like to make sense of the potentially infinite ordered product
\[ \prod_{p \in A} m_p, \]
which I'll refer to in writing as ``the product of $m$ over $A$.''

Recall that a function from a directed set $\Lambda$ into a topological space is called a {\em net}. For any $s \in \Lambda$, let's call the set $\{t \in \Lambda : t \ge s\}$ the {\em shadow} of $s$. A net is said to {\em converge} to a point if, for every open neighborhood $\Omega$ of that point, there is some element of $\Lambda$ whose shadow is sent by the net into $\Omega$. A net into a Hausdorff space, like $M$, converges to at most one point.

The finite subsets of $A$ form a directed set under inclusion, and the product of $m$ over any finite subset is well-defined, so we can define the product of $m$ over $A$ to be the limit of the net that sends each finite subset of $A$ to the product of $m$ over that subset. I'll call this net the ``product net'' for short.
\subsection{Calculation}\label{ord-prod-calculation}
Suppose $\Lambda'$ and $\Lambda$ are directed sets, $n$ is a net on $\Lambda$, and $f \maps \Lambda' \to \Lambda$ is an order-preserving map whose image intersects the shadow of every element of $\Lambda$. In these circumstances, the net $nf$ is called a {\em subnet} of $n$. If $n$ converges to a certain point, every subnet of $n$ converges to that point.

In particular, let $A$ be a totally ordered set, and $\Lambda$ its finite subsets. If $f \maps \N \to \Lambda$ is an increasing sequence of subsets whose union is all of $A$, the image of $f$ intersects the shadow of every finite subset. Thus, if a product over $A$ converges, we can find it by looking at partial products over any sequence of finite subsets whose union is $A$. Such a sequence must exist if $A$ is countable.
\subsection{Composition}
Given two totally ordered sets $B$ and $A$, let $B \sqcup A$ be the disjoint union of $B$ and $A$ with the total order that makes the inclusions order-preserving and puts every element of $B$ to the left of every element of $A$.
\begin{prop}\label{prod-compose}
Say we have a function $m \maps B \sqcup A \to M$. If the products of $m$ over $B$ and $A$ converge, then the product over $B \sqcup A$ converges too, and
\[ \prod_{p \in B \sqcup A} m_p = \left(\prod_{p \in B} m_p\right)\left(\prod_{p \in A} m_p\right). \]
\end{prop}
\begin{proof}
Let $\beta$ and $\alpha$ be the products of $m$ over $B$ and $A$, respectively. Given an open neighborhood $\Omega$ of $\beta\alpha$, we want to find a finite subset of $B \sqcup A$ whose shadow is sent by the product net into $\Omega$.

Using the fact that multiplication in $M$ is continuous, find open neighborhoods $\Omega_B$ of $\beta$ and $\Omega_A$ of $\alpha$ with $\Omega_B \Omega_A \subset \Omega$. Then, find finite subsets $S_B$ and $S_A$ of $B$ and $A$ whose shadows are sent into $\Omega_B$ and $\Omega_A$, respectively.

For any finite subset $R$ of $B \sqcup A$ containing $S_B \cup S_A$, we can use the fact that $B < A$ to rewrite the product of $m$ over $R$ as
\[ \left(\prod_{p \in R \cap B} m_p\right)
\left(\prod_{p \in R \cap A} m_p\right). \]
Observing that $R \cap B$ contains $S_B$ and $R \cap A$ contains $S_A$, we conclude that the product of $m$ over $R$ is in $\Omega$.
\end{proof}
\subsection{Equivariance}
\begin{prop}\label{prod-equivar}
If the product of $m \maps A \to M$ over $A$ converges,
\[ \phi \left( \prod_{p \in A} m_p \right) = \prod_{p \in A} \phi(m_p) \]
for any continuous homomorphism $\phi \maps M \to M$.
\end{prop}
\begin{proof}
Given an open neighborhood $\Omega$ of the left-hand side, we want to find a finite subset $S$ of $A$ with the property that
\[ \prod_{p \in R} \phi(m_p) \in \Omega \]
for all finite subsets $R \subset A$ containing $S$.

Since $\phi$ is continuous, $\phi^{-1}(\Omega)$ is open neighborhood of the product of $m$ over $A$. By the definition of convergence, we can find a finite subset $S$ of $A$ with the property that
\[ \prod_{p \in R} m_p \in \phi^{-1}(\Omega) \]
for all finite subsets $R \subset A$ containing $S$. Applying $\phi$ to both sides, we see that $S$ is just what we wanted.
\end{proof}
\subsection{Inversion}
Instead of just a map into a topological monoid, suppose we have a map $g \maps A \to G$ into a topological group. Let $A^{\op}$ be $A$ with the opposite order.
\begin{prop}
If the product of $g$ over $A$ converges,
\[ \prod_{p \in A^{\op}} g_p^{-1} = \left(\prod_{p \in A} g_p\right)^{-1}. \]
\end{prop}
\begin{proof}
Analogous to the proof of Proposition~\ref{prod-equivar}.
\end{proof}
We won't use this result for anything. It's only here to reassure you that the directionality of our construction of deviations from jumps in Section~\ref{jump-concept} doesn't introduce any actual asymmetry.
\subsection{Convergence}
Any Banach algebra, like $\End \C^2$ with the operator norm, can be thought of as a topological monoid by forgetting the addition. Since all the ordered products we care about will be taken in $\on{SL}_2 \C$, a closed submonoid of $\End \C^2$, understanding ordered products in a Banach algebra will be very helpful to us. The following results generalize Theorem~2.3, Corollary~2.4, and a simplified version of Theorem~2.7 from \cite{prod-banach} to products over arbitrarily ordered index sets.
\begin{prop}\label{ord-prod-conv}
Suppose we have a totally ordered set $A$, a unital Banach algebra $\mathcal{X}$, and a function $x \maps A \to \mathcal{X}$. If the sum $\sum_{p \in A} \|x_p - 1\|$ converges, the product $\prod_{p \in A} x_p$ converges as well.
\end{prop}
\begin{prop}\label{ord-prod-inv}
If the sum in the proposition above converges, and each factor $x_p$ is invertible, the product is invertible.
\end{prop}
Our proof will give a handy bound for free.
\begin{prop}\label{ord-prod-bound}
If the sum in Proposition~\ref{ord-prod-conv} converges, the product $\prod_{p \in A} \|x_p\|$ converges as well, and
\[ \left\| \prod_{p \in A} x_p \right\| \quad\le\quad \prod_{p \in A} \|x_p\| \quad\le\quad \exp\left(\sum_{p \in A} \|x_p - 1\|\right). \]
\end{prop}
We can also bound the product's distance from $1$.
\begin{prop}\label{ord-prod-near-id}
If the sum in Proposition~\ref{ord-prod-conv} converges,
\[ \left\| \prod_{p \in A} x_p - 1 \right\| \le \left(\sum_{p \in A} \|x_p - 1\|\right) \left(\prod_{p \in A} \max\big\{\|x_p\|, 1\big\}\right). \]
\end{prop}

Let's start with a less ambitious result.
\begin{prop}\label{ord-prod-abs-bound}
If the sum in Proposition~\ref{ord-prod-conv} converges, then for any $C > \exp\left(\sum_{p \in A} \|x_p - 1\|\right)$, there's a finite subset $S$ of $A$ with the property that
\[ \prod_{p \in R} \|x_p\| < C \]
for all finite subsets $R \subset A$ containing $S$.
\end{prop}
\begin{proof}
Assume $\sum_{p \in A} \|x_p - 1\|$ converges, and consider any positive constant $C$ with $\log C > \sum_{p \in A} \|x_p - 1\|$. By the definition of convergence, we can find a finite subset $S$ of $A$ with the property that
\[ \sum_{p \in R} \|x_p - 1\| < \log C \]
for all finite subsets $R \subset A$ containing $S$. Since
\[ \log \|x_p\| \le \bigl| \|x_p\| - 1 \bigr|
\qquad \text{and} \qquad
\bigl| \|x_p\| - 1 \bigr| \le \|x_p - 1\| \]
for all $p \in A$, we know
\[ \sum_{p \in R} \log \|x_p\| < \log C. \]
for all finite subsets $R \subset A$ containing $S$. Exponentiating both sides gives the desired result.
\end{proof}
\begin{proof}[Proof of Proposition~\ref{ord-prod-conv}]
By definition, $\mathcal{X}$ is complete, so we can prove that the product converges by showing that the product net is Cauchy. To that end, given any $\epsilon > 0$, we want to find a finite subset $S$ of $A$ with the property that
\[ \left\| \prod_{p \in R} x_p - \prod_{p \in S} x_p \right\| < \epsilon \]
for all finite subsets $R \subset A$ containing $S$.

Assume the sum $\sum_{p \in A} \|x_p - 1\|$ converges, and pick a constant $C > \exp\left(\sum_{p \in A} \|x_p - 1\|\right)$. Every convergent net is Cauchy~\cite[Proposition~3.2]{gen-top}, so we can find a finite subset $S'$ of $A$ with the property that
\[ \left| \sum_{p \in R} \|x_p - 1\| - \sum_{p \in S'} \|x_p - 1\| \right| < \epsilon/C \]
for all finite subsets $R \subset A$ containing $S'$. This inequality simplifies to
\[ \sum_{p \in R \smallsetminus S'} \|x_p - 1\| < \epsilon/C. \]
By Proposition~\ref{ord-prod-abs-bound}, we can also find a finite subset $S''$ of $A$ with the property that
\[ \prod_{p \in R} \|x_p\| < C \]
for all finite subsets $R \subset A$ containing $S''$. Defining $S$ as $S' \cup S''$, and observing that $R \smallsetminus S'$ contains $R \smallsetminus S$, we see that
\[ \sum_{p \in R \smallsetminus S} \|x_p - 1\| < \epsilon/C
\qquad \text{and} \qquad
\prod_{p \in R} \|x_p\| < C \]
for all finite subsets $R \subset A$ containing $S$.

Put the elements of $R \smallsetminus S$ in some order $r_1, \ldots, r_n$. Let
\begin{align*}
R_0 & = S, \\
R_1 & = R_0 \cup \{r_1\}, \\
R_2 & = R_1 \cup \{r_2\},
\end{align*}
and so on. Let
\[ \Delta_{k + 1} = \left\| \prod_{p \in R_{k + 1}} x_p - \prod_{p \in R_k} x_p \right\|. \]
Notice that
\begin{align*}
\prod_{p \in R_{k + 1}} x_p - \prod_{p \in R_k} x_p =
\left( \prod_{\substack{p \in R_k \\ p < r_{k + 1}}} x_p \right)
(x_{r_{k + 1}} - 1)
\left( \prod_{\substack{p \in R_k \\ r_{k + 1} < p}} x_p \right),
\end{align*}
yielding the bound
\begin{align*}
\Delta_{k + 1} & \le \|x_{r_{k + 1}} - 1\| \prod_{p \in R_k} \|x_p\| \\
& \le \|x_{r_{k + 1}} - 1\|C.
\end{align*}
Finally, observe that
\begin{align*}
\left\| \prod_{p \in R} x_p - \prod_{p \in S} x_p \right\| & \le \Delta_1 + \ldots + \Delta_n \\
& \le \|x_{r_1} - 1\|C + \ldots + \|x_{r_n} - 1\|C \\
& = C \sum_{p \in R \smallsetminus S} \|x_p - 1\| \\
& < \epsilon.
\end{align*}
Since $R$ could have been any finite subset of $A$ containing $S$, and the method we used to find $S$ works for any $\epsilon > 0$, we've proven that the product net is Cauchy.
\end{proof}
\begin{proof}[Proof of Proposition~\ref{ord-prod-inv}]
Given an ordered set $I$, let $I^{\op}$ be the same set in the opposite order. Assume $\sum_{p \in A} \|x_p - 1\|$ converges. Equivalently, because addition is commutative, $\sum_{p \in A^{\op}} \|x_p - 1\|$  converges. This is only possible if $\|x_p - 1\| \le 6/7$ for all but finitely many $p \in A^{\op}$. It follows, by the calculation in the proof of \cite[Lemma~2.6]{prod-banach}, that $\|x_p^{-1}\| \le 7$ for all but finitely many $p \in A^{\op}$. That means $\|x_p^{-1}\|$ has a maximum over all $p \in A^{\op}$, which I'll call $M$. Just as in the proof of \cite[Theorem~2.7]{prod-banach}, observe that
\begin{align*}
\|x_p^{-1} - 1\| & = \|x_p^{-1}(1 - x_p)\| \\
& \le M\|x_p - 1\|
\end{align*}
for all $p \in A^{\op}$. Then \cite[Exercise 7.40.c]{haf} tells us that $\sum_{p \in A^{\op}} \|x_p^{-1} - 1\|$ converges, so $\prod_{p \in A^{\op}} x_p^{-1}$ converges by Proposition~\ref{ord-prod-conv}.

Define $y'' = \prod_{p \in A^{\op}} x_p^{-1}$ and $y' = \prod_{p \in A} x_p$. Multiplication in a Banach algebra is continuous, so for any open neighborhood $\Omega$ of $y''y'$, we can find open neighborhoods $\Omega''$ of $y''$ and $\Omega'$ of $y'$ with $\Omega'' \Omega' \subset \Omega$. By convergence, we can find finite subsets $S''$ and $S'$ of $A$ such that
\[ \prod_{p \in R^{\op}} x_p^{-1} \in \Omega'' \]
for all finite subsets $R \subset A$ containing $S''$, and
\[ \prod_{p \in R} x_p \in \Omega' \]
for all finite subsets $R \subset A$ containing $S'$. Now, defining $S$ as $S'' \cup S'$, we can observe that
\[ \left( \prod_{p \in S^{\op}} x_p^{-1} \right) \left( \prod_{p \in S} x_p \right) \in \Omega. \]
But the product in the expression above is clearly equal to 1! We've shown that every open neighborhood of $y''y'$ contains 1, which means $y''y'$ is equal to 1. The same argument can be used to show that $y'y''$ is 1. Therefore, $\prod_{p \in A} x_p$ is invertible, with inverse $\prod_{p \in A^{\op}} x_p^{-1}$.
\end{proof}
\begin{proof}[Proof of Proposition~\ref{ord-prod-bound}]
Assume $\sum_{p \in A} \|x_p - 1\|$ converges. By Proposition~\ref{ord-prod-conv}, $\prod_{p \in A} x_p$ converges too. The convergence of $\prod_{p \in A} \|x_p\|$ follows immediately from the fact that the norm is continuous. The first inequality is easy to establish.

Because $\sum_{p \in A} \|x_p - 1\|$ is a sum of non-negative numbers, its convergence implies that at most countably many of its terms are nonzero. We can therefore assume, without loss of generality, that $A$ is countable. Combining this fact with Proposition~\ref{ord-prod-abs-bound} and the discussion in Section~\ref{ord-prod-calculation}, it's not hard to show that $\prod_{p \in A} \|x_p\|$ is less than any number greater than $\exp\left(\sum_{p \in A} \|x_p - 1\|\right)$. That gives the second inequality.
\end{proof}
\begin{proof}[Proof of Proposition~\ref{ord-prod-near-id}]
We'll do a simpler version of the trick we did to show the product net is Cauchy. Assume $\sum_{p \in A} \|x_p - 1\|$ converges. Pick a finite subset of $A$, and label its elements from left to right as $1, \ldots, n$. Consider the telescoping sum
\newlength{\stair}
\setlength{\stair}{18pt}
\begin{align*}
x_1 \cdots x_n - 1 & = (x_1 - 1) x_2 \cdots x_n \\
& \hphantom{=}\hspace{\stair} + (x_2 - 1) x_3 \cdots x_n \\
& \hphantom{=}\hspace{2\stair}\hspace{-2pt} \ddots \\
& \hphantom{=}\hspace{3\stair} + (x_{n-1} - 1) x_n \\
& \hphantom{=}\hspace{4\stair} + (x_n - 1).
\end{align*}
For each term,
\[ \|(x_k - 1) x_{k+1} \cdots x_n\| \le \|x_k - 1\| \left(\prod_{p \in A} \max\big\{\|x_p\|, 1\big\}\right), \]
so overall,
\[ \|x_1 \cdots x_n - 1\| \le \left(\sum_{p \in A} \|x_p - 1\|\right) \left(\prod_{p \in A} \max\big\{\|x_p\|, 1\big\}\right). \]
We can bring the product $x_1 \cdots x_n$ arbitrarily close to $\prod_{p \in A} x_p$ by picking a large enough finite subset of $A$. Hence, the bound above implies the bound we want.
\end{proof}
\section{Linear algebra on a complex Euclidean plane}\label{euclidean}
In the world of complex vector spaces, the analogue of a Euclidean plane is a two-dimensional inner product space $E$ with a volume form $D \maps E^2 \to \C$ that gives each unit square a volume of unit norm. Each real plane in $E$ comes with two natural Euclidean plane structures, with opposite orientations. As a result, many basic facts about real Euclidean planes extend to the complex setting. I'll collect a few of them here. Really well-known facts will be stated without proof.

The sine metric $d_\angle$ on $\Proj E$ is defined as in Section~\ref{stable-lipschitz}. The area of a parallelogram can be computed from the lengths of its sides and the angle between them.
\begin{secprop}\label{area-angle}
\[ |D(u, v)| = \|u\| \|v\|\;d_\angle(u, v) \]
for all nonzero $u, v \in E$.
\end{secprop}
By comparing the areas of some well-chosen parallelograms, you can deduce the law of sines.
\begin{secprop}[The law of sines]
\[ \|u\|\;d_\angle(u,u+v) = \|v\|\;d_\angle(v,u+v) \]
for all nonzero $u, v \in E$ with $u + v \neq 0$.
\end{secprop}
From this we see that the extent to which a linear map $T$ can move lines is limited by the operator norm of $T - 1$.
\begin{secprop}\label{motion-bound}
For any linear map $T \maps E \to E$,
\[ d_\angle(u,Tu) \le \|T - 1\| \]
whenever $Tu \neq 0$.
\end{secprop}
\begin{proof}
Because distances in $\mathbf{P}E$ are never greater than one, it follows from the law of sines that
\[ d_\angle(u,u+v) \le \frac{\|v\|}{\|u\|} \]
In particular,
\[ d_\angle(u,Tu) \le \frac{\|(T-1)u\|}{\|u\|}. \]
\end{proof}
The extent to which a volume-preserving map can expand angles is limited by the operator norm of its inverse.
\begin{secprop}\label{expansion-bound}
For any $T \in \on{SL}(E)$,
\[ d_\angle(Tu,Tv) \le \|T^{-1}\|^2\;d_\angle(u,v) \]
for all nonzero $u, v \in E$.
\end{secprop}
\begin{proof}
\begin{align*}
d_\angle(Tu,Tv) & = \frac{|D(Tu,Tv)|}{\|Tu\|\|Tv\|} \\
& = \frac{\|u\|}{\|Tu\|}\;\frac{\|v\|}{\|Tv\|}\;\frac{|D(u,v)|}{\|u\|\|v\|} \\
& = \frac{\|u\|}{\|Tu\|}\;\frac{\|v\|}{\|Tv\|}\;d_\angle(u,v) \\
& \le \|T^{-1}\|^2\;d_\angle(u,v).
\end{align*}
\end{proof}
\section{The stable lines of a Markov cocycle}\label{markov-stable-lines}
The bi-infinite sequences over a finite alphabet $\mathcal{A}$ form a topological space $\mathcal{A}^\Z$. The {\em shift map}, which acts on a sequence by moving each letter one step earlier, makes $\mathcal{A}^\Z$ a dynamical system. I'll call a dynamical cocycle over the shift map a {\em Markov cocycle} if its action at a given sequence depends only on the letter at index zero.

Pick a subset $X \subset \mathcal{A}^\Z$ which is invariant under forward and backward shifts. Consider a Markov cocycle $A \maps X \to \on{SL}(E)$, where $E$ is a two-dimensional complex inner product space. If $A$ is uniformly hyperbolic, we'll collect its forward- and backward-stable lines into a pair of {\em stable distributions} $E^\pm \maps X \to \Proj E$, following the notation of Section~\ref{local-unif-hyp}.

It can be useful to know how sensitively the stable lines $E^\pm_x$ depend on the point $x \in X$. In \cite{cmv-matrices}, Damanik, Fillman, Lukic, and Yessen show in passing that the sensitivity is reasonably low, as measured with respect to natural metrics on $\Proj E$ and $\mathcal{A}^\Z$. (These metrics are defined in Sections \ref{stable-lipschitz} and \ref{division-dynamics}.)
\begin{secprop}\label{markov-stable-lipschitz}
If the Markov cocycle $A$ is uniformly hyperbolic, with bounding exponent $K$, its stable distributions are Lipschitz continuous with respect to the sine metric on $\Proj E$ and the division metric on $X$ with steepness $e^{2K}$.
\end{secprop}
Here's a streamlined, stand-alone version of Damanik, Fillman, Lukic, and Yessen's argument. First, recall that a non-unitary transformation $T \in \on{SL}(E)$ must contract some vectors and expand others. The most contracted vectors form a line, and the most expanded vectors form an orthogonal line, because the most contracted and expanded vectors are the eigenvectors of $T^\dagger T$. Lemma~2.1 from \cite{cmv-matrices} shows that vectors which are contracted a lot must lie near the most contracted line. Here's the part of the lemma we'll need.
\begin{secprop}\label{growth-angle-bound}
Let $\Lambda$ be the line most contracted by a given non-unitary transformation $T \in \on{SL}(E)$. For any $v \in E$ and $R > 0$, the condition $\|Tv\| \le R\|v\|$ implies that $d_\angle(v, \Lambda) \le R\|T\|^{-1}$.
\end{secprop}
\begin{proof}
Let's assume $v$ is a unit vector, since the general case reduces easily to this one. Let $\Lambda'$ be the line most expanded by $T$. Recall that $\Lambda$ and $\Lambda'$ are orthogonal, and observe that their images under $T$ are orthogonal as well. Pick unit vectors $u \in \Lambda$ and $u' \in \Lambda'$ whose real span contains $v$. Note that $T$ maps the real span of $u, u'$ to the real span of $Tu, Tu'$. We then have the following picture, with the real span of $u, u'$ shown on the left and the real span of $Tu, Tu'$ on the right.
\begin{center}
\begin{tikzpicture}[x=0.25in, y=0.25in, color=magenta]
\node[anchor=south west, inner sep=0in] {\includegraphics[width=4.5in]{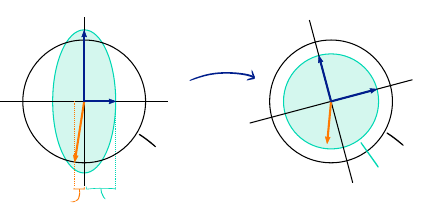}};

\node[black, anchor=south] at (3.5, 8) {$\Lambda$};
\node[black, anchor=west] at (7, 4.6) {$\Lambda'$};
\node[ietocean, anchor=east] at (3.5, 7.6) {$R \|T\| u$};
\node[ietocean, anchor=south west] at (4.5, 4.5) {$R \|T\|^{-1} u'$};
\node[ietpapaya, anchor=south east] at (3.2, 2.1) {$v$};
\node[ietpapaya, anchor=east] at (3, 0.4) {$d_\angle(v, \Lambda)$};
\node[ietlagoon, anchor=west] at (4.2, 0.4) {$R \|T\|^{-1}$};
\node[black, anchor=west] at (6.4, 2.6) {$1$};

\node[ietocean, anchor=south] at (9.2, 5.8) {$T$};

\node[black, anchor=south] at (12.7, 8) {$T\Lambda$};
\node[black, anchor=west] at (17, 5.6) {$T\Lambda'$};
\node[ietocean, anchor=east] at (13.2, 6.5) {$R \|T\| Tu$};
\node[ietocean, anchor=north west] at (15.5, 5.2) {$R \|T\|^{-1} Tu'$};
\node[ietpapaya, anchor=south east] at (13.6, 2.8) {$Tv$};
\node[ietlagoon, anchor=north west] at (15.5, 1.9) {$R$};
\node[black, anchor=west] at (16.7, 2.7) {$1$};
\end{tikzpicture}
\end{center}
Under $T^{-1}$, the circle of radius $R$ maps to an ellipse with long axis $\Lambda$, short axis $\Lambda'$, and short radius $R\|T\|^{-1}$. This gives the desired bound on $d_\angle(v, \Lambda)$, which you can see as the length of $v$ along $\Lambda'$.
\end{proof}
Because $A$ is uniformly hyperbolic, we can ensure that the transformations $A^{\pm n}_x$ are non-unitary for all $x \in X$ by making $n$ large enough. Say the numbers after $N \in \N$ are large enough. For each $x \in X$ and $n \in \N_{> N}$, let $E^\pm_{x,n}$ be the line most contracted by $A^{\pm n}_x$.
\begin{secprop}\label{stable-line-appx}
If $A$ is uniformly hyperbolic, with bounding exponent $K$, the most contracted lines $E^\pm_{x,n}$ are good approximations for the stable lines $E^\pm_x$, in the sense that
\[ d_\angle(E^\pm_x, E^\pm_{x,n}) \lesssim e^{-2Kn} \]
over all $x \in X$ and $n \in \N_{> N}$.
\end{secprop}
\begin{proof}
By the decay properties characterizing the forward and backward-stable lines, and the equivalent growth properties from Proposition~\ref{local-soft-turnaround}, we can find a constant $C > 0$ such that
\begin{align*}
\|A_x^{\pm n} v\| & \le Ce^{-Kn} \|v\| &
C^{-1} e^{Kn} \|w\| & \le \|A_x^{\pm n} w\|
\end{align*}
for all $x \in X$, $v \in E^\pm_x$, $w \in E^\mp_x$, and $n \in \N$. The first bound, together with Proposition~\ref{growth-angle-bound}, implies that
\[ d_\angle(v, E^\pm_{x,n}) \le Ce^{-Kn} \|A^{\pm n}_x\|^{-1} \]
for all $x \in X$, $v \in E^+_x$, and $n \in \N_{> N}$. The second bound implies that
\[ C^{-1} e^{Kn} \le \|A_x^{\pm n}\| \]
for all $x \in X$ and $n \in \N$. Combining, we see that
\[ d_\angle(E^\pm_x, E^\pm_{x,n}) \le C^2 e^{-2Kn} \]
for all $x \in X$ and $n \in \N_{> N}$.
\end{proof}
\begin{proof}[Proof of Proposition \ref{markov-stable-lipschitz}]
Assume $A$ is uniformly hyperbolic, with bounding exponent $K$. Give $X$ the division metric with steepness $e^{2K}$. Given $x, y \in X$, let $m_{xy}$ be the number of steps you have to take from position zero to reach a position where $x$ and $y$ differ. Then we can say $d(x, y) = e^{-2Km_{xy}}$.

Observe that $A^{\pm m_{xy}}_x$ and $A^{\pm m_{xy}}_y$ only depend on the letters of $x$ and $y$ up to $m_{xy}-1$ steps from position zero. Since $x$ and $y$ match until you get $m_{xy}$ steps from position zero, $A^{\pm m_{xy}}_x$ and $A^{\pm m_{xy}}_y$ are equal, and their most contracted lines $E^\pm_{x,m_{xy}}$ and $E^\pm_{y,m_{xy}}$ are equal as well. Hence, by the triangle inequality,
\[ d_\angle(E^\pm_x, E^\pm_y) \le d_\angle(E^\pm_x, E^\pm_{x,m_{xy}}) + d_\angle(E^\pm_y, E^\pm_{y,m_{xy}}). \]

We know from Proposition~\ref{stable-line-appx} that
\begin{align*}
d_\angle(E^\pm_x, E^\pm_{x,m_{xy}}) & \lesssim e^{-2Km_{xy}} &
d_\angle(E^\pm_y, E^\pm_{y,m_{xy}}) & \lesssim e^{-2Km_{xy}}
\end{align*}
over all $x, y \in X$ with $m_{xy} > N$. Plugging these bounds into the triangle inequality above, and recalling that $d(x, y) = e^{-2Km_{xy}}$, we learn that
\[ d_\angle(E^\pm_x, E^\pm_y) \lesssim d(x, y) \]
over all $x, y \in X$ with $d(x, y) < e^{-2KN}$. Hence, $E^\pm \maps X \to \Proj E$ are locally Lipschitz. Since $\Proj E$ has finite diameter, it follows that $E^\pm$ are globally Lipschitz.
\end{proof}
\section{Standard punctures for translation surfaces}\label{punk-shapes}
\subsection{Motivation}
To see where the standard puncture shapes come from, we need to talk about the complex geometry of translation surfaces. Every translation surface $\Sigma$ comes with a complex structure, which we get by identifying $\R^2$ with $\C$ in the usual way, and a complex-valued $1$-form $\omega$, which sends horizontal unit vectors to $1$ and vertical unit vectors to $i$. Observing that $\omega = dz$ for any local translation chart $z \maps \Sigma \to \C$, we see that $\omega$ is holomorphic. Conversely, a complex $1$-manifold equipped with a holomorphic $1$-form is canonically a translation surface. Where $\omega$ has a zero of order $n$, the translation structure has a conical singularity of angle $2(n+1)\pi$.

The complex point of view suggests a natural class of translation surfaces that are non-compact, but still well-behaved. Putting a meromorphic $1$-form on a compact Riemann surface defines a translation structure on the complement of the poles. The poles correspond to the ends of the translation surface, and the Riemann surface is its end compactification \cite[\S 1]{theory-of-ends}. The poles of a $1$-form have a limited variety of behaviors, so the ends of the translation surface have a limited variety of shapes, which we'll call the {\em standard punctures}.
\subsection{First-order punctures}
A first-order pole in $\omega$ makes a puncture shaped like a half-infinite cylinder. You can build one by rolling up a half-infinite rectangular strip and gluing its sides together:
\begin{center}
\begin{tikzpicture}
\begin{scope}[rotate=25]
\foliate{pwbeige}{30}{60} (-1.2, 0) rectangle (1.2, 5);
\foliate{pwpink}{30}{60} (-1.2, 0) rectangle (-0.6, 5);
\foliate{pwpink}{30}{60} (0.6, 0) rectangle (1.2, 5);
\freelabel{pwpink}{1}{(-0.9, 1)}{30}
\freelabel{pwpink}{1}{(0.9, 1)}{30}
\draw[white] (-1.3, 5) -- (1.3, 5);
\end{scope}
\begin{scope}[transform canvas={rotate=25}]
\fill[white, path fading=south] (-1.3, 5) rectangle (1.3, 4.8);
\end{scope}
\end{tikzpicture}
\end{center}
The translation structure of the cylinder you end up with is determined by two parameters: the width of the strip and its orientation in the plane. In most orientations, the vertical leaves spiral up or down the cylinder. When the strip is horizontal, the vertical leaves close up into circles.
\subsection{Higher-order punctures}
The puncture created by a higher-order pole in $\omega$ can be glued together out of planes with quadrants cut away, like this:
\begin{center}
\begin{tikzpicture}
\matrix[row sep=0.75cm, column sep=1cm]{
\hairdown{pwviolet}{$6$}{pwblue}{$5$} & \hairdown{pwgreen}{$4$}{pwgold}{$3$} & \hairdown{pworange}{$2$}{pwpink}{$1$} \\
\hairup{pwgreen}{$4$}{pwblue}{$5$} & \hairup{pworange}{$2$}{pwgold}{$3$} & \hairup{pwviolet}{$6$}{pwpink}{$1$} \\
};
\end{tikzpicture}
\end{center}
The notches at the centers of the pieces fit together into a polygonal hole that the rest of the surface can be connected to. For the pieces to fit together, the notches all have to be the same size, which can be adjusted to accommodate the part of the surface the puncture is supposed to hook up to.

Unlike first-order punctures, which come in a $\C^\times$-worth of shapes parameterized by the residues at their poles, higher-order punctures depend only on the orders of their poles.
\subsection{Counting ends}
As I remarked earlier, a puncture is an end of the surface it lives in. Intuitively, you can think of it as a point on the boundary at infinity of the surface. The notion of an end is purely topological, however, and you might wonder if there's a notion of boundary at infinity that takes the translation structure into account. Here's one proposal, which is suitable at least for translation surfaces with punctures.

Let's say two vertical rays on a translation surface are {\em translation-equivalent} if they're connected by a continuous family of vertical rays.\footnote{To be formal about it, define a vertical ray to be a local isometry of $[0, \infty)$ into a vertical leaf. We can then say a continuous family of vertical rays is a continuous map from $[0, 1] \times [0, \infty)$ into the surface which restricts to a vertical ray when the first argument is fixed.} A {\em vertical end} is an equivalence class of vertical rays. A generic first-order puncture, whose vertical leaves spiral up or down rather than closing up into circles, is a single vertical end. A higher-order puncture, on the other hand, comprises several vertical ends, one for each building block.

\bibliographystyle{utphys}
\bibliography{abelianization}
\end{document}